\documentclass{amsart}
\usepackage{amssymb,amsmath, amsthm,latexsym}
\usepackage{graphics}
\usepackage{psfrag}
\usepackage{amscd}
\usepackage{graphicx}
\newcommand{\cal}[1]{\mathcal{#1}}
\theoremstyle{plain}
\newtheorem{theo}{Theorem}
\newtheorem{cor}{Corollary} 
\newtheorem{lemma}{Lemma}[section]
\newtheorem{theorem}[lemma]{Theorem}
\newtheorem{proposition}[lemma]{Proposition}
\newtheorem{corollary}[lemma]{Corollary}
\theoremstyle{definition}
\newtheorem{definition}[lemma]{Definition}
\newtheorem{remark}[lemma]{Remark} 
\newtheorem{example}[lemma]{Example}

\parskip=\bigskipamount

\let\egthree=\phi 
\let\phi=\varphi
\let\varphi=\egthree

\newcounter{sebcomments}
\begin{document}
\title{Typical and atypical properties of periodic Teichm\"uller geodesics}
\author{Ursula Hamenst\"adt}
\thanks
{
Research
supported by ERC grant ``Moduli''}
\date{February 20, 2017} 

\begin{abstract} Consider a component ${\cal Q}$ of 
a stratum in the moduli space 
of area one abelian differentials 
on a surface of genus $g$. 
Call a property ${\cal P}$ for periodic
orbits of the Teichm\"uller flow on  
${\cal Q}$ \emph{typical} if
the growth rate of orbits with property ${\cal P}$ is maximal.
Typical are:
The logarithms of the eigenvalues of the symplectic matrix defined
by the orbit are arbitrarily close to the
Lyapunov exponents of ${\cal Q}$, and its 
trace field is a totally real splitting field 
of degree $g$ over $\mathbb{Q}$. 
If $g\geq 3$ 
then periodic orbits whose $SL(2,\mathbb{R})$-orbit
closure equals ${\cal Q}$ 
are typical. We also show that ${\cal Q}$ contains 
only finitely many
algebraically primitive Teichm\"uller curves, and only 
finitely many affine invariant submanifolds of rank $\ell\geq 2$.

\end{abstract}

\maketitle

\section{Introduction}

The mapping class group ${\rm Mod}(S)$ 
of a closed surface $S$ of genus $g\geq 2$ acts 
by precomposition of marking on the
Teichm\"uller space ${\cal T}(S)$ 
of marked complex structures on $S$.
The action is properly discontinuous, with 
quotient the moduli space 
${\cal M}_g$ of complex structures on $S$.

The goal of this paper is to describe properties of  
this action which
are invariant under conjugation and hold true
for conjugacy classes of 
mapping classes which are typical in the following sense.

The moduli space of area one abelian differentials on 
$S$ decomposes into \emph{strata} of differentials
with zeros of given multiplicities. 
There is a natural $SL(2,\mathbb{R})$-action on 
any connected component ${\cal Q}$ of this moduli space.
The action of the diagonal subgroup is the
\emph{Teichm\"uller flow} $\Phi^t$.

Let $\Gamma$ be the set of all 
periodic orbits for $\Phi^t$ on ${\cal Q}$. 
The length of a periodic orbit 
$\gamma\in \Gamma$ is denoted by $\ell(\gamma)$. 
Let $h>0$ be the entropy of the unique
$\Phi^t$-invariant Borel probability 
measure on ${\cal Q}$ in the
Lebesgue measure class \cite{M82,V86}.
As an application of \cite{EMR12} 
(see also \cite{EM11,H11}) 
we showed in \cite{H13} that
\[\sharp\{\gamma\in \Gamma\mid \ell(\gamma)\leq R\}\sim \frac{e^{hR}}{hR}.\]

Call a subset ${\cal A}$ of $\Gamma$ \emph{typical} 
if 
\[\sharp \{\gamma\in {\cal A}\mid
\ell(\gamma)\leq R\}\sim \frac{e^{hR}}{hR}.\]
Thus a subset of $\Gamma$ is typical if its growth rate
is maximal. The intersection of two typical subsets is typical.

A periodic orbit 
$\gamma\in \Gamma$ for $\Phi^t$ determines the conjugacy class of 
a pseudo-Anosov mapping class.
Each mapping class acts on $H_1(S,\mathbb{Z})$.
This defines a natural surjective \cite{FM12} homomorphism
\[\Psi:{\rm Mod}(S)\to Sp(2g,\mathbb{Z}).\]
Thus a periodic orbit $\gamma$ of $\Phi^t$ determines the 
conjugacy class $[A(\gamma)]$ 
of a matrix $A(\gamma)\in Sp(2g,\mathbb{Z})$.

Let 
\[1=\kappa_1>\kappa_2>\dots >\kappa_g>0\] 
be the positive Lyapunov exponents
of the Kontsevich Zorich cocycle with respect to 
the normalized Lebesgue
measure on ${\cal Q}$. The fact that there are no multiplicities
in the Lyapunov spectrum was shown in \cite{AV07}.
For $\gamma\in \Gamma$ 
let $\hat \alpha_i(\gamma)$ be the logarithm of the 
absolute value of the $i$-th 
eigenvalue of the matrix $A(\gamma)$ ordered in decreasing order
and write $\alpha_i(\gamma)=\hat \alpha_i(\gamma)/\ell(\gamma)$.
As $A(\gamma)$ is symplectic, with real
leading eigenvalue $e^{\ell(\gamma)}$,  we have 
\[1=\alpha_1(\gamma)\geq \dots \geq \alpha_g(\gamma)\geq 0
\geq -\alpha_g(\gamma)\geq \dots \geq -\alpha_1(\gamma)=-1.\]
As eigenvalues of matrices 
are invariant under conjugation, 
for $-g\leq i\leq g$ we obtain in this way a function
$\alpha_i:\Gamma\to [-1,1]$.

The characteristic polynomial of a symplectic
matrix $A\in Sp(2g,\mathbb{Z})$ 
is a reciprocal polynomial of degree $2g$ with integral
coefficients. Its roots define a number field $K$ 
of degree at most
$2g$ over $\mathbb{Q}$ which is a quadratic extension of 
the so-called \emph{trace field} of $A$.
The Galois group of $K$  
is isomorphic to a subgroup of the semi-direct product
$(\mathbb{Z}/2\mathbb{Z})^g\rtimes \Sigma_g$ where $\Sigma_g$ is the 
symmetric group in $g$ variables (see \cite{VV02} for 
details). The field $K$ and the Galois group 
only depend on the conjugacy class of $A$.

For $\gamma\in \Gamma$ let 
$G(\gamma)$ be the Galois group 
of the number field defined by the 
conjugacy class $[A(\gamma)]$. We show

\begin{theo}\label{theolyapunov}
\begin{enumerate}
\item For $\epsilon >0$ the set  
$\{\gamma\in \Gamma\mid 
\vert \alpha_i(\gamma)-\kappa_i\vert <\epsilon\}$ $(1\leq i\leq g)$
is typical.
\item 
The set of all $\gamma\in \Gamma$ such that
the trace field of $[A(\gamma)]$ is totally real, of degree
$g$ over $\mathbb{Q}$, and  
$G(\gamma)=(\mathbb{Z}/2\mathbb{Z})^g\rtimes\Sigma_g$
is typical.
\end{enumerate}
\end{theo}


The proof of Theorem \ref{theolyapunov} uses a 
result on the Zariski closure of the image under the 
map $\Psi$ of pseudo-Anosov mapping classes obtained from 
the first return map of the Teichm\"uller flow
on ${\cal Q}$ to a small contractible flow box in 
${\cal Q}$. Although our viewpoint is a bit different, this
discussion can be translated into properties of the 
\emph{Rauzy-Veech group} of ${\cal Q}$ and yields the following
result which 
was conjectured by Zorich \cite{Z99}.

\begin{cor}\label{zorich}  
The Rauzy-Veech group of any component of a stratum is a 
Zariski dense subgroup of
$Sp(2g,\mathbb{Z})$.
\end{cor}

For hyperelliptic strata, 
Avila, Matheus and Yoccoz \cite{AMY16}
showed that the Rauzy-Veech group is a subgroup of 
$Sp(2g,\mathbb{Z})$ of finite index. In Proposition \ref{surjectivity}
we observe that for strata with at least one simple zero, the
Rauzy Veech group coincides with $Sp(2g,\mathbb{Z})$.
Theorem \ref{monodromy} is a more precise version of Corollary \ref{zorich} 
which is valid for \emph{affine invariant manifolds}
as well.

By the groundbreaking work of Eskin, Mirzakhani and Mohammadi
\cite{EMM15}, such affine invariant manifolds
are precisely the closures of orbits 
for the $SL(2,\mathbb{R})$-action on ${\cal Q}$. 
Examples of non-trivial orbit closures are arithmetic
\emph{Teichm\"uller curves}. They arise from branched covers
of the torus, and they 
are dense in any stratum of abelian differentials.
Other examples of orbit closures different from entire
components of strata can be constructed using more general
branched coverings.



The \emph{rank} of an affine invariant manifold ${\cal M}$ is defined by
\[{\rm rk}({\cal M})=\frac{1}{2}{\rm dim}(pT{\cal M})\]
where $p$ is the projection of period coordinates into absolute cohomology
\cite{W15}. 
Teichm\"uller curves are affine invariant manifolds of rank one, and the
rank of a component of a stratum equals $g$.

We establish 
a finiteness result for 
affine invariant submanifolds of rank at least two 
which is independently due to Eskin, Filip and Wright \cite{EFW17}.


\begin{theo}\label{finite}
Let $g\geq 2$ and let 
${\cal Q}$ be a component of a stratum in the moduli space of 
abelian differentials. 
For every $2\leq \ell\leq g$, there are only 
finitely many proper affine invariant submanifolds 
in ${\cal Q}$ of rank $\ell$.
\end{theo}





As a corollary, we obtain

\begin{cor}\label{cyclic}
Let ${\cal Q}$ be any component of 
a stratum in genus $g\geq 3$. Then the
set of all $\gamma\in \Gamma$ 
whose $SL(2,\mathbb{R})$-orbit closure equals ${\cal Q}$ 
is typical.
\end{cor}

For $g=2$, Corollary \ref{cyclic} is false
in a very strong sense. 
Namely, McMullen \cite{McM03a} showed
that in this case, the orbit closure of any periodic orbit
is an affine invariant manifold
of rank one. If the trace field $K$ 
of the orbit is quadratic, then 
$K$ defines a Hilbert modular surface 
in the moduli space of principally polarized abelian 
varieties which contains the image of the 
orbit closure under the Torelli map.
Such a Hilbert modular surface is 
a quotient of ${\bf H}^2\times {\bf H}^2$ by the
lattice $PSL(2,{\cal O}_K)$ where ${\cal O}_K$ is the ring of 
algebraic integers in $K$. This insight is the starting
point of a complete 
classification
of orbit closures in genus $2$ \cite{McM03b}. 


In higher genus, Apisa \cite{A15} classified all orbit closures
in hyperelliptic components of strata. 
For other components of strata,  
a classification of orbit closures is not available.
However, there is 
substantial recent progress towards a geometric understanding
of orbit closures. In particular, Mirzakhani and Wright 
\cite{MW16} showed
that all affine invariant manifolds of maximal rank either are
components of strata or are contained in the hyperelliptic locus.
We refer to the work \cite{LNW15} of Lanneau, Nguyen and Wright for 
an excellent recent overview of what is known and for 
a structural result for rank one affine
invariant manifolds. 



Teichm\"uller curves are affine invariant manifolds of dimension
$2$. To each such 
Teichm\"uller curve, there is associated a \emph{trace field}
which is an algebraic number field of degree at most $g$ over
$\mathbb{Q}$. This trace field coincides with the trace field of 
every periodic orbit contained in the curve \cite{KS00}. 
The Teichm\"uller curve is called \emph{algebraically primitive}
if the degree of its trace field equals $g$. 

The stratum ${\cal H}(2)$ of abelian differentials with a single
zero on a surface of genus 2 contains infinitely many
algebraically primitive Teichm\"uller curves \cite{McM03b}. 
Recently, Bainbridge, Habegger and M\"oller \cite{BHM14} 
showed finiteness
of algebraically primitive Teichm\"uller curves in any stratum
in genus $3$.
Finiteness of algebraically primitive
Teichm\"uller curves in strata
of differentials with a single zero for surfaces of 
prime genus $g\geq 3$ was established in  
\cite{MW15}. Our final result generalizes this to every
stratum in every genus $g\geq 3$, with a different proof.
A stronger finiteness result is contained in \cite{EFW17}.

\begin{theo}\label{finiteteich}
Any component ${\cal Q}$ of a stratum in genus $g\geq 3$ 
contains only finitely many algebraically 
primitive Teichm\"uller curves.
\end{theo}



\noindent
{\bf Plan of the paper:} In Section \ref{lyapunov} we establish 
the first part of Theorem 1 as a fairly easy consequence
of the results in \cite{H13}. 

Section \ref{zariski} introduces the idea of 
local monodromy groups and their Zarisky closures and uses it
to show Corollary \ref{zorich}. 
In Section \ref{galois}, this result together with 
group sieving and the first part of 
Theorem \ref{theolyapunov} leads
to the second part of Theorem \ref{theolyapunov} and to 
Corollary \ref{zorich}. 

Section \ref{localaffine} contains some properties of the
absolute period foliation of an affine invariant
manifold.
In Section \ref{local} we look at the local 
monodromy group of an affine invariant manifold and 
show that it is Zariski dense in the symplectic group of 
rank corresponding to the rank of the manifold. We then compare
in Section \ref{connections} 
the Chern connection on the Hodge bundle to the Gauss Manin
connection which  leads to the proof of the first part of Theorem \ref{finite}
in Secction \ref{uniqueness}. Section \ref{nested} uses information on the absolute period
foliation to complete the proof of Theorem \ref{finite}.
The proofs of Theorem \ref{cyclic} and 
Theorem \ref{finiteteich} are contained in Section 8.

\noindent
{\bf Acknowledgement:} During the various stages of this work, 
I obtained generous help from many collegues. I am particularly
grateful to Alex Wright for
pointing out a mistake in an earlier version of this work.
Both Matt Bainbridge and Alex Eskin notified me about parts
in an earlier version of the paper which needed clarification. 
Yves Benoist provided the proof of Proposition \ref{latticeordense}, 
and discussions with Curtis McMullen inspired me to the differential
geometric approach in Section \ref{nested}.
Part of this article is based on work which was supported by the National
Science Foundation under Grant No. DMS-1440140 while the author
was in residence at the MSRI in Berkeley, California, 
in spring 2015.

\section{Lyapunov exponents}\label{lyapunov}

In this section we consider any component
${\cal Q}$ of a stratum of area one 
abelian
differentials. The results in this section are equally
valid for any affine invariant submanifold for the 
$SL(2,\mathbb{R})$-action \cite{EMM15}, in particular
they hold true for all components of strata of quadratic 
differentials.  

The \emph{Teichm\"uller flow} $\Phi^t$ acts on ${\cal Q}$ preserving
a Borel probability measure $\lambda$ in the Lebesgue
measure class, the so-called \emph{Masur Veech measure}.
Let $h>0$ be the entropy of $\Phi^t$ 
with respect to the measure $\lambda$.

Denote by 
${\cal H}={\cal Q}\times H^1(S,\mathbb{R})\to {\cal Q}$ the
trivial vector bundle over ${\cal Q}$ with fibre
$H^1(S,\mathbb{R})$. 
The \emph{Kontsevich Zorich cocycle} over ${\cal Q}$
is an extension of the Teichm\"uller flow $\Phi^t$ on ${\cal Q}$ 
to a flow on ${\cal H}$ which
can roughly be described as follows. Consider a component
$\tilde {\cal Q}$ of the preimage of ${\cal Q}$ in the
Teichm\"uller space of marked area one abelian differentials.
Then ${\cal Q}$ is the quotient of $\tilde {\cal Q}$
under the action of the stabilizer ${\rm Stab}(\tilde {\cal Q})$ 
of $\tilde {\cal Q}$
in the mapping class group ${\rm Mod}(S)$. 
For $q\in \tilde {\cal Q}$ and
large $T$, the differential
$\Phi^Tq$ can be brought back to a fixed fundamental
domain for the action of ${\rm Stab}(\tilde {\cal Q})$ 
by an element $\phi\in {\rm Stab}(\tilde {\cal Q})$.
The action of $\phi$ on the first cohomology group
$H^1(S,\mathbb{R})$ is essentially the Kontsevich Zorich 
cocycle. 

The Kontsevich Zorich cocycle can also be described 
using the \emph{Gauss Manin connection}
on ${\cal H}$. The Gauss Manin connection is a flat 
connection on ${\cal H}$ 
defined by local trivializations which identify
nearby integral cohomology classes. 
Parallel transport for the Gauss Manin connection
defines a lift of the Teichm\"uller flow $\Phi^t$ to a flow
$\Theta^t$ on ${\cal H}$ which preserves the symplectic 
structure on ${\cal H}$ defined by the algebraic 
intersection form on $H^1(S,\mathbb{R})$.

There are some technical difficulties 
due to nontrivial 
point stabilizers for the action of the mapping class
group on Teichm\"uller space. To avoid dealing with this issue
(although this can be done with some amount of care)
we define the \emph{good subset} ${\cal Q}_{\rm good}$ 
of ${\cal Q}$ to be the set of all points $q\in {\cal Q}$ with the
following property. Let $\tilde {\cal Q}$ be 
a component of the preimage of ${\cal Q}$ in 
the Teichm\"uller space of marked abelian differentials
and let $\tilde q\in \tilde {\cal Q}$ be a lift of $q$;
then an element of ${\rm Mod}(S)$ which fixes $\tilde q$
acts as the identity on $\tilde {\cal Q}$ 
(compare \cite{H13} for more information on this technical
condition).
The good subset is open, dense and $\Phi^t$-invariant
\cite{H13}.

The Kontsevich Zorich cocycle is bounded, with values in 
the symplectic group $Sp(2g,\mathbb{R})$, and therefore its
\emph{Lyapunov exponents} for the invariant measure $\lambda$ 
are defined. These exponents measure the asymptotic 
growth rate of vectors along orbits of $\Phi^t$ which 
are typical for $\lambda$.
Since the Gauss Manin connection preserves the symplectic
structure on ${\cal H}$, these exponents 
are
invariant under multiplication with $-1$. 
Let 
$1=\kappa_1> \dots >  \kappa_g> 0$ be the 
largest $g$ Lyapunov exponents of 
the Teichm\"uller flow on ${\cal Q}$. 
That these exponents are all positive and pairwise distinct
was shown in \cite{AV07}.
For more general \emph{affine invariant manifolds}, 
the analogue statement need not hold true.  
We refer to \cite{A13} for a discussion
and examples.

Let 
\[\Gamma\subset {\cal Q}\] 
be the countable collection of all 
periodic orbits for $\Phi^t$ contained in ${\cal Q}$.
Denote by $\ell(\gamma)$ 
the period of $\gamma\in \Gamma$.
The orbit $\gamma\in \Gamma$ 
determines a conjugacy class in ${\rm Mod}(S)$
of pseudo-Anosov elements. Let $\phi\in {\rm Mod}(S)$ be
an element in this conjugacy class; 
then $A(\gamma)=\Psi(\phi)\in {\rm Sp}(2g,\mathbb{Z})$ is determined
by $\gamma$ up to conjugation. 
Furthermore, the
largest absolute value of an eigenvalue of $A(\gamma)$
equals $e^{\ell(\gamma)}$.
More precisely, 
the matrix $A(\gamma)$ is Perron Frobenius, with 
leading eigenvalue $e^{\ell(\gamma)}$, and  the eigenspace for
the eigenvalue $e^{\ell(\gamma)}$ is spanned by the 
real cohomology class defined by intersection with the
attracting measured geodesic lamination of $\phi$.

If we define 
$1= \alpha_1(\gamma)> \dots\geq \alpha_g(\gamma)\geq 0$
to be the quotients by $\ell(\gamma)$ of the logarithms 
of the $g$ largest absolute values of the eigenvalues of the matrix
$A(\gamma)$, ordered in decreasing order and counted with multiplicities, 
then the numbers $\alpha_i(\gamma)$ 
only depend on $\gamma$ but not on any choices made. 

Let $\epsilon >0$. For $\gamma\in \Gamma$ 
define
$\chi_\epsilon(\gamma)=1$ 
if $\vert \alpha_i(\gamma)-\kappa_i\vert <\epsilon$ for every
$i\in \{1,\dots,g\}$ 
and 
define $\chi_\epsilon(\gamma)=0$ otherwise.

For $R_1<R_2$ let $\Gamma(R_1,R_2)\subset \Gamma$ be the 
set of all periodic orbits of prime period contained
in the interval $(R_1,R_2)$.
For an open or closed subset $V$ of ${\cal Q}$ 
denote by $\chi(V)$ the characteristic function of $V$ and 
define 
\begin{align}H(V,R_1,R_2)& =
\sum_{\gamma\in \Gamma(R_1,R_2)}\int_\gamma\chi(V)
\quad \text{ and }\notag\\
H_\epsilon(V,R_1,R_2)&=
\sum_{\gamma\in \Gamma(R_1,R_2)}\int_\gamma\chi(V)\chi_\epsilon(\gamma).
\notag
\end{align}
Clearly we have 
\[H_\delta(V,R_1,R_2)\leq H_\epsilon(V,R_1,R_2)\leq H(V,R_1,R_2)\]
for all $\epsilon >\delta >0$.

Call a point $q\in {\cal Q}$ \emph{birecurrent} if 
it is contained in its own $\alpha$- and 
$\omega$ limit set. By the Poincar\'e recurrence theorem,
the set of birecurrent points
in ${\cal Q}$ has full Lebesgue measure.
In \cite{H13} (Corollary 4.8 and Proposition 5.4) we showed

\begin{proposition}\label{count}
For every good birecurrent point $q\in {\cal Q}_{\rm good}$, for every
neighborhood $U$ of $q$ in ${\cal Q}$ and for every
$\delta >0$ there is an open neighborhood $V\subset U$ 
of $q$ in ${\cal Q }$ 
and a number $t_0>0$ such that
\[H(V,R-t_0,R+t_0)e^{-hR}\in 
(2t_0\lambda(V)(1-\delta),2t_0\lambda(V)(1+\delta))\]
for all sufficiently large $R>0$.
\end{proposition}

The proof of Proposition \ref{count} is based on a more technical 
result which will be used several times in the sequel. 
Lemma \ref{technical} below 
combines Lemma 4.7 and Proposition 5.4 of \cite{H13}.
For its formulation, we say that a closed curve $\eta$ in ${\cal Q}_{\rm good}$
defines the conjugacy class of a pseudo-Anosov mapping class $\phi\in 
{\rm Mod}(S)$ if the following holds true.  
Let $\tilde \eta$ be a lift of $\eta$
to an arc in the Teichm\"uller space of abelian differentials, parametrized
one some interval $[0,a]\subset \mathbb{R}$; then 
$\tilde \eta(a)=\psi(\tilde \eta(0))$ for a mapping class 
$\psi$ which is conjugate to $\phi$. This definition does not depend on any
choices made.

\begin{lemma}\label{technical}
Let $q\in {\cal Q}_{\rm good}$ be a good birecurrent point and let $\delta >0$. 
Then every
neighborhood $U$ of $q$ contains an open
and contractible neighborhood $V\subset {\cal Q}_{\rm good}$ of $q$,
such that there exists a nested sequence of neighborhoods 
$Z_0\subset Z_1\subset Z_2\subset V$ of $q$, 
with $Z_i$ closed and $Z_i$ contained in the interior of $Z_{i+1}$, 
and there is a number $R_0>0$ with the following 
properties.

Let $z\in Z_0$ and let $R>R_0$ be such that $\Phi^Rz\in Z_0$. 
Let $\hat E$ be the connected component
containing $z$ of the intersection $\Phi^tV\cap V$.
\begin{enumerate}
\item[$a)$]
The length of the connected subsegment of the orbit  
$\cup_{t\in \mathbb{R}}\Phi^tz\cap Z_0$ containing $\Phi^Rz$ 
equals $2t_0$.
\item[$b)$]
$\lambda(Z_0)>(1-\delta)\lambda(V)$, and 
the Lebesgue measure of the intersection 
$\Phi^RZ_1\cap Z_2\cap \hat E$ is contained in the interval
\begin{equation}\label{measureestimate} 
[e^{-hR}\lambda(V)(1-\delta),
e^{-hR}\lambda(V)(1+\delta)].\notag\end{equation}
\item[$c)$]
Connect $\Phi^Rz$ to $z$ by an arc in $V$ and 
let $\eta$ be the concatenation of the
orbit segment $\cup_{0\leq t\leq R}\Phi^tz$ with this arc.
We call $\eta$ a \emph{characteristic curve} of the
orbit segment $\cup_{t\in [0,R]} \Phi^tz$.  
There is a unique periodic orbit $\gamma$ 
for $\Phi^t$ of length at most $R+\delta$ 
which intersects 
$\Phi^RZ_2\cap Z_2\cap \hat E$. 
The curve $\eta$ and the orbit $\gamma$ define the same 
conjugacy class in ${\rm Mod}(S)$.
\end{enumerate}
\end{lemma}

Note that in the above statement, we slightly adjusted the 
choice of the sets $Z_i$ compared to the 
terminology in \cite{H13} 
for clarity of exposition. 

We use Lemma \ref{technical} to show

\begin{proposition}\label{lowerestimate}
For every
birecurrent point $q\in {\cal Q}_{\rm good}$, for every
neighborhood $U$ of $q$ in ${\cal Q}$ 
and for every $\delta >0$ there is 
an open neighborhood $V\subset U$ of $q$ in ${\cal Q}$
and a number $t_0>0$ 
with the properties stated in Proposition \ref{count} 
such that for every $\epsilon>0$ we have
\[\lim\inf_{R\to\infty}H_\epsilon(V,R-t_0,R+t_0)e^{-hR}
\geq 2t_0\lambda(V)(1-\delta).\]
\end{proposition}
\begin{proof} 
Let $\Vert\,\Vert$ be the \emph{Hodge norm} on 
the bundle ${\cal H}\to {\cal Q}$ 
(see \cite{ABEM12} for definitions and the most
important properties). Denote as before by
$\Theta^t$ the lift of the Teichm\"uller flow to a flow on ${\cal H}$ 
defined 
by parallel transport for the Gauss Manin 
connection. Recall that $\Theta^t$ preserves
the symplectic structure on ${\cal H}$, but in general it
does not preserve the Hodge
norm. 
For $z\in {\cal Q}$ let ${\cal H}_z$ be the fibre of ${\cal H}$ at 
$z$. For  $1\leq i\leq g$ and for 
$t>0$ let 
\[\zeta_i(t,z)\] be the minimum of the 
operator norms of the restriction of $\Theta^t (z)$
to a symplectic subspace of 
${\cal H}_z$ of real dimension $2(g-i+1)$.
 Define
\[\kappa_i(t,z)=\frac{1}{t}\log \zeta_i(t,z).\]

Let $\epsilon >0, \delta >0$ and let $U$ be a neighborhood of 
a birecurrent point $q\in {\cal Q}_{\rm good}$. 
Since the Kontsevich Zorich cocycle is locally constant
(or, equivalently, the Gauss Manin connection is flat),
we can find a collection of nested neighborhoods 
$Z_0\subset Z_1\subset Z_2\subset V\subset U$
with the properties in Lemma \ref{technical} and such that
furthermore, with the notations from the lemma, 
if $z\in Z_0,R>R_0$ and if $\Phi^Rz\in Z_0$ then the 
periodic orbit $\gamma$ for $\Phi^t$ determined by the 
characteristic curve $\eta$ of the orbit segment
$\cup_{t\in [0,R]}\Phi^tz$ satisfies
\begin{equation}\label{alphakappa}
\vert \kappa_i(R,z)-\alpha_i(\gamma)\vert \leq \epsilon/2.
\end{equation} 

Namely, for a sufficiently small contractible 
neighborhood $V$ of $q$
in ${\cal Q}_{\rm good}$, the trivialization 
of ${\cal H}\vert V$ defined by the Gauss Manin connection
almost preserves the 
Hodge norm. Then the estimate
(\ref{alphakappa}) holds true if we replace $\alpha_i(\gamma)$ be the
$i$-th absolute value in decreasing order of an eigenvalue 
of the symplectic transformation $A_\eta$ of
${\cal H}_z$ which is 
defined by parallel transport for the Gauss Manin
connection along a characteristic curve $\eta$ for 
the orbit segment $\cup_{t\in [0,R]}\Phi^tz$. 
But by property (c) in Lemma \ref{technical}, the characteristic curve 
$\eta$ defines the same conjugacy class of a pseudo-Anosov mapping
class as $\gamma$. This means that the numbers $\alpha_i(\gamma)$ are
precisely the absolute values of the eigenvalues of the 
transformation $A_\eta$. Thus the estimate 
for the characteristic curve $\eta$ and the transformation 
$A_\eta$ implies the estimate for
$\gamma$.

By Oseledec's theorem and ergodicity, 
there is number $R(\epsilon)>R_0$ and a Borel subset
$B$ of $Z_0$ of measure $\lambda(B)>\lambda(Z_0)(1-\delta)$
with the following property. Let $u\in B$ and let $R>R(\epsilon)$; 
then 
$\vert \kappa_i(R,u)-\kappa_i\vert \leq \epsilon/2$.

Since the Lebesgue measure is mixing for the 
Teichm\"uller flow,
there is a number $R_1>R(\epsilon)$ such that
\[\lambda(\Phi^R B\cap B)\geq \lambda(B)^2(1-\delta)
\geq \lambda(Z_0)^2(1-\delta)^3
\geq \lambda(V)^2(1-\delta)^5\] for all 
$R\geq R_1$. By the estimate (b) in Lemma \ref{technical}, 
this implies that the number of 
components of the intersection 
$\Phi^R V\cap V$ 
containing points in $\Phi^R B\cap B$ is at least 
$e^{hR}\lambda(V)(1-\delta)^6$. By property (c) 
in Lemma \ref{technical},  
for each such component there is a periodic orbit of $\Phi^t$ 
passing through $Z_2\subset V$.
The estimate (\ref{alphakappa}) together with the definition of $B$ yields that  
each such orbit $\gamma$ satisfies $\chi_\epsilon(\gamma)=1$. 
Thus by (a) of Lemma \ref{technical}, 
each such component of intersection $\Phi^RV\cap V$ 
contributes $2t_0$ to the value $H_\epsilon(V,R-t_0,R+t_0)$.
Together this shows the proposition.
\end{proof}

As a corollary, we obtain the first part of Theorem \ref{theolyapunov}.
We formulate it more generally for strata of abelian or
quadratic differentials. As before, $\kappa_i$ denotes
the $i$-th Lyapunov exponent of the Kontsevich Zorich cocycle.

\begin{corollary}\label{sec2typical}
For $\epsilon >0$, the set $\{\gamma\in \Gamma\mid 
\vert\alpha_i(\gamma)-\kappa_i\vert <\epsilon\}$
$(1\leq i\leq g)$ is typical.
\end{corollary}
\begin{proof}
In \cite{H13} the following is shown.
As $R\to \infty$, the measures 
\[\mu_R=e^{-hR}\sum_{\gamma\in \Gamma,\ell(\gamma)\leq R}\delta(\gamma)\]
converge weakly to the Lebesgue measure on ${\cal Q}$.
The Lebesgue measure of ${\cal Q}-{\cal Q}_{\rm good}$ vanishes,
so it suffices to study the Lebesgue measure on ${\cal Q}_{\rm good}$. 

By \cite{EM11,EMR12,H11}, 
there is no escape of mass: We have
\begin{equation}\label{escapeofmass}
\sharp \{\gamma\in \Gamma\mid \ell(\gamma)\leq R\}\sim 
\frac{e^{hR}}{hR}.\end{equation}
In fact, there is a compact subset $K$ of ${\cal Q}$ such that
the growth rate of all periodic orbits
which do not intersect $K$ is strictly smaller than $h$.

Let $\epsilon >0$. By Proposition \ref{count} and 
Proposition \ref{lowerestimate}, the measures 
\[\mu_R=e^{-hR}\sum_{\gamma\in \Gamma,\ell(\gamma)\leq R}
\chi_\epsilon(\gamma)\delta(\gamma)\]
also converge weakly to the Lebesgue measure on ${\cal Q}$. 
Since there is no escape of mass, 
as in \cite{H13} it now follows from  
(\ref{escapeofmass}) that periodic orbits 
$\gamma$ with $\chi_\epsilon(\gamma)>0$ are typical.
\end{proof}

\section{Local Zariski density: The Zorich conjecture}\label{zariski}

In this section we prove the Zorich
conjecture which is stated as Corollary \ref{zorich} 
in the introduction.
Throughout the section, 
we denote by ${\cal Q}$ a component of a stratum
in the moduli space of area one 
abelian differentials on $S$,  
and by $\tilde {\cal Q}$ a component of its preimage in the
Teichm\"uller space of abelian differentials.
As in Section \ref{lyapunov}, we denote by 
${\cal Q}_{\rm good}$ the open dense $\Phi^t$-invariant 
subset of good points. 
All results of this section are also true for components of quadratic
differentials, however the proof requires some more details which 
we postpone to forthcoming work.

A periodic orbit of $\Phi^t$ on ${\cal Q}$ 
determines a conjugacy class in $Sp(2g,\mathbb{Z})$.
However, it will be convenient to look at actual elements of 
$Sp(2g,\mathbb{Z})$ rather than at conjugacy classes. 

Choose a birecurrent point 
$q\in {\cal Q}_{\rm good}$. Let 
$\delta >0$ and let
$Z_0\subset Z_1\subset Z_2\subset V\subset {\cal Q}_{\rm good}$ 
be a nested family of neighborhoods 
of $q$ in ${\cal Q}_{\rm good}$ 
as in Lemma \ref{technical}.

For sufficiently large $R_0>0$ 
let $z\in Z_0$ and let $T>R_0$ be such that
$\Phi^Tz\in Z_0$. 
A characteristic curve of this orbit segment 
determines uniquely a periodic
orbit $\gamma$ of $\Phi^t$ which intersects $Z_2$ in an arc 
of length $2t_0$. There may be more than one such 
intersection arc, but there is a unique arc determined by
the component of the intersection $Z_2\cap \Phi^TZ_2$ containing 
the point $z$ as described in (c) of Lemma \ref{technical}. 
Choose the midpoint of this intersection arc 
as a basepoint for $\gamma$ and as an initial point for 
a parametrization of $\gamma$. 

Let $\Gamma_0$ be 
the set of all parametrized 
periodic orbits of this form for points $z\in Z_0$ with
$\Phi^Tz\in Z_0$.
By Lemma \ref{technical},
the map which associates to a component of 
$\Phi^TV\cap V$ containing points in 
$\Phi^TZ_0\cap Z_0$ the corresponding parametrized 
periodic orbit in $\Gamma_0$ is a bijection.

Fix once and for all a 
lift $\tilde V$ 
of the contractible set $V$ to $\tilde {\cal Q}$. 
A periodic orbit
$\gamma$ which intersects $Z_2$ in an arc of length $2t_0$ 
lifts to a subarc of a flow line 
of the Teichm\"uller flow on $\tilde {\cal Q}$ 
with starting point in $\tilde V$. The 
endpoints of this arc are identified by a 
pseudo-Anosov element 
$\Omega(\gamma)\in  
{\rm Mod}(S)$.

The following \emph{shadowing property} 
is a variation of Lemma \ref{technical},
using the same notations. Versions of this lemma 
are familiar in hyperbolic dynamics.

\begin{proposition}\label{grouplaw}
For $\gamma_1,\dots,\gamma_k\in \Gamma_0$, there is 
a point $z\in Z_2$, and there are numbers
$0<t_1<\dots <t_k$ with the following properties.
\begin{enumerate}
\item $\Phi^{t_i}z\in Z_2$.
\item For each $i\leq k$, a characteristic curve 
$\zeta_i$ of the orbit segment
$\{\Phi^tz\mid t_{i-1}\leq t\leq t_i\}$ defines the same conjugacy class
in ${\rm Mod}(S)$ as the  
periodic orbit $\gamma_i$. 
\item 
There is a parametrized periodic orbit $\gamma$ for $\Phi^t$
with initial point in $Z_2$ which defines the same
conjugacy class in ${\rm Mod}(S)$ as a 
characteristic curve $\zeta$ of 
the orbit segment $\{\Phi^t z\mid 0\leq t\leq t_k\}$.
\item 
$\Omega(\gamma)=\Omega(\gamma_k)\circ \dots\circ \Omega(\gamma_1)$.
\end{enumerate}
\end{proposition}
\begin{proof} In the case that the arcs $\gamma_i$ are
contained in a fixed compact invariant subset $K$ for $\Phi^t$ and that
the set $V$ is chosen small in dependence of $K$,
the lemma is identical with the slight weakening of 
Theorem 4.3 of \cite{H10}.  That the 
statement holds true in the form presented here is immediate from the
construction of the set $Z_2$ in \cite{H13} and Proposition 5.4 of 
\cite{H13} which is based on an extension of Theorem 4.3 of \cite{H10}
to arbitrary orbits of the Teichm\"uller flow which recur to $Z_2$.
\end{proof}

As a consequence, the subsemigroup 
$\Omega(\Gamma_0)$ of ${\rm Mod}(S)$ generated by 
$\{\Omega(\gamma)\mid \gamma\in \Gamma_0\}$ 
consists of pseudo-Anosov elements whose 
corresponding periodic orbits are contained in the stratum
${\cal Q}$ and pass
through the set $Z_2$. This can be viewed as a version of 
Rauzy-Veech induction as used in \cite{AV07,AMY16} which is valid
for strata of quadratic differentials as well, or as a version of symbolic
dynamics for the Teichm\"uller flow on strata as in \cite{H11,H16}. 

Our next goal is to obtain information on the image of
this subsemigroup under the homomorphism
$\Psi:{\rm Mod}(S)\to SL(2g,\mathbb{Z})$. 
For this we choose an odd prime $p$ and study the image
of $\Psi\Omega(\Gamma_0)$ under the natural reduction map
\[\Lambda_p:Sp(2g,\mathbb{Z})\to Sp(2g,F_p)\]
where $F_p$ is the field with $p$ elements.
Denote by $\iota$ the symplectic form on $F_p^{2g}$ which is
preserved by $Sp(2g,F_p)$.

The following lemma relies on the results in \cite{Hl08}.
For its formulation, define a \emph{transvection} in 
$Sp(2g,F_p)$ to be a map $A\in Sp(2g,F_p)$ which 
fixes a subspace of $F_p^{2g}$ of codimension one and has determinant one
(see \cite{Hl08}). Any map of the form 
\[\alpha\to \alpha +\iota(\alpha,\beta)\beta\] for some 
$0\not=\beta\in F_p^{2g}$
(here as before, $\iota$ is the symplectic form)
is a transvection. We call this map 
a \emph{transvection by $\beta$}.

\begin{lemma}\label{transvection}
Let $p\geq 3$ be an odd prime and let 
$G<Sp(2g,F_p)$ be a subgroup generated by $2g$ transvections
by the elements of a set ${\cal E}=\{e_1,\dots,e_{2g}\}\subset 
F_p^{2g}$ which spans $F_p^{2g}$.
Assume that there is no nontrivial partition 
${\cal E}={\cal E}_1\cup {\cal E}_2$ so that 
$\iota(e_{i_1},e_{i_2})=0$ for all $e_{i_j}\in {\cal E}_j$. 
Then $G=Sp(2g,F_p)$.
\end{lemma}
\begin{proof} For each $i$ write 
$A_i(x)=x+\iota(x,e_i)e_i$. 
Let $G<Sp(2g,F_p)$ be the subgroup 
generated by the transvections 
$A_1,\dots,A_{2g}$. Since the vectors $e_1,\dots,e_{2g}$ span 
$F_p^{2g}$, the intersection of the invariant subspaces 
of the transvections 
$A_i$ $(i\leq 2g)$ is trivial.

We claim that the standard representation of $G$
on $F_p^{2g}$ is irreducible. 
Namely, assume to the contrary that there is an
invariant proper linear subspace $W\subset F_p^{2g}$. 
Let $0\not=w\in W$; then 
there is at least one $i$ so that 
$\iota(w,e_i)\not=0$. By invariance, we have  
$w+\iota(w,e_i)e_i\in W$ and hence 
$e_i\in W$ since $F_p$ is a field.

As a consequence, $W$ is spanned by some of the $e_i$, say by
$e_{i_1},\dots,e_{i_k}$, and if $j$ is such that
$\iota(e_{i_s},e_j)\not=0$ for some $s\leq k$ 
then $e_j\in W$. However, this implies that
$W=F_p^{2g}$ by the assumption on the set ${\cal E}=\{e_i\}$.

To summarize, $G$ is an irreducible subgroup of 
$Sp(2g,F_p)$ generated by transvections
(where irreducible means that the standard
representation of $G$ on $F_p^{2g}$ is irreducible).
Furthermore, as $p$ is an odd prime by assumption,  
the order of 
each of these transvections is not divisible by $2$.
Theorem 3.1 of \cite{Hl08} now yields that 
$G=Sp(2g,F_p)$ which is 
what we wanted to show.
\end{proof}

\begin{remark} 
By Proposition 6.5 of \cite{FM12}, Lemma \ref{transvection} is not
true for $p=2$.
\end{remark}

The following proposition is the main step towards the proof
of Corollary \ref{zorich} and is of independent interest. 
A slightly weaker version 
for affine invariant manifolds will be
established in Section \ref{local}.

\begin{proposition}\label{surjectivity}
\begin{enumerate}
\item For an odd prime $p\geq 3$ 
the subgroup of $Sp(2g,F_p)$ generated by 
$\{\Lambda_p\Psi(\Omega(\gamma))\mid
\gamma\in \Gamma_0\}$ equals the entire group $Sp(2g,F_p)$.
\item If ${\cal Q}$ is a stratum of abelian 
differentials with
at least one simple zero then the subgroup of 
${\rm Mod}(S)$ generated by $\{\Omega(\gamma)\mid \gamma\in \Gamma_0\}$
equals the entire mapping class group.
\end{enumerate}
\end{proposition}
\begin{proof} We begin with showing the second part of the proposition.

The mapping class group is generated by  
finitely many Dehn twists about simple closed curves. 
A specific generating system are the 
so-called \emph{Humphries generators} (see p.112 of \cite{FM12}). 
These generators
are Dehn twists about 
the set of simple closed curves
$a_1,\dots,a_g,c_1,\dots,c_{g-1},m_1,m_2$ shown in the Figure 1.
\begin{figure}[ht]
\begin{center}
\includegraphics[width=0.9\linewidth]{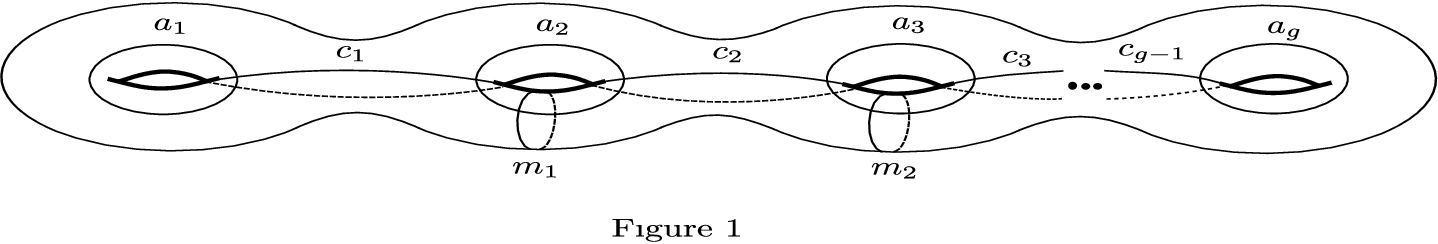}
\end{center}
\end{figure}
 
Let ${\cal Q}$ be a stratum of abelian
differentials with at least
one simple zero. By \cite{KZ03}, ${\cal Q}$ is connected.
For a simple closed curve $c$ 
denote by $T_c$ the positive Dehn twist about $c$. 

For the second part of the proposition it suffices to show 
that for each of the simple closed curves $c$ shown in 
Figure 1, there is $\sigma(c)\in \{1,-1\}$ so that 
the subgroup of ${\rm Mod}(S)$  generated by 
$\Omega(\gamma)$ $(\gamma\in \Gamma_0)$ contains
$T_c^{\sigma(c)}$.

By \cite{H16}, to each component ${\cal Q}$ of a stratum
there is associated a collection ${\cal T}({\cal Q})$ of 
\emph{large train tracks} which parametrize the component in 
a sense described in \cite{H16}. 
In particular, if $\tau\in {\cal T}({\cal Q})$ and if $\phi\in {\rm Mod}(S)$ 
is a pseudo-Anosov mapping class with \emph{train track expansion}
$\tau$ (this means that the train track $\phi(\tau)$ is \emph{carried} by
$\tau$, a property which we denote by $\phi(\tau)\prec \tau$ 
in the sequel), then the periodic orbit of the Teichm\"uller flow corresponding
to the conjugacy class of $\phi$ is contained in the closure of ${\cal Q}$. 
This statement can be viewed as 
a version of Rauzy Veech induction.

If ${\cal U}$ is a component
of a stratum and if ${\cal U}$ is contained in the closure of 
${\cal Q}$, then any train track 
associated to ${\cal U}$ can be obtained from some
train track associated to ${\cal Q}$ by removing some branches \cite{H16}. 
Furthermore, if we call a train track $\tau$ \emph{orientable} if
there exists a consistent orientation of the branches of 
$\tau$ (here consistent means that the orientation 
is compatible at the switches), then orientable train tracks
correspond to strata of abelian differentials. 

Let now ${\cal U}$ be the stratum of abelian 
differentials with one simple zero
and one zero of order $2g-3$.
We claim that 
there is a train track $\tau$ for ${\cal U}$ with the following
property. For each 
of the Humphries generators $T_c$ of ${\rm Mod}(S)$ 
there exists 
$\sigma(c)\in \{1,-1\}$ such that
we have $T_c^{\sigma(c)}(\tau)\prec \tau$, ie   
the train track
$T_c^{\sigma(c)}(\tau)$ is \emph{carried} by $\tau$.

Let as before $\iota$ be the intersection form on $H_1(S,\mathbb{Z})$. 
Denote by $[c]$ the homology class of an 
oriented simple closed curve $c$.
Orient the curves in Figure 1 in such a way that for each
$i$ we have $\iota([c_{i}],[a_{i}])=1= 
\iota([c_{i}],[a_{i+1}])$ and 
$\iota([a_2],[m_1])=-1,\iota([a_3],[m_2])=1$. 
For example, we can find such an orientation so that 
the curves $a_{2i-1}$ in Figure 1 are oriented 
counter-clockwise, and the curves 
$a_{2i}$ clockwise.

Construct from this oriented curve system  
a train track $\eta$ by replacing each intersection of 
oriented curves by a large branch as shown in Figure 2.
\begin{figure}[ht]
\begin{center}
\includegraphics[width=0.7\linewidth]{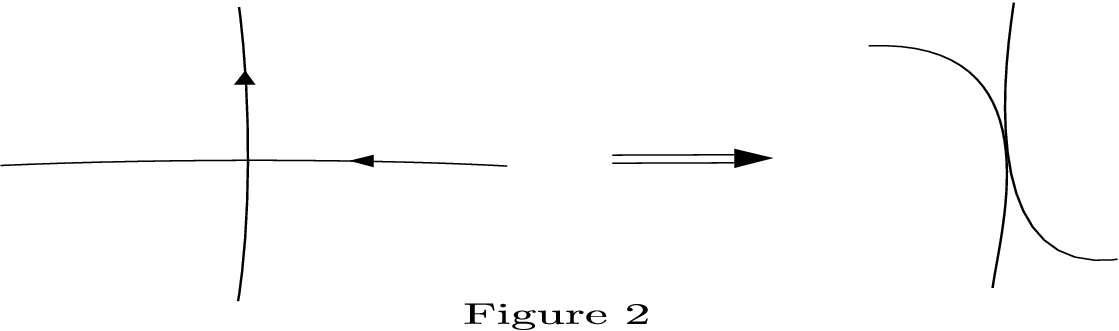}
\end{center}
\end{figure}
Informally, this amounts to following the oriented curves 
$a_i,c_i$ in the direction prescribed by their orientation. 
It is easy to check that $\eta$ is orientable. 

Each of the curves $a_i,c_i,m_i$ is embedded in $\eta$, and 
the image of $\eta$ under  
either the positive or the negative 
Dehn twist about any of these curves is 
carried by $\eta$. 
Namely, the curves are embedded
as a subgraph which either consists of a single large branch and
a single small branch (in the case of the curves
$a_1,a_g,m_1,m_2$) or of two large branches and two 
small branches as shown in Figure 3 below, or of three
large branches and three small branches (for the curves $a_2,a_3$). 
For a simple closed curve $c$ as shown in 
Figure 3, the train track obtained from $\tau$ by a single
positive Dehn twist about $c$ is obtained from two 
\emph{splits} at the two large branches shown in the figure. 
\begin{figure}[ht]
\begin{center}
\includegraphics[width=0.7\linewidth]{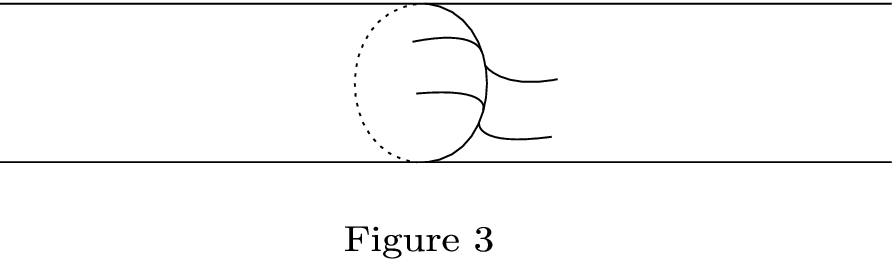}
\end{center}
\end{figure}

The train track $\eta$ has two complementary components, one
of them is a four-gon. It is not hard to see that $\eta$ is 
\emph{large} in the sense of \cite{H16} (ie it is birecurrent, and it carries
a large geodesic lamination of the same combinatorial type 
as $\eta$). Thus $\eta$ is associated to the 
stratum ${\cal U}$ of differentials
with one simple zero and one zero of order $2g-3$ \cite{H16}. 
By construction, it has the properties required in the above claim.

A stratum ${\cal Q}$ of abelian differentials with a simple zero
contains ${\cal U}$ in its closure \cite{KZ03}.
There is a train track $\tau$ for ${\cal Q}$ so that
$\eta$ can be obtained from $\tau$ by removing some 
branches \cite{H16}, i.e. $\eta$ is a subtrack of $\tau$.
We can choose $\tau$ in such a way that for each of the 
Humphries generators $T_c$, we have $T_c^{\sigma(c)}\tau\prec \tau$.
Thus $\tau$ is a train track as required in the above claim
(see again \cite{H16} for details of this construction).

Choose a pseudo-Anosov mapping class $\phi$ 
which admits $\tau$
as a \emph{train track expansion}.
This means that $\phi\tau\prec\tau$.
Assume that $\phi$ maps every branch of $\tau$ \emph{onto} $\tau$.
This will guarantee that the periodic orbit $\gamma$ defined 
by $\phi$ is contained
in ${\cal Q}$ rather than in the boundary of ${\cal Q}$, see \cite{H16}. 
Then for each of the Humphries generators $T_c^{\sigma(c)}
\in {\rm Mod}(S)$,
the composition
$\phi\circ T_c^{\sigma(c)}$ satisfies 
$(\phi\circ T_c^{\sigma(c)})\tau\prec\tau$, moreover 
$\phi\circ T_c^{\sigma(c)}$ maps every 
branch of $\tau$ \emph{onto} $\tau$. 
But this just means
that $\phi\circ T_c^{\sigma(c)}$ is a pseudo-Anosov 
mapping class which admits $\tau$
as a train track expansion. In particular, 
$\phi\circ T_c^{\sigma(c)}$ determines a
periodic orbit in ${\cal Q}$. We may assume without
loss of generality that this orbit is contained 
in ${\cal Q}_{\rm good}$.

We next show  
that for every periodic orbit 
$\gamma$ in ${\cal Q}_{\rm good}$ defined by the conjugacy
class of a pseudo-Anosov mapping class $\phi$ 
with train track expansion $\tau$
as above and for 
every neighborhood $U_0$ of $\gamma(0)$ (see the
proof of Proposition \ref{lowerestimate}), the subgroup of 
${\rm Mod}(S)$ generated by 
the parametrized orbits in the set $\Gamma_0$ contains 
each of the Humphries generators $T_c^{\sigma(c)}$
and hence this group
is the entire mapping class group. 

Namely, write $\beta=T_c^{\sigma(c)}$ for one
of these Dehn twists. By the above discussion,
for each $k$ the mapping class
$\phi^k\circ \beta\circ\phi^k$ is pseudo-Anosov, 
with train track expansion $\tau$,  
and it defines a periodic orbit
$\gamma_{\phi^k\circ \beta\circ \phi^k}$ in ${\cal Q}$. 
As $k\to \infty$, the normalized $\Phi^t$-invariant
measures supported on the periodic orbits 
$\gamma_{\phi^k\circ \beta\circ \phi^k}$
converge weakly to the normalized Lebesgue measure on $\gamma$. 
Namely, the vertical projective measured laminations 
$\mu_k,\mu$ of  
abelian differentials $q_k,q\in {\cal Q}$ which
generated the periodic orbit
$\gamma_{\phi^k\circ \beta\circ \phi^k},\gamma$
are all carried by 
$\tau$, and $\mu_k\to \mu$ in  
the space of projective transverse measures on $\tau$.
By the same reasoning, the 
horizontal projective measured geodesic laminations of 
$q_k$ converge to the projective measured geodesic lamination of $q$.

Thus for sufficiently large $k$ the periodic orbit 
$\gamma_{\phi^k\circ \beta\circ\phi^k}$ 
passes through the set $Z_0\subset U_0$ 
used for the construction of the set $\Gamma_0$ 
and hence it defines an element of $\Gamma_0$. 
Then 
$\beta=\phi^{-k}\circ (\phi^k\circ\beta\circ\phi^k)\circ\phi^{-k}$ 
is contained in the
group generated by $\{\Omega(\gamma)\mid \gamma\in \Gamma_0\}$
as claimed. As a consequence,
the second part of the proposition holds true for any choice of a 
neighborhood $U_0$ of $\gamma(0)$.

To show the second part of the proposition for an 
arbitrary open neighborhood $U$ in ${\cal Q}$ of a birecurrent
point $q\in {\cal Q}_{\rm good}$, recall that 
by ergodicity of the Teichm\"uller flow and the Anosov closing
lemma established in \cite{H13}, ``generic'' periodic
orbits for $\Phi^t$ passing through $U$ become 
equidistributed for the Lebesgue measure
(see \cite{H13} for a detailed discussion). 
In particular, they pass through the set $Z_0\subset U_0$ chosen 
as above.  

Now use the argument for the set $Z_0\subset U_0$ as follows.
Let $\Gamma_1$ be the set of periodic orbits constructed from $q$
and the open neighborhood $U$ of $q$. Let $\gamma\in \Gamma_1$ 
be a generic periodic orbit which passes
through the set $Z_0$.
Let $T>0$ be such that $\gamma(T)\in Z_0$.
Let $\hat \gamma$ be the reparametrization of $\gamma$
which satisfies $\hat\gamma(0)=\gamma(T)$. 
Apply the 
above construction to $\hat \gamma$ and the pseudo-Anosov mapping class 
$\hat \phi$ defined by the parametrized orbit $\gamma$. 
We conclude that 
for a sufficiently large $k$, a reparametrization of 
the periodic orbit corresponding to $\phi^k\circ \beta\circ \phi^k$
is contained in $\Gamma_1$. 
By the reasoning used for the set $U_0$, we obtain 
the second part of the proposition for 
periodic orbits passing through $U$.

The first part of this proof also immediately implies the first 
part of the proposition for strata of abelian differentials
with at least one simple zero. We are left with showing 
the first part of the proposition for arbitrary strata ${\cal Q}$
of abelian differentials.

Let $c$ be an oriented non-separating simple closed curve on $S$
which defines the homology class $[c]$. 
The action on 
homology of the positive Dehn twist $T_c$ about $c$ equals 
\begin{equation}\label{dehntwist}
T_c(\alpha)=
\alpha+\iota(\alpha,[c])[c]
\end{equation} 
 (Proposition 6.3 of \cite{FM12}).
In other words, the Dehn twist $T_c$ 
acts on $H_1(S,\mathbb{Z})$ as a 
transvection by $[c]$.

By the main result of \cite{KZ03}, for $g\geq 4$  
the stratum ${\cal H}(2g-2)$ 
of abelian differentials with a single zero
consists of three connected components.
One of these components is hyperelliptic, the other two components
are distinguished by the parity of the spin structure they define.
The stratum ${\cal H}(4)$ consists of two components;
one component is hyperelliptic, 
the second component has odd spin structure. 
The stratum ${\cal H}(2)$ is connected.

Any component ${\cal Q}$ of a stratum with more than one zero
contains a component of ${\cal H}(2g-2)$ in its closure. 
Thus following the reasoning in the first part of this proof,
it suffices to find for each of the  components 
${\cal V}$ of ${\cal H}(2g-2)$
a train track $\zeta$ 
associated to ${\cal V}$ with the following 
property. Let $p\geq 3$ be an odd prime. Then the subgroup of 
$Sp(2g,F_p)$ which is generated by those transvections in $Sp(2g,F_p)$ which 
are images under $\Lambda_p\circ \Psi$ of 
Dehn twists about embedded curves
$c$ in $\zeta$ with $T_c^{\sigma(c)}\zeta\prec\zeta$ 
$(\sigma\in \{-1,1\})$ is all of 
$Sp(2g,F_p)$.

Let again $\eta$ be the train track
constructed in the beginning of this proof from the Humphries 
generators of ${\rm Mod}(S)$. 
Let $\zeta$ be the train track obtained from $\eta$ by removing the small
branch contained in the curve $m_2$. This train track is orientable,
filling and birecurrent. 
If $g\geq 3$ then $\zeta$ is not invariant
under a hyperelliptic involution. Hence it corresponds to one of the two 
non-hyperelliptic components of ${\cal H}(2g-2)$. 
These components are distinguished by the parity of the spin structure they
define. The formula in the proof of Proposition 4.9 of \cite{H16} calculates
this parity (see also the formulas in \cite{KZ03}). 
It is odd if $g$ is odd, and even if $g$ is even.


For each curve $c\in \{c_i,a_i,m_1\}$ 
we have $T_c^{\sigma(c)}\zeta\prec\zeta$.
Now for any choice of orientations of the curves in 
$\{c_i,a_i,m_1\}$, the homology classes $\{[c_i],[a_i],[m_1]\}$
are a basis of $H_1(S,\mathbb{Z})$. 
Moreover, any two curves intersect in at most one point and
their union is a connected subset of 
$S$. This implies that the transvections by the homology classes of
these curves satisfy the assumptions in 
Lemma \ref{transvection}. As a consequence,  
for each odd prime $p\geq 3$ these transvections
generate $Sp(2g,F_p)$. 
This is what we wanted to show.

The other two components of ${\cal H}(2g-2)$ are treated in the
same way. 
A train track for the hyperelliptic component of ${\cal H}(2g-2)$ 
can be constructed from the following curve system.
Remove the curves $m_1,m_2$ from the Humphries generators and 
add a simple closed curve $c$ which intersects $a_1$ in a single point
and does not intersect any other of the curves shown in Figure 1.
The orientation of the curve $c$ can be chosen in such a way that all the
properties used above hold true. 
The resulting set of curves is clearly invariant under the hyperelliptic
involution. The same argument as for the non-hyperelliptic components 
discussed above yields the case of the hyperelliptic
component. In particular, we established the proposition for 
$g=3$ \cite{KZ03}.

We are left with finding for $g\geq 4$ a train track 
for the second non-hyperelliptic component of ${\cal H}(2g-2)$ with parity of 
the spin structure opposite to the parity of $g$. 
Thus let ${\cal Q}$ be the
component of abelian differentials in genus $g\geq 4$ with a single
zero with parity of the spin structure $g+1$ mod 2. By 
Lemma 14 of \cite{KZ03}, there are differentials $q\in {\cal Q}$ 
which can be obtained from an abelian differential
with a single zero in genus $g-1$ by ``bubbling a handle''.
This can be translated into train tracks as follows. 
Start with a train track $\eta$ on a surface $S^\prime$
of genus $g-1$ 
constructed above for the non-hyperelliptic 
component with a single zero in genus $g-1$
whose spin structure equals $g-1$ mod $2$. 
Attach to $S^\prime$ a 
cylinder $C$ by removing from $S^\prime-\eta$ two
small disks. There is no ambiguity here since
$S^\prime-\eta$ is connected. 
Following the strategy in \cite{H16}
we extend the train track $\eta$ by adding first 
an embedded simple closed curve
$c$ defining the core curve of $C$. Attach two small branches
$b_1,b_2$ to $c$, one at each side of $c$, 
such that after adding these
branches, $c$ consists of a single large branch and a single
small branch. Connect the branch $b_1$ to the curve $a_1$ 
and connect the branch
$b_2$ to the curve $a_2$ (notations are as in Figure 1)
using the above
orientation rule. It is now easy to check that the resulting
train track 
$\tau$ is large and defines the stratum of differentials
with parity of spin structure $g-1$ mod 2. 
Furthermore, $\tau$ 
carries a system of embedded simple closed curves
$a_i,b_i$ 
defining a basis for $H_1(S,\mathbb{Z})$ with the 
properties stated in Lemma \ref{transvection}. 
For each $i$ there is a train track 
$\alpha_i,\beta_i$ carried by $\eta$ and a choice of a sign
$\sigma(a_i),\sigma(b_i)$ such that
$T_{a_i}^{\sigma(a_i)}\alpha_i\prec\alpha_i,
T_{b_i}^{\sigma(b_i)}\beta_i\prec\beta_i$. 
The reasoning in the beginning of this proof 
can now be applied to suitably chosen 
pseudo-Anosov mapping classes
$\phi_i,\zeta_i$ with train track expansion $\tau$ 
so that $\phi_i\tau\prec \alpha_i,\zeta_i\tau\prec \beta_i$.
This completes the proof of the proposition.
\end{proof}

The first part of the following corollary 
establishes Zorich's conjecture \cite{Z99}
(we leave the easy translation into the language of 
Rauzy induction to the reader).  
For its formulation, 
call a component ${\cal Q}$ of a stratum  
\emph{locally Zarisky dense} 
if the following holds true.
Let $U$ be any open subset of ${\cal Q}_{\rm good}$ and let 
$\Omega(\Gamma_0)$ be the sub-semigroup of 
${\rm Mod}(S)$ generated 
in the sense discussed above by the
periodic orbits $\Gamma_0$ for $\Phi^t$ passing through $U$.  
We require that the sub-semigroup $\Psi\Omega(\Gamma_0)$ 
of $Sp(2g,\mathbb{Z})$ 
is Zariski dense in $Sp(2g,\mathbb{R})$.

\begin{corollary}\label{connected}
\begin{enumerate}
\item A component ${\cal Q}$ of a stratum is locally Zariski dense.
\item If ${\cal Q}$ is a stratum of abelian differentials with 
a simple zero then the preimage  
of ${\cal Q}$ in the Teichm\"uller space of area one
abelian differentials is connected.
\end{enumerate}
\end{corollary}
\begin{proof} The first part of the corollary
follows from the first part of 
Proposition \ref{surjectivity} and the well known fact that
a subgroup 
of $Sp(2g,\mathbb{Z})$ which surjects onto $Sp(2g,F_p)$ for all  
odd primes $p\geq 3$ is Zariski dense in $Sp(2g,\mathbb{R})$
(see \cite{Lu99} for more and for references).

Now let ${\cal Q}$ be a stratum of abelian or
quadratic differentials with a simple zero and 
let $\tilde {\cal Q}$ be a component of 
the preimage of ${\cal Q}$ in the Teichm\"uller space of 
abelian differentials. 
By the second part of Proposition \ref{surjectivity},
the stabilizer of $\tilde {\cal Q}$ in the mapping class group
equals the entire mapping class group. As the components of
the preimage of ${\cal Q}$ 
are permuted by the mapping class group, the
second part of the corollary follows.
\end{proof}

\begin{remark}\label{topology}
More generally, one can ask about the orbifold fundamental
group of a component ${\cal Q}$ of a stratum
of abelian or quadratic differentials.
If ${\cal Q}$ is a stratum with $k\geq 2$ zeros then 
the zero forgetful map maps this orbifold fundamental group into
the mapping class group ${\rm Mod}(S)$ of $S$. 
Proposition \ref{surjectivity} shows that for strata with a simple zero,
this map is onto. For some strata in genus 3, the
fundamental group has been identified with tools from algebraic 
geometry in \cite{LM14}. 

For most components of strata, we do not even know the
image of the orbifold fundamental group in 
$Sp(2g,\mathbb{Z})$ besides surjecting onto $Sp(2g,F_p)$ for 
all odd primes $p$, see 
Proposition \ref{surjectivity}.  
For hyperelliptic components, this image 
has been determined in \cite{AMY16}.
One finds that the group is precisely
the subgroup stabilized by 
the hyperelliptic involution, in particular it is of finite index in 
$Sp(2g,\mathbb{Z})$. We conjecture that the
latter property holds true for all components of all strata. 
\end{remark}

\section{Galois groups}\label{galois}

In this section we consider again 
an arbitrary component ${\cal Q}$ of a stratum of abelian 
differentials. We continue to use the assumptions
and notations from section \ref{lyapunov} and Section \ref{zariski}.
Recall in particular the construction of the
set $\Gamma_0$ of parametrized periodic orbits in ${\cal Q}_{\rm good}$ 
defined by a small neighborhood
of a point $q\in {\cal Q}_{\rm good}$ which is birecurrent under the 
Teichm\"uller flow. We showed in Section \ref{zariski} that this
set determines a sub-semigroup $\Omega(\Gamma_0)$ 
of ${\rm Mod}(S)$ consisting of 
pseudo-Anosov elements whose image under 
the homomorphism $\Psi$ is 
Zariski dense  $Sp(2g,\mathbb{R})$. 
More precisely, by the first part of Proposition \ref{surjectivity},
for every odd prime $p$ the image   
$G_p=\Lambda_p(\Psi\Omega(\Gamma_0))$ 
of the semi-group $\Psi(\Omega(\Gamma_0))$ 
generates the entire group 
$Sp(2g,F_p)$. Here as before, $\Lambda_p:Sp(2g,\mathbb{Z})\to 
Sp(2g,F_p)$ denotes reduction modulo $p$. 

Now $Sp(2g,F_p)$ is a finite group and therefore for 
every $A\in Sp(2g,F_p)$ there is some
$\ell \geq 1$ such that $A^\ell=A^{-1}$. As a consequence,
for all $x,y\in G_p$ we have $xy^{-1}\in G_p$ as well and 
hence $G_p<Sp(2g,F_p)$ is a group. Then $G_p$ equals the group
generated by $\Lambda_p(\Psi\Omega(\Gamma_0))$ and thus
$G_p=Sp(2g,F_p)$.

Our next goal is to makes this statement quantitative. 
To this end denote for 
a periodic orbit $\gamma$ for $\Phi^t$ by $\delta_\gamma$ the 
$\Phi^t$-invariant measure supported on $\gamma$ whose total mass 
equals the period $\ell(\gamma)$ of $\gamma$. 
Recall that periodic orbits in the set $\Gamma_0$ 
are parametrized, so a single unparametrized
periodic orbit may give rise to many different elements of $\Gamma_0$.

Let as before $p$ be an odd prime and 
let $N(p)$ be the number of elements
of $Sp(2g,F_p)$. The following proposition holds true for any finite
group $G$  of order $N$ with the property that there is a homomorphism
$\rho:{\rm Mod}(S)\to G$ whose restriction to the semigroup 
$\Omega(\Gamma_0)$ is surjective.

\begin{proposition}\label{equiproject}
Let $B\in Sp(2g,F_p)$ be arbitrary, let $p$ be an odd prime
and define
\[{\cal B}(R,B)=\{\gamma\in \Gamma_0\mid 
 \ell(\gamma)\leq R,\Lambda_p\circ \Psi\circ\Omega(\gamma)=B\}.\]
Then as $R\to\infty$, 
\[\sharp {\cal B}(R,B)\sim \frac{e^{hR}\lambda(Z_0)}{2ht_0N(p)}\]
independent of $B$ up to a multiplicative error which 
only depends on $Z_0$ and which can be arranged to be 
arbitrarily close to
one. 
%
\end{proposition}
\begin{proof}
We show first that there is a number $a>0$ such that
\[\sharp{\cal B}(R,B)\geq ae^{hR}\] 
for all $B\in Sp(2g,F_p)$ and 
for all sufficiently
large $R$.  

To this end let $U$ be any neighborhood of a birecurrent 
point $q\in {\cal Q}_{\rm good}$.  
Let $R_0>0,\delta>0$ and let  
$Z_0\subset Z_1\subset Z_2\subset V\subset U$ 
be as in Lemma \ref{technical}.  
Let ${\cal G}$ be the sub-semigroup
of ${\rm Mod}(S)$ generated by 
$\{\Omega(\gamma)\mid \gamma\in \Gamma_0\}$.
By Lemma \ref{grouplaw}, this semigroup consists of 
pseudo-Anosov elements. Furthermore,  
each $\rho\in {\cal G}$ is represented by  
a parametrized periodic orbit $\zeta$ for $\Phi^t$ which intersects
the set $Z_2$ 
in a segment of length $2t_0$ containing $\zeta(0)$ as
its midpoint. Vice versa, 
every periodic orbit which passes through $Z_0$
admits a parametrization so that the corresponding
element of ${\rm Mod}(S)$ is contained in ${\cal G}$.
The 
sub-semigroup $\Psi({\cal G})$ 
of $Sp(2g,\mathbb{Z})$
is mapped by
$\Lambda_p$ onto the finite group 
$Sp(2g,F_p)$.

Since $Sp(2g,F_p)$ is a \emph{finite} group and the above 
argument applies to every neighborhood $U$ of 
$q$ in ${\cal Q}_{\rm good}$, in particular to $U=Z_0$, 
there is a number 
$\hat R>R_0$ with the following property. Let 
$A\in Sp(2g,F_p)$ be arbitrary. Then there is some 
$z\in Z_0$ with $\Phi^tz\in Z_0$ for some $t\in (R_0,\hat R)$ 
which determines a parametrized
periodic orbit $\gamma(A)\in \Gamma_0$ 
with $\Lambda_p\Psi(\Omega(\gamma(A)))=A$.

Write $Z=Z_0$.  
Let $v\in Z$ be such that
$\Phi^Tv\in Z,\Phi^{T+U}v\in Z$ for some $T,U>R_0$.
By Lemma \ref{technical} and Lemma \ref{grouplaw}, 
the pseudo-orbits $\{\Phi^tv\mid 0\leq t\leq T\}$
and $\{\Phi^tv\mid T\leq t\leq T+U\}$ determine 
two parametrized periodic orbits $\gamma_1,\gamma_2$ 
for $\Phi^t$ which define elements
$\Lambda_p\Psi(\Omega(\gamma_1)),
\Lambda_p\Psi(\Omega(\gamma_2))
\in Sp(2g,F_p)$. Let
$\gamma=\gamma_2\hat\circ \gamma_1$ be the periodic orbit
for $\{\Phi^tv\mid 0\leq t\leq T+U\}$. The notation 
$\gamma=\gamma_2\hat\circ\gamma_1$
indicates that the element
$\gamma$ of ${\rm Mod}(S)$ defined by $\gamma$ 
is the product of the elements of ${\rm Mod}(S)$ defined by 
$\gamma_1$ and $\gamma_2$.
We have 
\[\Lambda_p\Psi(\Omega(\gamma_2\hat\circ \gamma_1))=
\Lambda_p\Psi(\Omega(\gamma_2))\circ 
\Lambda_p\Psi(\Omega(\gamma_1)).\]

Fix an element $B\in Sp(2g,F_p)$. 
If $A\in Sp(2g,F_p)$ is arbitrary and if 
$\zeta\in \Gamma_0$ is such that  
$\Lambda_p\Psi(\Omega(\zeta))=A$ then
$\Lambda_p\Psi(\Omega(\gamma(BA^{-1})\hat\circ \zeta))=B$.
In particular, 
by Proposition \ref{count}, 
for sufficiently large $R>R_0$
the number of parametrized periodic orbits $\gamma\in \Gamma_0$ 
with $\ell(\gamma)\leq R+\hat R$ and 
$\Lambda_p\Psi(\Omega(\gamma))=B$ 
is not smaller than the number of orbits in $\Gamma_0$ 
of length at most $R$. 

For large enough $R$, the number of orbits in $\Gamma_0$
of length at most $R$ can be estimated as in \cite{H13}. 
Namely, 
by Proposition \ref{count} and the choice of the set $Z$, 
for the fixed number $\delta>0$ used in the construction of 
$Z$ and  
for large $R$, the volume of $Z$ with 
respect to the sum of the measures supported on 
periodic orbits of $\Phi^t$ of length
in the interval $[R-t_0,R+t_0]$ is contained in the interval
\[  [2t_0e^{hR}\lambda(Z)(1-\delta),2t_0e^{hR}\lambda(Z)(1+\delta)] .\]

A periodic 
orbit intersects the set  $Z$ in arcs of length $2t_0$
\cite{H13}. As each such intersection component defines an element
of $\Gamma_0$ (recall that $\Gamma_0$ consists of
parametrized orbits), we conclude that for large enough $R$ 
the number of elements of 
$\Gamma_0$ of length contained in $[R-t_0,R+t_0]$ 
roughly equals $e^{hR}\lambda(Z)$.
Summation yields that for sufficiently large $k$ 
the number of elements of 
$\Gamma_0$ of length at most $2kt_0$ roughly equals 
\[\sum_{i=1}^{k}e^{h(2i-1)t_0}\lambda(Z)\sim e^{2hkt_0}\lambda(Z)/2ht_0.\]


As a consequence, there is a number $a>0$ not depending on $B$ 
such that 
up to passing to a subsequence, the
measures 
\[he^{-hR}\sum_{\gamma\in {\cal B}(R,B)}\delta_\gamma\chi_Z[\gamma(-t_0),
\gamma(t_0)]\]
converge to a measure $\hat\lambda_B$ on $Z$ of total mass contained in 
$[a\lambda(Z)(1-\delta),
\lambda(Z)(1+\delta)]$.
Here as before, $\delta_\gamma$ is the $\Phi^t$-invariant 
measure supported on $\gamma$ 
of total mass $\ell(\gamma)$, and $\chi_Z$ 
is the characteristic function of $Z$.

It is immediate from the construction that $\hat\lambda_B$ is 
the restriction to $Z$ of a $\Phi^t$-invariant Borel measure 
$\lambda_B$ on ${\cal Q}$. 
By the main result of \cite{H13} (see also Proposition \ref{count}),
this measure is contained in the measure class of the
Lebesgue measure $\lambda$. Thus by
ergodicity of $\lambda$ under the Teichm\"uller flow,
we have $\lambda_B=c(B) \lambda$ for a number 
$c(B)\in [a,1]$ and such that $\sum_Bc(B)=1$. 

Our goal is to show that $c(B)$ is independent of $B$.
To this end recall that the Lebesgue measure $\lambda$ is 
mixing of all orders \cite{M82}. In particular, for 
large enough numbers $R,S>0$ we have
\[\lambda(Z\cap \Phi^RZ\cap \Phi^{R+S}(Z))\sim \lambda(Z)^3\]
and therefore $\lambda_B(Z\cap \Phi^RZ\cap \Phi^{R+S}Z)\sim c(B)\lambda(Z)^3$.

For $R>0,T>0$ and $B\in Sp(2g,F_p)$ 
let $\Gamma(R,T,B,Z)$ be the set of 
all parametrized periodic orbits $\gamma\in \Gamma_0$ for 
$\Phi^t$ with the following properties.
\begin{enumerate}
\item The length of $\gamma$ is contained in the interval 
$[R+T-t_0,R+T+t_0]$.
\item $\gamma$ is determined by some $v\in Z$ and a return time
to $Z$ which is close to $R+T$. Moreover,  
there is a number $U\in [R-t_0,R+t_0]$ such that
$\Phi^Uv\in Z$.
\item 
$\Lambda_p\Psi(\Omega(\gamma))=B$.
\end{enumerate}
Define similarly a set $\Gamma(R+T,B,Z)$ containing all orbits with 
properties
(1) and (3) above. 
It follows from the above discussion (compare \cite{H13}) that
\[\sharp \Gamma(R+T,B,Z)\sim c(B)\lambda(Z)e^{h(R+T)}/2ht_0\]
for large enough $R$ and similarly
\[\sharp \Gamma(R,T,B,Z)\sim c(B)\lambda(Z)^2e^{h(R+T)}/(2ht_0)^2.\]


Each orbit $\gamma\in \Gamma(R,T,B,Z)$ can be represented in the
form $\gamma=\gamma_2\hat\circ \gamma_1$ for some
$\gamma_1\in \Gamma(R,A,Z)$ and some
$\gamma_2\in \Gamma(T,BA^{-1},Z)$.
Moreover, if $\gamma_1\in 
\Gamma(R,A,Z)$ for some $A\in Sp(2g,F_p)$ then  
for any $\gamma_2\in \Gamma(T,BA^{-1},Z)$ we have
$\gamma_2\hat \circ \gamma_1\in \Gamma(R,T,B,Z)$.  

As a consequence,
for an arbitrarily chosen $\epsilon >0$ 
and for sufficiently large $R>0$, 
as $T\to \infty$ we observe that  
\begin{align}
\sharp \Gamma(R,T,B,Z)=\sum_{A\in Sp(2g,F_p)}
\sharp\Gamma(T,BA^{-1},Z) \sharp\Gamma(R,A,Z) \notag\\
\sim \frac{1}{(2ht_0)^2}
\lambda(Z)^2 e^{h(R+T)}\sum_{A\in Sp(2g,F_p)}c(BA^{-1})c(A).
\notag
\end{align}


Now let $B\in Sp(2g,F_p)$ be such that 
$c(B)=\min\{c(A)\mid A\}$.
Such an element exists since $Sp(2g,F_p)$ is finite.
Since 
$\sum_Ac(BA^{-1})=1$ we have 
$\sharp\Gamma(R,T,B,Z)\sim c(B) \lambda(Z)^2e^{h(R+T)}/(2ht_0)^2$ only if 
$c(A)=c(B)=\frac{1}{N(p)}$ for all $A$. 
The proposition follows.
\end{proof}

\begin{remark}\label{equidistribution}
Proposition \ref{equiproject} can be viewed as an 
equidistribution result for conjugacy classes of elements
in $Sp(2g,F_p)$ defined by 
periodic orbits of the Teichm\"uller flow which parallels the 
familiar equidistribution of random walks
on finite connected graphs. 
The main difficulty lies in the 
fact that periodic orbits represent conjugacy classes
in the mapping class group rather than actual elements, 
moreover we look at periodic orbits in strata rather than
at all periodic orbits. 
\end{remark}

Now we are ready to complete the proof of the second part of 
Theorem \ref{theolyapunov}.
To this end recall that the characteristic polynomial of 
a symplectic matrix $A\in Sp(2g,\mathbb{Z})$ 
is reciprocal of degree $2g$. The roots of such a polynomial 
come in pairs. The Galois group 
of the number 
field defined by the polynomial 
is a subgroup of the
semidirect product 
\[(\mathbb{Z}/2\mathbb{Z})^g\rtimes S_g\]
where $S_g$ is the symmetric group in $g$ elements
(see \cite{VV02} for a nice account on this classical fact),
and $S_g$ acts on $(\mathbb{Z}/2\mathbb{Z})^g$
by permutation of the factors. 

In the sequel we call 
the Galois group of the field defined by the characteristic polynomial of 
a matrix 
$A\in Sp(2g,\mathbb{Z})$ simply the Galois group of $A$.
It only depends on the conjugacy class of $A$. We say that the 
Galois group of $A$ is 
\emph{full} if it coincides with 
$(\mathbb{Z}/2\mathbb{Z})^g\rtimes S_g$.

Having a full Galois group makes also sense for an element in 
$Sp(2g,F_p)$. We use this as in \cite{R08} as follows.

Let $p\geq 5$ be a prime and let $N(p)$ be the number of elements of 
$Sp(2g,F_p)$. 
By Proposition \ref{equiproject}, for large enough $R$ and 
every $B\in Sp(2g,F_p)$, the number of orbits
$\gamma\in \Gamma_0$ 
of length at most $R$ with $\Lambda_p\circ\Psi\circ \Omega(\gamma)=B$ 
roughly equals $\frac{e^{hR}}{N(p)}\lambda(Z_0)/2ht_0$.
On the other hand, if we denote by
$R_p(2g)$ the subset of $Sp(2g,F_p)$
of elements with reducible 
characteristic polynomial then
\[\frac{\vert R_p(2g)\vert}{N(p)}<1-\frac{1}{3g}\] 
(see Theorem 6.2 of \cite{R08} for a reference to this
classical result of Borel).

We follow the proof of Theorem 6.2 of \cite{R08}.
Namely, let $p_1,\dots,p_k$ be $k$ distinct primes, and let
$K=p_1\cdots p_k$. 
Then the reduction $\Lambda_K(A)$
modulo $K$ of any element $A\in Sp(2g,\mathbb{Z})$ is defined, and
we have 
\[\Lambda_K(A)=\Lambda_{p_1}(A)\times \dots \times \Lambda_{p_k}(A).\]

It follows from the discussion preceding Proposition \ref{equiproject} 
that reduction mod $K$ defines a surjection of 
the semigroup $\Psi\Omega(\Gamma_0)$ 
onto the finite group 
$\Lambda_{p_1}(A)\times \dots \times \Lambda_{p_k}(A)=\Lambda_K(A)$.
Now if $A\in Sp(2g,\mathbb{Z})$ has a reducible characteristic
polynomial then the same holds true for 
$\Lambda_{p_i}(A)$ for all $i$. 
The proportion of the number of 
elements in $\Lambda_{p_1}(A)\times \dots \times \Lambda_{p_k}(A)$ 
with this property is at most $(1-\frac{1}{3g})^k$. 

By Proposition \ref{equiproject} (taking into account the comment
preceding the proposition), 
the above estimate implies that 
for large
enough $R$, the proportion 
of all orbits $\gamma$ of length at most $R$ with the property 
that the Galois group of 
the characteristic polynomial of $A(\gamma)$ is 
\emph{not} full is at most of the order of 
$(1-\frac{1}{3g})^k$. As $k\to \infty$, we conclude that
the Galois group of a typical periodic orbit for $\Phi^t$ is full.
Thus we have shown

\begin{corollary}\label{full}
Let ${\cal Q}$ be a component of the stratum of abelian or quadratic
differentials. The set of all $\gamma\in \Gamma$ such that the trace field
of $[A(\gamma)]$ is of degree $g$ over $\mathbb{Q}$, and 
$G(\gamma)=(\mathbb{Z}/2\mathbb{Z})^g\rtimes \Sigma_g$ is typical.
\end{corollary}

\begin{remark} 
The only property used in the proof of Corollary \ref{full} 
which is not available for components of strata of quadratic differentials
is the first part of Proposition \ref{surjectivity}. However, it is not 
hard to establish this part for strata of quadratic differentials, with the same
proof (and some extra combinatorial discussion). 
Thus we could use this
to show the analogue of Corollary \ref{full} for components of strata of 
quadratic differentials. This leads to stating that for a typical such orbit,
the trace field of the corresponding symplectic matrix $A$ has maximal
Galois group and hence either is totally real or completely imaginary. 
Furthermore, if this field is totally real then the roots of the characteristic
polynomial of $A$ are pairwise distinct. In spite of 
Corollary \ref{sec2typical}
and in contrast to the case of random walks on ${\rm Mod}(S)$ 
(see \cite{R08} for details), this does not imply however that the 
Lyapunov exponents of the Kontsevich Zorich cocycle over components of 
strata of quadratic differentials are pairwise distinct. 
\end{remark}

Let $\omega\in \tilde{\cal Q}$ be a lift of a point on a 
typical periodic orbit $\gamma$ for $\Phi^t$. The
periods of $\omega$ define an abelian subgroup 
$\Lambda=\omega(H_1(S,\mathbb{Z}))$ of 
$\mathbb{C}$ of rank two. Let $e_1,e_2\in \Lambda$ be two points
which are linearly independent over $\mathbb{R}$. 
Let $K$ be the smallest subfield of 
$\mathbb{R}$ such that every element 
of $\Lambda$ can be written as $ae_1+be_2$, with 
$a,b\in K$; then $\Lambda\otimes_KK=K^2$.
%
If we write $T=\Psi(A(\gamma))+\Psi(A(\gamma))^{-1}$, then the field $K$ 
also is the field of the characteristic polynomial
of $T$. We call $K$ the \emph{trace field} of $\gamma$
(see the appendix of \cite{KS00} for more details).

\begin{definition}\label{algebraic}
The periodic orbit $\gamma$ 
is called \emph{algebraically primitive} if 
the trace field $K$ of $\gamma$ is a totally real 
number field of 
degree $g$ over $\mathbb{Q}$, with maximal Galois group.
\end{definition}


The following corollary completes the proof of the second part of  
Theorem \ref{theolyapunov}.

\begin{corollary}\label{periodicprim}
Algebraically primitive periodic orbits for $\Phi^t$ are typical.
\end{corollary}
\begin{proof} Since the Lyapunov spectrum of ${\cal Q}$ is simple
\cite{AV07}, the first part of Theorem \ref{theolyapunov} 
implies that for a typical periodic orbit $\gamma$, 
the absolute values of the eigenvalues of 
$A(\gamma)$ are pairwise distinct and hence all eigenvalues are real. Thus 
by the discussion following Proposition \ref{equiproject},
we only have to show that for a symplectic matrix 
$A\in Sp(2g,\mathbb{R})$ with $2g$ distinct real eigenvalues
$r_i,r_i^{-1}$ $(i\leq g, r_i>1)$
the field defined by $A+A^{-1}$ is totally real.
However, this is immediate from the fact that
the roots of the polynomial defining the trace field
are of the form $r_i+r_i^{-1}$ where 
$r_i$ are the roots of the characteristic polynomial of $A$.
\end{proof}


\section{The local structure of affine invariant manifolds}\label{localaffine}

In this section we begin the investigation of affine invariant manifolds. 
Our first goal is to gain some understanding of their local structure.
Most or perhaps all of the statements in this section
are known to the experts but hard to find in the literature.

We begin with introducing  the geometric setup which will be 
used throughout the remainder of this paper.





A point in \emph{Siegel upper half-space} 
${\mathfrak{D}}_g={\rm Sp}(2g,\mathbb{R})/U(g)$ is a principally polarized 
abelian variety of dimension $g$. 
Here as usual, $U(g)$ denotes the unitary group of rank $g$.
There is a natural rank $g$ 
holomorphic vector bundle $\tilde {\cal V}\to {\mathfrak{D}}_g$ whose fibre
over $y$ is just the complex vector space defining $y$.
The polarization and the complex structure define a
Hermitean metric $h$ on $\tilde {\cal V}$.
The group $Sp(2g,\mathbb{R})$ acts from the left on 
the bundle $\tilde {\cal V}$ 
as a group of bundle automorphisms 
preserving the polarization and the complex structure, 
and hence this action preserves the 
Hermitean metric. Thus 
the bundle $\tilde{\cal V}$ projects to a holomorphic Hermitean
(orbifold) vector bundle 
\[{\cal V} 
\to Sp(2g,\mathbb{Z})\backslash Sp(2g,\mathbb{R})/U(g)={\cal A}_g.\]

Let ${\cal M}_g$ be the moduli space of closed Riemann surfaces
of genus $g$. 
The \emph{Hodge bundle} 
\[\Pi:{\cal H}\to {\cal M}_g\] is the pullback of the
holomorphic bundle ${\cal V}\to 
{\cal A}_g$
under the \emph{Torelli map}
\[{\cal J}:{\cal M}_g\to {\cal A}_g=
Sp(2g,\mathbb{Z})\backslash Sp(2g,\mathbb{R})/U(g).\]
As the Torelli map is holomorphic, 
${\cal H}$ is a $g$-dimensional holomorphic Hermitean
vector bundle on ${\cal M}_g$ (in the orbifold sense). 
Its fibre over $x\in {\cal M}_g$
can be identified with 
the vector space of holomorphic one-forms on $x$.
The Hermitean inner product on ${\cal H}$ 
is given by 
\[(\omega,\zeta)=\frac{i}{2}\int \omega\wedge\overline{\zeta}.\]
With this interpretation, 
the sphere bundle in ${\cal H}$ for the 
inner product $(,)$ is just  the moduli space 
of area one abelian differentials.


Each point $x\in {\cal M}_g$ determines a 
complex structure $J_x$ on the first real cohomology $H^1(S,\mathbb{R})$.
Namely, every cohomology class $\alpha\in H^1(S,\mathbb{R})$ 
can be represented by a unique harmonic one-form for the
complex structure $x$, and this one-form is the real part of 
a unique holomorphic one-form $\omega$ on $x$. 
The imaginary part of this holomorphic one-form 
is a harmonic one-form which represents the 
cohomology class $J_x\alpha$. The complex structure $J_x$ is 
\emph{compatible} with the intersection form $\iota$, i.e. we have
\[\iota(J_x\alpha,J_x\beta)=\iota(\alpha,\beta)\text{ for all }
\alpha,\beta\in H^1(S,\mathbb{R}).\]

With this interpretation, the flat vector bundle 
over ${\cal M}_g$ whose fibre at any point equals the complex
cohomology group $H^1(S,\mathbb{C})=H^1(S,\mathbb{R})\otimes \mathbb{C}$
can be decomposed as 
\[H^1(S,\mathbb{C})={\cal H}\oplus \overline{\cal H}\]
where the \emph{holomorphic} bundle 
${\cal H}=\{\alpha+ iJ\alpha\mid \alpha\in H^1(S,\mathbb{R})\}$
admits a natural identification with the bundle of holomorphic one-forms 
on ${\cal M}_g$, i.e. ${\cal H}$ is just the Hodge bundle over ${\cal M}_g$.

As a real vector bundle, the bundle ${\cal H}$ is isomorphic to 
the bundle with fibre 
$H^1(S,\mathbb{R})$.
Via the identification
${\cal H}\sim H^1(S,\mathbb{R})$ as real vector bundles, 
we can view the Gauss Manin connection as a flat connection on
${\cal H}$. 
This bundle ${\cal H}$ is equipped with a 
complex structure $J$ which varies real analytically with 
the basepoint.



Denote by ${\cal H}_+\subset {\cal H}$ 
the complement of the zero section in
the Hodge bundle ${\cal H}$. This is a complex orbifold. 
The pull-back
\[\Pi^*{\cal H}\to {\cal H}_+\]
to ${\cal H}_+$ of the Hodge bundle on ${\cal M}_g$ is a 
holomorphic vector bundle on ${\cal H}_+$. The pull-back of the
Gauss-Manin connection is a flat connection on $\Pi^*{\cal H}$.
In the sequel we simply write ${\cal H}$ for the bundle
$\Pi^*{\cal H}$ whenever what is meant is clear from the context.
This is precisely the notation we used in the
previous sections, where we also used the Gauss Manin connection
on $\Pi^*{\cal H}$.

Let ${\cal Q}_+\subset {\cal H}_+$ be a component of 
a stratum of abelian differentials. 
We use the notation ${\cal Q}_+$ to indicate that
unlike in the previous sections,  
we do not normalize the area of an abelian differential.
Then ${\cal Q}_+$ is a complex suborbifold
of ${\cal H}_+$. Period coordinates for ${\cal Q}_+$ define the
complex structure.

The component $GL^+(2,\mathbb{R})$ 
of the identity of the full linear group $GL(2,\mathbb{R})$ 
acts on ${\cal H}$, and this 
action preserves ${\cal Q}_+$.
An \emph{affine invariant manifold} ${\cal C}_+$ in 
${\cal Q}_+$ is the closure in ${\cal Q}_+$ of an orbit 
of the $GL^+(2,\mathbb{R})$-action. 
Such a manifold is complex affine in period coordinates \cite{EMM15}. 
In particular, 
${\cal C}_+\subset {\cal Q}_+$ is a complex suborbifold. 
Period coordinates determine a projection
\[p:T{\cal C}_+\to {\cal H}\oplus \overline{\cal H}=H^1(S,\mathbb{C})\]
to absolute periods (see \cite{W14} for a 
clear exposition). 
The image $p(T{\cal C}_+)$ is a flat subbundle of 
the flat complex vector bundle 
${\cal H}\oplus\overline{\cal H}\vert_{{\cal C}_+}=
H^1(S,\mathbb{C})\vert_{{\cal C}_+}$.
The map $p$ is compatible with respect to the 
complex structure on ${\cal C}_+$ defined  
by period coordinates and the constant complex structure on 
$H^1(S,\mathbb{C})$.

By the main result of \cite{F16}, 
there is a \emph{holomorphic}
subbundle ${\cal Z}$ of ${\cal H}\vert_{{\cal C}_+}$ such that
\[p(T{\cal C}_+)={\cal Z}\oplus \overline{\cal Z}.\]  
We call ${\cal Z}$ the \emph{absolute holomorphic tangent bundle}
of ${\cal C}_+$. 
In particular, the bundle $p(T{\cal C}_+)$ is invariant under the 
complex structure.

As a real vector bundle, 
${\cal Z}$ is isomorphic to $p(T{\cal C}_+)\cap H^1(S,\mathbb{R})$.
This bundle is
invariant under the compatible complex structure, in particular
${\cal Z}\subset H^1(S,\mathbb{R})\vert_{{\cal C}_+}$ is 
symplectic. Moreover, the bundle is flat, i.e. it is 
invariant under the restriction of the Gauss Manin connection to
${\cal C}_+$ \cite{F16}.

Define the \emph{rank} of the affine invariant manifold
${\cal C}_+$ as 
\[{\rm rk}({\cal C}_+)=\frac{1}{2}{\rm dim}_{\mathbb{C}} \, p(T{\cal C}_+)
={\rm dim}_{\mathbb{C}}{\cal Z}.\] 
With this definition, components of strata are affine invariant
manifolds of rank $g$.

Every component ${\cal Q}$ of a stratum in 
the bundle of area one abelian differentials 
which consists of differentials with at least two zeros
admits 
a foliation ${\cal A\cal P}({\cal Q})$ whose leaves locally consist of differentials with 
the same absolute periods. This foliation is called the
\emph{absolute period foliation} 
(we adopt this terminology from \cite{McM13},
other authors call it the relative period foliation).
The leaves of this foliation admit  a complex 
affine structure (see e.g. \cite{McM13}).  




If ${\cal C}\subset {\cal Q}$ is an affine invariant manifold whose 
dimension is strictly bigger than its rank then 
${\cal C}$ intersects the leaves of the absolute period foliation of 
${\cal Q}$ nontrivially. This fact alone does not imply that 
${\cal C}\cap {\cal A\cal P}({\cal Q})$ is a foliation 
of ${\cal C}$. Our next goal is to establish
some
structural results on ${\cal C}\cap {\cal A\cal P}({\cal Q})$. 

We begin with 
collecting some more detailed information on the absolute 
period foliation of the stratum ${\cal Q}$. 
 Its tangent bundle $T{\cal A\cal P}({\cal Q})$ 
has an explicit description via so-called 
\emph{Schiffer variations} \cite{McM13}.

Let first $\omega$ be an abelian differential with a 
simple zero $p$.
There are four horizontal separatrices at $p$ for the flat metric 
defined by $\omega$. In a complex coordinate $z$ near  
$p$ so that $\omega=(z/2)dz$, 
the horizontal separatrices are the four rays contained in the 
real or the imaginary axis. 
The restriction of $\omega$ to these rays defines
an orientation on the rays. With respect to this orientation, 
the two rays contained in the real
axis are outgoing from $p$, 
while the rays contained in the
imaginary axis are incoming. The \emph{Schiffer variation} 
with weight one at $p$ is the tangent at $\omega$ of a deformation
obtained by cutting $S$ open along the vertical axis at $p$ and 
refold so that the singular point $p$ slides  
backwards along the incoming rays in the 
imaginary axis. We refer to \cite{McM13} for a more detailed
description.


If $\omega$ has a zero of order $n\geq 2$ at $p$ then 
the Schiffer variation at $p$ is defined as follows (see
p.1235 of \cite{McM13}).  Choose a coordinate $z$ near
$p$ so that $\omega=(z^n/2)dz$ in this coordinate. 
This choice of coordinate is unique up to multiplication with
$e^{\ell 2\pi i/(n+1)}$ for some $\ell\leq  n$. 
There are $n+1$ horizontal separatrices 
at $p$ for the flat metric defined by $\omega$ whose orientations
point towards $p$. For small $u>0$
cut the surface $S$ open along the initial subsegments of length $2u$ 
of these 
$n+1$ segments.  
The result is a $2n+2$-gon which we refold as in the 
case of a simple zero. As before, we call the tangent at $\omega$
of this deformation the Schiffer variation with weight one at $p$.

Now let ${\cal Q}$ be a stratum of abelian differentials consisting
of differentials with $k\geq 1$ zeros.
By passing to a finite cover we may assume that
the zeros are \emph{numbered}. For $\omega\in {\cal Q}$ 
let $Z(\omega)$ be the set of numbered zeros of $\omega$.
Let moreover $V(\omega)\sim \mathbb{C}^k$ be the complex vector space
freely generated by the set $Z(\omega)$. 
Then the tangent space $T{\cal A\cal P}({\cal D})$ 
of the absolute period foliation of ${\cal D}$ 
at $\omega$ is naturally isomorphic to the hyperplane in $V(\omega)$ 
of all points whose coordinates sum up to zero \cite{McM13,H15},
i.e. of points with zero mean.

More explicitly, 
let 
$\mathfrak{a}=(a_1,\dots,a_{k})\in \mathbb{R}^{k}$  be any 
$k$-tuple of \emph{real} numbers with $\sum_ia_i=0$.  Then 
$\mathfrak{a}$ defines a smooth vector field $X_{\mathfrak{a}}$ on 
${\cal Q}$ as follows. For each $\omega\in {\cal D}$, the value 
of $X_{\mathfrak{a}}$ at $\omega$ is the Schiffer variation for the 
tuple
$(a_1,\dots,a_{k})$ of speed parameters 
at the numbered zeros of $\omega$.
Thus $X_{\mathfrak{a}}$ is tangent to the absolute period foliation.
The real vector space of dimension $k-1$ spanned by these
vector fields is the tangent bundle of the \emph{real rel foliation}
${\cal R}$ which is the intersection of the absolute period foliation
with the \emph{strong unstable foliation} $W^{su}$ of ${\cal Q}$.
The leaf of this foliation through $q\in {\cal Q}$ locally consists
of all differentials with the same horizontal measured geodesic lamination
as $q$. 
We refer again to \cite{H15} for references. 
Denote by $\Lambda^t_{\mathfrak{a}}$ the flow defined by
the Schiffer variation $X_{\mathfrak{a}}$ for the weight $\mathfrak{a}$. 
The flow lines for this flow 
are contained in the leaves of ${\cal R}$. 

Similarly, we define the \emph{imaginary rel foliation} of ${\cal Q}$ 
to be the intersection of the absolute period foliation with 
the \emph{strong stable foliation} $W^{ss}$ of ${\cal Q}$.
The leaf of the foliation $W^{ss}$ through $q$ locally consists of 
all differentials with the same vertical measured geodesic lamination as $q$.
Exchanging the roles of the horizontal and the vertical foliation
in the definition of the Schiffer variations identifies the tangent bundle
of the imaginary rel foliation of ${\cal Q}$ with the purely imaginary
weight vectors of zero mean on the numbered zeros of the differentials
in ${\cal Q}$. As the tangent bundle of the absolute period 
foliation is spanned by its intersection with the tangent bundle
of the strong stable and the strong unstable foliation, 
mapping a real weight vector to its multiple with $i=\sqrt{-1}$
defines a natural almost complex structure $J$ on 
$T{\cal A\cal P}({\cal Q})$. This almost complex structure is in fact
integrable \cite{McM13}.

The Teichm\"uller flow $\Phi^t$ preserves the absolute period 
foliation.  
The following is Lemma 2.2 of \cite{H15}.

\begin{lemma}\label{invariance3}
$d\Phi^tX_{\mathfrak{a}}=e^{t}X_{\mathfrak{a}}$ 
and $d\Phi^tX_{i\mathfrak{a}}=e^{-t}X_{i \mathfrak{a}}$ for every
$\mathfrak{a}\in \mathbb{R}^k$ with zero mean.
\end{lemma}

We observe next that 
an affine invariant submanifold ${\cal C}$ of ${\cal Q}$ 
intersects the absolute period foliation of ${\cal Q}$ in a 
real analytic foliation 
${\cal A\cal P}({\cal C})$ with complex affine leaves. 
For the formulation, we denote by $\hat{\cal Q}$ a finite cover of ${\cal Q}$
on which the zeros of the differentials are numbered.

\begin{lemma}\label{affine}
Let ${\cal C}$ be an affine invariant submanifold of $\hat{\cal Q}$ 
with $r={\rm dim}({\cal C}_+)-2{\rm rk}({\cal C})>0$.
Then ${\cal C}$ intersects the real rel foliation in a real
analytic foliation of real dimension $r$. Furthermore, if 
$q\in {\cal C}$ and 
$\mathfrak{a}\in \mathbb{R}^{k}$ are such that
$X_{\mathfrak{a}}(q)\in T{\cal A\cal P}({\cal C})$ then 
$X_{\mathfrak{a}}(z)\in 
T{\cal A\cal P}({\cal C})$, $X_{i\mathfrak{a}}(z)=
JX_{\mathfrak{a}}(z)\in T{\cal A\cal P}({\cal C})$
for every $z\in {\cal C}$.
\end{lemma}
\begin{proof} 
Let ${\cal C}\subset \hat{\cal Q}$ be an affine invariant manifold. 
For the purpose of the lemma, we may assume that 
\[r={\rm dim}({\cal C}_+)-2 {\rm rk}({\cal C})>0.\] 
Then for each $q\in {\cal C}$ there is a vector 
$X\in T_q{\cal A\cal P}(\hat{\cal Q})$ which is 
tangent to ${\cal C}$.
By invariance of ${\cal C}$ under
the Teichm\"uller flow, we have
$d\Phi^t(X)\in 
T{\cal A\cal P}(\hat{\cal Q})\cap T{\cal C}$  
for all $t$.

A vector 
$X\in  T{\cal C}\cap T{\cal A\cal P}(\hat{\cal Q})$ decomposes as
$X=X^u+X^s$ where $X^u\in T{\cal A\cal P}(\hat{\cal Q})$
is real (and hence tangent to the strong unstable foliation) 
and $X^s$ is imaginary (and hence tangent
to the strong stable foliation). 
We claim that we can find a vector 
$Y\in  T{\cal C}\cap T{\cal A\cal P}(\hat{\cal Q})$ 
which either is tangent to the strong
unstable or to the strong stable foliation. 
To this end we may assume that $X^u\not=0$. Since this is an open
condition and since the Teichm\"uller flow on ${\cal C}$
is topologically transitive, we may furthermore assume 
that the $\Phi^t$-orbit 
of the footpoint $q$ of $X$ is 
dense in ${\cal C}$. 
Then there is a sequence
$t_i\to \infty$ such that $\Phi^{t_i}(q)\to q$. 

Choose any smooth norm $\Vert \,\Vert$ on $T\hat{\cal Q}$. 
As $X\not=0$, up to passing to a subsequence, 
\[d\Phi^{t_i}(X)/\Vert d\Phi^{t_i}(X)\Vert\] converges to a vector
$Y\in T_q{\cal A\cal P}(\hat {\cal Q})$ 
which is tangent to the strong unstable foliation.
As the bundle $T{\cal C}\cap T{\cal A\cal P}(\hat{\cal Q})$ 
is a smooth $d\Phi^t$-invariant subbundle of the 
restriction of the tangent bundle of 
$\hat{\cal Q}$ to ${\cal C}$, we have 
$Y\in T{\cal C}\cap T{\cal A\cal P}(\hat{\cal Q})$ which is what we
wanted to show.

Using Lemma \ref{invariance3} and density of the $\Phi^t$-orbit of $q$,
if $0\not=\mathfrak{a}\in \mathbb{R}^{k}$ is a vector of zero mean
such that 
$Y=X_{\mathfrak{a}}(q)$
then $X_{\mathfrak{a}}(u)\in T{\cal C}$
for all $u\in {\cal C}$. 
As a consequence,
${\cal C}$ is invariant under the flow $\Lambda_{\mathfrak{a}}^t$.

By invariance of $T{\cal C}_+$ 
under the complex structure $J$,
if $r=1$ then 
\[T{\cal C}\cap T{\cal A\cal P}(\hat {\cal Q})=
\mathbb{R}X_{\mathfrak{a}}\oplus J\mathbb{R}X_{\mathfrak{a}}\]
and we are done. Otherwise 
there is a tangent vector $Y\in T{\cal C}\cap
T{\cal A\cal P}(\hat {\cal Q})-\mathbb{C}X_{\mathfrak{a}}$.
Apply the above argument to $Y$, perhaps via
replacing the Teichm\"uller flow by its inverse. In finitely many
such steps we conclude that
there is a smooth subbundle 
${\cal B}$ of $T{\cal C}\cap T{\cal A\cal P}(\hat {\cal Q})$
which is tangent to the strong unstable foliation 
(i.e. real for the real structure), of rank $r$, and such that
$T{\cal C}\cap T{\cal A\cal P}(\hat{\cal Q})=\mathbb{C}{\cal B}$.
Moreover,  
if $\omega\in {\cal C}$ and if  $\mathfrak{a}\in \mathbb{R}^k$ is such that 
$X_{\mathfrak{a}}(\omega)\in {\cal B}$ 
then $X_{\mathfrak{a}}(q)\in {\cal B}$ for every $q\in {\cal C}$. 
As a consequence, 
${\cal C}$ is invariant under the flow
$\Lambda^t_{\mathfrak{a}}$.
The same argument applied to the imaginary
subbundle $i{\cal B}$ of $T{\cal C}\cap T{\cal A\cal P}(\hat {\cal Q})$ 
yields the statement of the lemma.
\end{proof}

Recall the definition of the foliation $W^{ss},W^{su}$ of ${\cal Q}$. 
As a corollary, we obtain

\begin{corollary}\label{affinelocal}
Let ${\cal C}$ be an affine invariant manifold with 
${\rm dim}_{\cal C}({\cal C}_+)=r$. Then 
${\cal C}\cap W^i$ is a smooth foliation of ${\cal C}$ into 
leaves of real dimension $r-1$. 
\end{corollary}
\begin{proof} We show the statement of the corollary for 
the foliation $W^{su}$, the statement for the foliation 
$W^{ss}$ is completely analogous.

By Lemma \ref{affine}, the intersection of ${\cal C}$ with a
leaf of $W^{su}$ is foliated into leaves of the real Rel
foliation of real dimension $r-2{\rm rk}({\cal C}_+)=m$. 

The image of the projection $p:T{\cal C}_+\to H	^1(S,\mathbb{C})$ 
is a flat complex subbundle of the flat bundle 
$H^1(S,\mathbb{C})\vert {\cal C}_+$ which is invariant 
under multiplication with $i$. This shows that 
for a differential $z\in {\cal C}_+$ near $q$,
the set of all differentials in ${\cal C}_+$ whose absolute periods
coincide with the absolute periods of 
$z$ and whose imaginary parts coincide with the imaginary
part of $z$ is a submanifold of 
${\cal C}_+$ of dimension $m$.
The intersection of ${\cal C}_+$ with the leaf of the
strong unstable foliation is the union of these
submanifolds over all points with the property that
the imaginary part of the absolute period coincides
with the imaginary part of the absolute period of $q$.
From this the corollary follows. 
\end{proof}

\section{Local Zariski density for affine invariant manifolds}
\label{local}

The goal of this section is to prove a weaker analogue of the first part of 
Proposition \ref{surjectivity} for affine invariant manifolds. 
Throughout this section we assume that $g\geq 3$. We use the 
assumptions and notations from Section \ref{localaffine}.

Let ${\cal Q}_+\subset {\cal H}_+$ 
${\cal C}_+$
Recall from Section \ref{localaffine} that the image 
of the projection $p:T{\cal C}_+\to H^1(S,\mathbb{C})$ 
to absolute periods is a flat subbundle 
${\cal Z}\oplus \overline{\cal Z}$ of 
$H_1(S,\mathbb{C})\vert {\cal C}_+$ which is invariant 
under the complex structure. We denote by $\ell\geq 1$ its complex
dimension. Then $p(T{\cal C}_+)\cap H^1(S,\mathbb{R})$ is a 
symplectic subspace of dimension $2\ell$ which can be identified
with the holomorphic bundle ${\cal Z}$ as a real vector bundle.

\begin{definition}\label{monodromygroup}
The \emph{monodromy group} of the affine invariant invariant manifold
${\cal C}_+$ of rank $\ell$ is the subgroup of $Sp(2\ell,\mathbb{R})$ which is  
the monodromy of the bundle ${\cal Z}$ 
for the restriction of the Gauss Manin connection.
\end{definition}

\begin{example} 
If ${\cal C}_+$ is a Teichm\"uller curve then the monodromy
group of ${\cal C}_+$ is just the Veech group of ${\cal C}_+$. 
Thus this monodromy group is a lattice in $Sp(2,\mathbb{R})=
SL(2,\mathbb{R})$, in particular it is Zariski dense in 
$SL(2,\mathbb{R})$. 

The monodromy group of a component of a stratum is a subgroup of 
$Sp(2g,\mathbb{Z})$. 
\end{example}

The goal of this section is to show

\begin{theorem}\label{monodromy}
The monodromy group of an affine invariant manifold ${\cal C}_+$ 
of rank $\ell$ is a Zariski dense subgroup of 
$Sp(2\ell,\mathbb{R})$.
\end{theorem}

We begin with summarizing a result of Wright \cite{W15}. 
He introduced the following two deformations of a translation
surface $(X,\omega)$ (i.e. a Riemann surface $X$ equipped
with a holomorphic one-form $\omega$).
 
Recall that  the \emph{horocycle flow} is defined as part of the 
$SL(2,\mathbb{R})$-action,
\[u_t=
\begin{pmatrix} 1 &t\\
0& 1\end{pmatrix} \subset SL(2,\mathbb{R}),\]
and the \emph{vertical stretch} is defined by
\[a_t=
\begin{pmatrix}
1 & 0\\
0 & e^t\end{pmatrix}\subset GL^+(2,\mathbb{R}).\]
For a collection ${\cal Y}$ of horizontal cylinders on a
translation surface $X$, define the \emph{cylinder shear}
$u_t^{\cal Y}(X)$ to be the
translation surface obtained by applying the horocycle flow
to the cylinders in ${\cal Y}$ but not to the rest of $X$.
Similarly, the \emph{cylinder stretch} $a_t^{\cal Y}(X)$ 
is obtained by applying the vertical stretch only to the cylinders in 
${\cal Y}$. 

The following lemma is a fairly easy consequence
of the work of 
Wright \cite{W15}. For its formulation, a translation surface
$(X,\omega)$ is called \emph{horizontally periodic} if 
it is a union of horizontal cylinders.

\begin{lemma}\label{rational}
Let ${\cal C}_+$ be an affine invariant manifold of rank $\ell$.
Suppose that $(X,\omega)\in {\cal C}_+$ is horizontally periodic,
and that there is a decomposition of $X$
into $\ell$ collections ${\cal Y}_1,\dots,{\cal Y}_\ell$ of horizontal
cylinders so that for each $i$, the cylinder shear 
$u_t^{{\cal Y}_i}(X)$
remains in ${\cal C}_+$. Then for each $i$, the moduli
of the cylinders in the collection ${\cal Y}_i$ are rationally dependent.  
\end{lemma}
\begin{proof}
For each $i$, the collection ${\cal Y}_i$ consists of $r_i\geq 1$ cylinders.
By assumption, the cylinder shear $u_t^{{\cal Y}_i}(X)$ remains entirely in 
${\cal C}_+$.
A local version of Lemma 3.1 of \cite{W15}, applied to this
cylinder shear rather than to the full horocycle flow 
(note that the proof of this local version is
identical to the proof given in Section 3 of \cite{W15}) shows the following.
If the moduli $m_1^i,\dots,m_{r_i}^i$ of the cylinders in ${\cal Y}_i$
are not rationally dependent then
there is a proper subcollection ${\cal V}$ of ${\cal Y}_i$ consisting of 
$1\leq s<r_i$ cylinders so that
the cylinder shear $u_t^{\cal V}$ 
for this subcollection is contained 
in ${\cal C}_+$. 

But then there are at least $\ell+1$ pairwise distinct 
collections of horizontal cylinders in $(X,\omega)$ with the property that 
the cylinder shear of $X$ 
for each of these collections
is contained in ${\cal C}_+$ (see Section 3 of \cite{W15} for details).
This violates Theorem 1.10 of \cite{W15} and yields that indeed,
for fixed $i$ the moduli $m_j^i$ $(1\leq j\leq r_i$) of the cylinders in 
${\cal Y}_i$ are rationally dependent. 
\end{proof}

Define a \emph{piecewise affine automorphism} of a translation surface
$(X,\omega)$ to be a continuous self-map $F:X\to X$ with the property
that there exists a decomposition $X=\cup_iX_i$ into finitely many
components with geodesic boundary which is preserved by $F$ and such 
that the restriction of $F$ to each of these components is affine. 
In contrast to an affine automorphism of $(X,\omega)$, we allow 
that the restriction of $F$ to some of the components $X_i$ equals the 
identity. 

The main consequence of Lemma \ref{rational} we are going to use
is the following

\begin{corollary}\label{veechgroup}
Let ${\cal C}_+$ be an affine invariant manifold of rank $\ell\geq 2$.
Then there is a horizontally periodic surface $(X,\omega)\in {\cal C}_+$
which admits an abelian group $D$ of rank $\ell$ of 
piecewise affine automorphisms. The group $D$ acts on $S$ as a group of 
Dehn-multitwists, and it acts on  
$H^1(S,\mathbb{R})$ as a group of transvections of rank $\ell$.
\end{corollary}
\begin{proof}
By Theorem 1.10 of \cite{W15} and its proof (more precisely
the results in Section 8 of \cite{W15}), 
the affine invariant manifold ${\cal C}_+$ contains 
a horizontally periodic surface $(X,\omega)$ 
which admits a decomposition into $\ell$ cylinder families
${\cal Y}_1,\dots,{\cal Y}_\ell$ with the properties
stated in Lemma \ref{rational}. 
Moreover, for each
$i$ and each $t$ the image of $X$ under the vertical stretch
$a_t^{{\cal Y}_i}(X)$ is contained in ${\cal C}_+$. These vertical
stretches commute. 

The vertical stretch $a_t^{{\cal Y}_i}$ changes the heights 
of the horizontal cylinders in the family ${\cal Y}_i$ 
while keeping their circumferences fixed. The image translation surfaces
are horizontally periodic. 
Using Lemma \ref{rational}, this implies that 
we can find
$t_1,\dots,t_\ell\in \mathbb{R}$ so that the modulus of \emph{every} horizontal cylinder 
in 
\[Z=a_{t_1}^{{\cal Y}_1}\cdots a_{t_\ell}^{{\cal Y}_\ell}(X)\] is rational.

Using again the results in Section 8 of \cite{W15}, 
the affine invariant
manifold ${\cal C}_+$ contains the images of the translation surface $Z$ 
under the cylinder shears $u_t^{{\cal Y}_i}(Z)$ where 
by abuse of notation, we denote again by ${\cal Y}_i$ the cylinder family on 
$Z$ which is the image of the horizontal cylinder family ${\cal Y}_i$ on $(X,\omega)$. 
As the moduli of all cylinders in the family ${\cal Y}_i$ 
are rational, these cylinder shears are eventually periodic. This means that 
for each $i$ there exists some number $r_i>0$ such that 
for some 
fixed marking of $Z$, the surface $u_{r_i}^{{\cal Y}_i}(Z)$ is the image of $Z$ by a 
Dehn multitwist $T_i$ about the core curves of the cylinders in ${\cal Y}_i$. 

Since the core curves of the horizontal cylinders in $Z$ are pairwise disjoint, 
the Dehn multitwists $T_i$ commute. Therefore these
multitwists generate an abelian group of rank $\ell$ of piecewise
affine automorphisms of $Z$.
The multitwist $T_i$ acts as a transvection on
$H_1(S,\mathbb{R})$ by some multiple of 
the homology class defined by the family ${\cal Y}_i$. The proposition now follows
from another application of Theorem 1.10 of  
\cite{W15}: The rank of the subspace of 
$H_1(S,\mathbb{R})$ spanned by the homology
classes of the cylinder families ${\cal Y}_i$ $(i\leq \ell)$ 
equals $\ell$.
\end{proof}

The following result is the analogue of 
the first part of Corollary \ref{connected} for 
affine invariant manifolds.
For its formulation, we define 
as in Section 3 an affine invariant
manifold ${\cal C}_+$ of rank $\ell$ 
to be \emph{locally Zariski dense} if for every open contractible subset
$U$ of ${\cal C}_+$ 
the subsemigroup of $Sp(2\ell,\mathbb{R})$ generated by
the monodromy of those periodic orbits for $\Phi^t$ in ${\cal C}_+$ 
which pass through $U$ 
is Zariski dense in $Sp(2\ell,\mathbb{R})$. Note that this makes
sense because the bundle ${\cal Z}\to {\cal C}_+$ is flat
and therefore admits a family of canonical 
symplectic trivializations over 
the contractible set $U$. Replacing one of these trivializations
by another one changes the local monodromy group by 
a conjugation.  
 

\begin{proposition}\label{localzariski}
An affine invariant manifold is locally Zariski dense.
\end{proposition} 
\begin{proof} 
Let ${\cal C}_+$ be an affine invariant manifold of 
rank $\ell\geq 1$, and let ${\cal C}\subset {\cal C}_+$ be its subset of 
differentials of area one.
We use an argument similar to the proof of Proposition \ref{surjectivity}
(compare also the proof of Theorem 5.1 of \cite{W14} for an
argument along these lines). Following 
the reasoning in the proof of Proposition \ref{surjectivity},
it suffices to show the existence of a single birecurrent point
$q\in {\cal C}$ with the following property. For every
open neighborhood $U$ of $q$, the subgroup of $SL(2\ell,\mathbb{R})$
generated by the monodromoy of those periodic orbits for
$\Phi^t$ in ${\cal C}$ which pass through $U$ is Zariski
dense in $SL(2\ell,\mathbb{R})$.

To this end choose a translation surface 
$(X,\omega)\in {\cal C}$ with the properties stated in 
Corollary \ref{veechgroup}. Denote by $H$ the abelian group of 
rank $\ell$ of Dehn multitwists which is contained in the group of 
piecewise affine automorphisms of $X$.

Choose a contractible 
neighborhood $U$ of $\omega$ in ${\cal C}$ which is small 
enough that the bundle ${\cal Z}\vert U$ is 
trivial as a symplectic vector bundle,
with a real analytic 
trivialization.
Let $\tilde {\cal C}$ be a component of the preimage of ${\cal C}$ in the
Teichm\"uller space of area one abelian differentials and 
let $\tilde U$ be a lift of $U$ to $\tilde {\cal C}$. We may assume that
the canonical
projection $\tilde U\to U$ is a homeomorphism. 

By Corollary \ref{affinelocal}, every point $q\in {\cal C}$ admits a 
neighborhood with a \emph{local product structure}. This means 
the following. Let $\Sigma$ be the set of zeros of a differential in ${\cal Q}$. 
By possibly decreasing the size of $U$ 
we may assume that there are two disjoint  
compact subsets $D,K$ of the 
set of (marked) projective measured geodesic laminations on $S$,
viewed as points in $H^1(S,\Sigma)$ via intersection,
with the following property. The sets $D,K$ are 
homeomorphic to closed balls of dimension 
$h={\rm dim}_\mathbb{C}({\cal C}_+)-1$. Furthermore, 
there is a continuous map $\omega:D\times K\to \tilde U$ 
such that 
for any pair $(\xi,\nu)\in D\times K$,
the vertical projective measured geodesic lamination of $\omega(\xi,\nu)$
equals $\xi$ and its horizontal projective measured geodesic lamination 
equals $\nu$. Moreover, 
there is some $\epsilon >0$ such that
\[\tilde V=\cup_{-\epsilon\leq t\leq \epsilon}\cup_{(\mu,\nu)\in D\times K}  
\Phi^t\omega(\mu,\nu)\]
is a compact subset of $\tilde U$ whose dense interior 
contains the lift $\tilde \omega$ of $\omega$ to $\tilde U$.

%

The cylinder shears of the translation surface $(X,\omega)$
which are used to construct the Dehn multi-twists $T_i$
preserve the horizontal projective measured lamination of $\omega$, but
they modify the vertical projective measured lamination. 
These cylinder shears define $\ell$ smooth paths $c_i$ in ${\cal C}$ which lift
to smooth paths $\tilde c_i$ in $\tilde {\cal C}$ beginning at the 
preimage $\tilde \omega$ of $\omega$ in $U$ and connecting
$\tilde \omega$ to $T_i\tilde \omega$. 

Fix $i\leq \ell$ and write $T=T_i$ and $\tilde c=\tilde c_i$ for simplicity. 
Recall that the set $D$ is homeomorphic to a compact ball of dimension $h$, and 
the same holds true for $TD$. Furthermore, $D\cap TD$ contains the 
projective class $\mu$ of the horizontal projective measured lamination 
of $\tilde \omega$. 
We claim that there is an open neighborhood $E$ of $\mu$ in $D$ which 
is contained in $D\cap TD$. 

To this end cover the compact path $\tilde c$ by finitely many
open subsets $U_i$ $(i=0,\dots,k)$ 
of $\tilde {\cal C}$ whose closures $\overline{U_i}$  
have a product structure as described above.
These product structures are defined by compacts sets
$D_i, K_i$ in 
the space of projective measured geodesic laminations. 
For each $i$, the set $D_i$ is  
homeomorphic to
an $h$-dimensional compact ball and coincides with the
set of all vertical projective measured laminations of all points 
in $\overline{U_i}$.

Let $i,j$ be such that 
$U_i\cap U_j\cap \tilde c\not=\emptyset$. Then 
$D_i\cap D_j$ contains the projective class $\mu$ of the horizontal
measured geodesic lamination of $\tilde \omega$. 
Now $U_i\cap U_j$ is open and therefore 
locally near a point $x\in U_i\cap U_j\cap \tilde c$, 
period coordinates define a local product structure on a neighborhood
$E$ of $x$ of the form described above, determined by
a pair of subsets of the space of projective measured geodesic
laminations which are homeomorphic to compact balls of dimension
$h$. These sets
contain all projective classes of vertical and horizontal measured
geodesic laminations, respectively, of points in $E$.
As $E\subset U_i\cap U_j$, this implies that 
$D_i\cap D_j$ contains a compact neighborhood of 
$\mu$ in $D$, 
and by induction, the same holds true for $\cap_iD_i$. 
In particular, there is a compact neighborhood $R\subset D\cap TD$
of $\mu$ which is homeomorphic to a compact ball of dimension $h$. 
Similarly, by making $K$ smaller we may assume that
the sets $K$ and $TK$ are disjoint.

Let $\tilde V=\cup_{-\epsilon\leq t\leq \epsilon}
\cup_{(\xi,\nu)\in R\times K}\Phi^t\omega(\xi,\nu)
\subset \tilde U$. 
Let $V\subset U$ be the projection of $\tilde V$ into ${\cal C}$.
Choose a periodic orbit $\gamma$ for $\Phi^t$ which is generated by
an abelian differential $q$ contained in 
the interior of $V$. Let $\tilde q$ be the preimage of $q$ in 
$\tilde V$. We may assume that the vertical projective 
measured lamination $\zeta_1$ of $\tilde q$ 
as well as its image under $T$ are contained in the interior of $R$.
We use the birecurrent point $q$ as a 
starting point for the construction of a set $\Gamma_0$ of periodic
orbits passing through $V$ as in Proposition \ref{grouplaw}. 
Denote by $\Omega(\Gamma_0)$ the corresponding sub-semigroup of 
${\rm Mod}(S)$ and let $G<{\rm Mod}(S)$ be the subgroup generated 
by $\Omega(\Gamma_0)$.  

We claim that $G$ contains the 
Dehn multitwist $T=T_i$. Since $G$ is a group, 
this is the case
if we can find a pseudo-Anosov
element $\phi\in \Gamma_0$ such that $T \phi\in \Gamma_0$. 

We establish this fact using a 
fixed point argument for the action of ${\rm Mod}(S)$ on 
the sphere of projective measured geodesic laminations 
which  is motivated by
the argument in the proof of Proposition 5.4 of \cite{H13}. 
Let $\phi$ be the 
pseudo Anosov mapping class which satisfies $\phi(\tilde q)=\Phi^\tau(\tilde q)$ 
where $\tau>0$ is the period of $\gamma$. 
The projective measured geodesic lamination 
$\zeta_1\in R$ is the attracting fixed point for the action of the
map  $\phi$ on the space of projective measured geodesic laminations of 
$S$. As $\phi$ preserves the component $\tilde {\cal C}$ of the
preimage of ${\cal C}$ containing $\tilde q$, 
by possibly replacing
$\phi$ by a large power we may assume that $\phi(D)$ is contained in 
the interior of $R$ and that the same holds true for $T\phi(D)$.
Recall to this end that $T\zeta_1$ is contained in the interior of $R$ 
by assumption. 
Similarly, we may assume that 
$\phi^{-1}T^{-1}K$ is contained in the interior of $K$.

To summarize, the mapping class $\psi=T\circ \phi$ maps
the compact ball $D$ into its interior. Therefore it has an 
attracting fixed point
in the interior of $D$. Similarly, it has a repelling fixed point in the
interior of $K$. As $\psi$ is pseudo-Anosov, it acts on the
space of projective measured geodesic laminations with 
north-south dynamics. As a consequence, the fixed points 
of $\psi$ in $D,K$, respectively, 
are the attracting and repelling measured geodesic laminations of 
$\psi$,  and the periodic orbit defined by 
$\psi$ is contained in ${\cal C}$. 

As this argument is valid for each $i\leq \ell$
and we can choose the pseudo-Anosov mapping class $\phi$
such that its attracting measured geodesic lamination
is arbitrarily close to $\mu$, 
we conclude that for each $i$ 
there is a pseudo-Anosov mapping class $\phi_i\in \Omega(\Gamma_0)$ such that
$T_i\phi\in \Omega(\Gamma_0)$ for all $i$.
Since $G$ is a group, we deduce that $T_i\in G$ for all $i$ and hence
$H<G$ as claimed.

The subgroup $\Psi(H)< 
Sp(2\ell,\mathbb{R})$ is an abelian group generated by $\ell$ 
transvections. The intersection of their fixed sets intersects 
${\cal Z}= p(T{\cal C}_+)\cap H^1(S,\mathbb{R})$
in a Lagrangian linear subspace $A_1$ of 
${\cal Z}$.

Recall that the bundle ${\cal Z}\vert U$ is equipped with a fixed
trivialization. Let again $\tilde U$ be a lift of $U$ to $\tilde {\cal C}$. 
For $\tilde q\in \tilde U$, the real part 
${\rm Re}(\tilde q)$ of $\tilde q$ is 
a harmonic one-form for the complex structure underlying $\tilde q$ 
which defines a cohomology class
$[{\rm Re}(\tilde q)]\in {\cal Z}\subset H^1(S,\mathbb{R})$. 
As $\tilde q$ varies in $\tilde U$ these
cohomology classes vary through an open subset of 
${\cal Z}\sim \mathbb{R}^{2\ell}$. 

Denote by $[c_i]\in H_1(S,\mathbb{Z})$ 
the homology class of the oriented weighted 
multicurve $c_i$ which determines the Dehn multi-twist $T_i$.
The evaluation of the real part of the marked 
abelian differential  $\tilde \omega$ on 
each of the classes $[c_i]$ is positive. As this is an open condition,
by decreasing the size of $U$ we may assume that
$\langle [{\rm Re}(\tilde q)],[c_i]\rangle\not=0$ for all $i$ 
and all $\tilde q\in \tilde U$ 
where $\langle , \rangle$ is the natural pairing
$H^1(S,\mathbb{R})\times H_1(S,\mathbb{R})\to \mathbb{R}$.  
Now periodic orbits for 
$\Phi^t$ are dense in ${\cal C}$ and therefore 
we can find a periodic point $z\in U$ with preimage
$\tilde z\in \tilde U$. By assumption on $\tilde U$ we have   
$\langle [{\rm Re}(\tilde z)],[c_i]\rangle\not=0$ for all $i$.

Let $\phi\in {\rm Mod}(S)$ be a
pseudo Anosov element whose cotangent line passes through 
$\tilde z$. There is a number $\kappa>1$ such that
$\phi^*{\rm Re}(\tilde z)=\kappa{\rm Re}(\tilde z)$, moreover
$\kappa$ is the Perron Frobenius eigenvalue
for the action of $\phi$ on $H_1(S,\mathbb{R})$. By invariance of
the natural pairing $\langle , \rangle$ 
under $\phi$, as $k\to \infty$ 
the homology classes $\phi^k([c_i])$ converge up to rescaling to 
a class $u\in H_1(S,\mathbb{R})$ whose contraction with the
intersection form defines $\pm[{\rm Re}(\tilde z)]$, viewed as a linear
functional on $H_1(S,\mathbb{R})$. By this we mean that 
$\iota (u,a)=\langle \pm[ {\rm Re}(\tilde z)], a\rangle$ for all 
$a\in H_1(S,\mathbb{R})$. 

As a consequence,
for sufficiently large $j$ and all $i,\ell$ we have 
$\iota([\phi^jc_i],[c_\ell])\not=0$. 
It now follows from the 
arguments in the beginning of this proof and equivariance 
that the group generated by $\Psi\Omega(\Gamma_0)$ contains
a subgroup generated by at least $\ell+1$ transvections with 
integral homology classes
which are independent over $\mathbb{R}$. These are the images under 
$\Psi$ of the Dehn multitwists $T_i=T_{c_i}$ $(i\leq \ell)$ and 
images under $\Psi$ of the Dehn multitwists
$\phi^jT_i\phi^{-j}=T_{\phi^j c_i}$. Moreover, $\iota([\phi^jc_i],[c_u])\not=0$ 
for all $i,u$.

Let $A_2\subset A_1$ be the common
fixed set in ${\cal Z}$ of the transvections which 
are the images of the multitwists $T_i,\phi^jT_u\phi^{-j}$.
Then $A_2$ is a linear subspace of $A_1$, and 
for large enough $j$ its
codimension in $A_1$ is $s\geq 1$. 
Let $i_1,\dots,i_s\subset \{1,\dots,\ell\}$ be such that the
homology classes 
$[c_i],[\phi^jc_{i_u}]\in H_1(S,\mathbb{Z})$ $(i\leq \ell,u\leq s)$ are 
independent over $\mathbb{R}$ and that the 
common fixed set 
in ${\cal Z}$ of the transvections defined by the corresponding 
Dehn multitwists is $A_2$.  
Using again the fact that the set of real
parts of differentials in $\tilde U$ 
define an open subset of 
$p(T{\cal C}_+)\cap H^1(S,\mathbb{R})$,  
we can find some $\bar z\in \tilde U$ and some 
$a\in A_2$ so that
$\langle [{\rm Re}(\bar z)],a\rangle>0$.
As  before, we may assume that $\bar z$ is the preimage of 
a periodic point. 
Argue now as in the previous paragraph and find a  
multitwist $\beta$ in the group generated by $\Omega(\Gamma_0)$ so that 
the common fixed set of the subgroup generated
by $\Psi(\beta)$ and $A_2$ has codimension at least one in $A_2$.

Repeat this construction. In at most $\ell$ steps we find 
integral homology classes 
$a_1,\dots,a_\ell,a_{\ell+1},\dots,a_{2\ell}\in H_1(S,\mathbb{Z})$
(where for $i\leq \ell$ the class $a_i$ is  
the class of the weighted multicurve 
which determines $T_i$)
with the following properties. 
\begin{enumerate}
\item Let $E\subset H_1(S,\mathbb{R})$ be the 
real vector space spanned by the classes $a_i$. 
Then the dimension of $E$ equals $2\ell$.
Each element $a\in E$ defines a linear functional on 
$H^1(S,\mathbb{R})$ by evaluation,  
and the restriction to ${\cal Z}$ 
of this linear subspace of $H^1(S,\mathbb{R})^*$
is non-degenerate. 
In particular, $E$ is a symplectic subspace of $H_1(S,\mathbb{R})$.
\item $\iota(a_{j},a_i)\not=0$ for all $i\leq \ell$, 
$j\geq \ell+1$.
\item   
For each $i$ the transvection $b\to b+\iota(b,a_i)a_i$ 
is contained in the group generated by $\Psi(\Omega(\Gamma_0))$. 
\end{enumerate}

By the choice of the homology classes $a_i$, 
the $(2\ell,2\ell)$-matrix $(\iota(a_i,a_j))$ 
whose $(i,j)$-entry is the intersection 
$\iota(a_i,a_j)$ is integral and of maximal rank. 
Choose a prime 
$p\geq 5$ so that each of the entries of $(\iota(a_i,a_j))$ 
is prime to $p$. 
All but finitely many primes will do. 
Then the reduction mod $p$ of the matrix 
$(\iota(a_i,a_j))$ is of maximal rank as well.
In particular, if $F_p$ denotes the field with $p$ elements then 
the reductions mod $p$ of the 
homology classes $a_i$ span a $2\ell$-dimensional
symplectic subspace $E_p$ of $H_1(S,F_p)$.

Let $L<{\rm Sp}(E)$ be the 
subgroup of the symplectic group of $E$ which 
is generated by the transvections with the elements $a_i$.
Its reduction $L_p$ mod $p$ acts on $E_p$ as a group of symplectic transformations. 
Lemma \ref{transvection} shows that 
$L_p=Sp(2\ell,F_p)$. 
Note that property (2) above guarantees that all conditions
in Lemma \ref{transvection} are fulfilled. 
Then $L$ is a Zariski dense subgroup of the group of symplectic 
automorphisms of $E$ \cite{Lu99}. 
By duality and the discussion in the proof of 
Proposition \ref{surjectivity}, this just implies that 
${\cal C}_+$ is locally Zariski dense.
\end{proof}




\begin{corollary}\label{localirreducible}
Let ${\cal C}$ be the hyperplane of area one differentials in 
an affine invariant manifold ${\cal C}_+$ of rank $\ell\geq 1$.
Then for every open subset $U$ of ${\cal C}$ there exists a periodic
orbit $\gamma$ for $\Phi^t$ through $U$
with the following properties.
\begin{enumerate}
\item The eigenvalues of the
matrix $A=\Psi(\Omega(\gamma))\vert {\cal Z}$ are real and pairwise distinct.
\item No product of two eigenvalues of $A$ is an eigenvalue.
\end{enumerate}
\end{corollary}
\begin{proof} By Proposition \ref{localzariski}, 
for every small open contractible subset
$U$ of ${\cal C}$, the image under $\Psi$ of the subsemigroup 
$\Psi(\Omega(\Gamma_0))$ defined as in Proposition \ref{grouplaw} by
the monodromy along 
periodic orbits through $U$ is Zariski dense in $Sp(2\ell,\mathbb{R})$.
The statement of the corollary is now an immediate consequence of the 
main result of \cite{Be97}.
\end{proof}

Recall that for an affine invariant manifold ${\cal C}$ of rank $\ell\leq g$ 
the projected
tangent space $p(T{\cal C})$ can be identified with 
the complexification of a symplectic subspace
$\mathbb{R}^{2\ell}$ of $\mathbb{R}^{2g}=H^1(S,\mathbb{R})$. 
The stabilizer of this
subspace is a subgroup $G$ of $Sp(2g,\mathbb{R})$ which is isomorphic to 
$Sp(2\ell,\mathbb{R})\times Sp(2(g-\ell),\mathbb{R})$. 

Let $\Pi:G\to G_1=Sp(2\ell,\mathbb{R})$ be the natural projection. 
Proposition \ref{localzariski} shows that $\Pi(G\cap Sp(2g,\mathbb{Z}))$
is a Zariski dense subgroup of $G_1$. 
The following consequence of this fact
was communicated to me by Yves Benoist. Although
it is not used for the proofs of the results stated in the
introduction, we include it here since it relates
affine invariant manifolds to proper subvarieties of 
${\cal A}_g$.

\begin{proposition}\label{latticeordense}
If $\Pi(G\cap Sp(2g,\mathbb{Z}))$ is Zariski dense in 
$Sp(2\ell,\mathbb{R})$ then 
either $\Pi(G\cap Sp(2g,\mathbb{Z}))$ 
is a lattice in $Sp(2\ell,\mathbb{R})$ or dense.
\end{proposition} 
\begin{proof} Using the above notations, 
write $G_{\mathbb{Z}}=Sp(2g,\mathbb{Z})\cap G$ and  
let $F<Sp(2\ell,\mathbb{R})$ be the Zariski closure
of $\Pi(G_{\mathbb{Z}})$.  

The group $F$ is defined over $\mathbb{Q}$. Namely, the set of polynomials
$P$ which vanish on $G_{\mathbb{Z}}$ is invariant under the Galois action.  
As a consequence, either $F_{\mathbb{Z}}=G_{\mathbb{Z}}$ is a lattice in 
$F$, or there is a nontrivial character on $F$ defined over $\mathbb{Q}$.

Assume for contradiction that 
there exists a nontrivial character on $F$ defined
over $\mathbb{Q}$. 
Define
\[F^0=\cap\{{\rm ker}(\chi)\mid 
\chi\text{ is a character on }F\text{ defined over }\mathbb{Q}\} .\] 
Then $F^0=F$ since up to multiplication with an integer, the evaluation 
on $G_{\mathbb{Z}}$ of 
a nontrivial character $\chi$ defined over $\mathbb{Q}$ has to be
integral in $\mathbb{C}^*$ which is impossible.  This contradiction 
yields that $F_{\mathbb{Z}}$ is a lattice in $F$.

The group $G_1=Sp(2\ell,\mathbb{R})$ is simple, 
and $\Delta=\Pi(G_{\mathbb{Z}})<G_1$
is Zariski dense.  Then $\Delta<G_1$ either is discrete or dense.
We have to show that if $\Delta$ is discrete then $\Delta$ is a lattice.

Thus assume that $\Delta$ is discrete. 
Consider the surjective homomorphism $\phi:F\to G_1$. 
Its kernel $K$ is a locally
compact group which intersects the lattice $F_{\mathbb{Z}}$ in a discrete subgroup.
The exact sequence
\[1\to K\to F\to G_1\to 1\]
induces a sequence 
\[K/K\cap F_{\mathbb{Z}}\to F/F_{\mathbb{Z}}\to G_1/\phi(F_{\mathbb{Z}}).\]
Now the Haar measure on $F$ can locally be represented as a product of 
the Haar measure on the orbits of $K$ and the quotient Haar measure.
If the volume of $G_1/\phi(F_{\mathbb{Z}})$ is infinite then this shows that 
the volume of $F/F_{\mathbb{Z}}$ has to be infinite.
But $F_{\mathbb{Z}}$ is a lattice in $F$ which is a contradiction.
\end{proof}

\section{Connections on the Hodge bundle}\label{connections}

The main goal of this section is to analyze differential geometric
properties of the projected tangent bundle of an affine invariant manifold
and use this is to establish a first rigidity result geared towards
Theorem \ref{finite} from the introduction. We always assume that $g\geq 3$.

Recall from Section \ref{local} that the Hodge bundle 
on the moduli space ${\cal M}_g$ of curves is 
the pull-back under the Torelli map of the Hermitean holomorphic
(orbifold) vector bundle 
${\cal V}\to {\cal A}_g$. 
The pull-back 
$\Pi^*{\cal H}\to {\cal H}_+$
to ${\cal H}_+$ of the Hodge bundle on ${\cal M}_g$ is
a holomorphic vector bundle on ${\cal H}_+$. 
The Hermitean structure on ${\cal H}$ defines 
a Hermitean structure on $\Pi^*{\cal H}$. 
The bundle $\Pi^*{\cal H}$ naturally splits as a direct
sum 
\[\Pi^*{\cal H}={\cal T}\oplus {\cal L}\]
of complex vector bundles. Here 
the fibre of ${\cal T}$ over a point $q$ is just the
$\mathbb{C}$-span of $q$, and the fibre of 
${\cal L}$ is the orthogonal complement
of ${\cal T}$ for the natural Hermitean metric, or, equivalently,
the orthogonal complement of ${\cal T}$ for the
symplectic form. 
The complex line bundle ${\cal T}$ is holomorphic. Via identification
of ${\cal L}$ with the quotient bundle $\Pi^*{\cal H}/{\cal T}$,
we may assume that ${\cal L}$ is holomorphic. 
Its complex dimension equals $g-1\geq 2$.

The group 
$GL^+(2,\mathbb{R})$ acts on $\Pi^*{\cal H}\to {\cal H}_+$ 
as a group of bundle automorphisms. For each point $q\in {\cal H}_+$,
the symplectic subspace of $H^1(S,\mathbb{R})$ spanned by the 
real and imaginary part of $q$ is locally constant along the
orbits of the $GL^+(2,\mathbb{R})$-action and hence
its symplectic complement is locally constant as well. 
Thus

\begin{lemma}\label{flat}
The restriction of the bundle ${\cal L}$ to the orbits of the 
$GL^+(2,\mathbb{R})$-action is flat. 
\end{lemma}

The space ${\cal H}_+$ is foliated into the orbits of the 
$GL^+(2,\mathbb{R})$-action. 
By Lemma \ref{flat}, the Gauss-Manin connection 
on the flat bundle ${\cal H}\to {\cal H}_+$ 
restricts to a flat leafwise 
connection $\nabla^{GM}$ on the bundle 
${\cal L}\to {\cal H}_+$. 
Here a leafwise connection
is a connection whose covariant derivative is only defined
for vectors tangent to the foliation.
As the Gauss Manin connection is real analytic 
in period coordinates, the $GL^+(2,\mathbb{R})$-action on ${\cal H}_+$ 
is real analytic and the splitting
$\Pi^*{\cal H}={\cal T}\oplus {\cal L}$
is real analytic, the leafwise connection is real analytic.
This means that it is defined by a connection matrix which 
is real analytic in real analytic coordinates. The leafwise 
connection preserves the symplectic structure of 
${\cal L}$, but there is no information on the
complex structure.

For each $k\leq g-2$, the leafwise
connection $\nabla^{GM}$ 
extends to a flat leafwise connection on the bundle
$\wedge^{2k}_{\mathbb{R}}{\cal L}$ whose fibre at $q$ 
is the $2k$-th exterior
power of the fibre of ${\cal L}$ at $q$, viewed as a real vector
space.

The Hermitean holomorphic vector bundle $\Pi^*{\cal H}$  
admits a unique \emph{Chern connection}
$\nabla$ (see e.g. \cite{GH78}). 
The Chern connection defines
parallel transport of the fibres of ${\cal H}$ along smooth 
curves in ${\cal H}_+$.
As the Chern connection is Hermitean, parallel transport
preserves the metric. 
Viewing again $\Pi^*{\cal H}$ as the real vector bundle 
$H^1(S,\mathbb{R})$ equipped with the complex structure $J$, the complex
structure $J$ is parallel for $\nabla$. In particular,
parallel transport preserves $J$ and the 
Hermitean metric.
Since the $GL^+(2,\mathbb{R})$-orbits on ${\cal H}_+$ are holomorphic
suborbifolds of ${\cal H}_+$ and the 
restriction of the bundle ${\cal T}$ to each
leaf of the foliation is just the tangent bundle of the 
quotient 
$GL^+(2,\mathbb{R})/(\mathbb{R}^+\times S^1)={\bf H}^2$, 
by naturality the restriction of the Chern connection to the leaves
of the foliation of ${\cal H}_+$ into the orbits of the 
$GL^+(2,\mathbb{R})$-action preserves the decomposition
$\Pi^*{\cal H}={\cal T}\oplus {\cal L}$.

For every $k\leq g-2$, the complex structure $J$ on ${\cal L}$ 
extends to an 
automorphism of the real tensor
bundle $\wedge^{2k}_{\mathbb{R}}{\cal L}$. The restriction of 
the connection
$\nabla$ to the orbits of the $GL^+(2,\mathbb{R})$-action
extends to a connection on $\wedge^{2k}_{\mathbb{R}}{\cal L}$ 
which commutes with this automorphism. 
%

The Hermitean metric which determines the Chern connection
is defined by the polarization and the complex structure.
These data are real analytic in period coordinates
(recall that the Torelli map is holomorphic) and consequently
the connection matrix  for the Chern connection in period coordinates
is real analytic (see \cite{GH78}).



To summarize, for every $k\leq g-2$, both 
the Chern connection and the Gauss Manin connection 
restrict to leafwise connections
of the restriction of the bundle 
$\wedge^{2k}_{\mathbb{R}}{\cal L}\to {\cal H}_+$ to 
the orbits of the $GL^+(2,\mathbb{R})$-action, 
and these leafwise connections depend in a real analytic
fashion on period coordinates.

The tangent bundle ${\cal F}$ of the foliation of 
${\cal H}_+$ into the orbits of the $GL^+(2,\mathbb{R})$-action is a
real analytic subbundle of $T{\cal H}_+$.
By the above discussion, 
$\nabla-\nabla^{GM}$ defines 
a real analytic tensor field
\[\Xi\in \Omega({\cal F}^*\otimes 
{\cal L}^*\otimes{\cal L})\] 
where we denote by $\Omega({\cal F}^*\otimes 
{\cal L}^*\otimes{\cal L})$
the vector space of real analytic
sections of the real analytic vector bundle
${\cal F}^*\otimes {\cal L}^*\otimes {\cal L}$. 
For every $k\leq g-2$ the tensor field 
\[\Xi^k\in \Omega({\cal F}^*\otimes  \wedge^{2k}_{\mathbb{R}}{\cal L}^*
\otimes \wedge^{2k}_{\mathbb{R}}{\cal L})\]
defined as the 
action of $\nabla-\nabla^{GM}$ on 
$\wedge^{2k}_{\mathbb{R}}{\cal L}$ is real analytic as well. 
If ${\cal C}_+\subset {\cal H}_+$ is an affine invariant manifold
of rank $2\leq \ell\leq g-1$, 
with absolute holomorphic tangent bundle ${\cal Z}$, then
the restriction of the tensor field
$\Xi^{\ell-1}$ to ${\cal C}_+$ preserves  the $J$-invariant 
section of $\wedge^{2\ell-2}_{\mathbb{R}}{\cal L}$ 
which is defined by $p(T{\cal C}_+)\cap H^1(S,\mathbb{R})$. 
This section associates to a point $q$ the exterior product
of a normalized oriented basis of the  
(real) vector space $p(T{\cal C}_+)\cap {\cal L}$. 

The next proposition is the key step towards Theorem \ref{finite}.

\begin{proposition}\label{realanalytic}
Let ${\cal C}_+\subset {\cal Q}_+$ be an affine
invariant manifold of rank $\ell\geq 3$. 
Then one of the following two possibilities holds true.
\begin{enumerate}
\item 
There are finitely
many proper affine invariant submanifolds
of ${\cal C}_+$  
which contain every affine invariant submanifold of ${\cal C}_+$
of rank $2\leq k\leq \ell-1$.
\item Up to passing to a finite cover, 
the restriction of the bundle ${\cal L}\cap {\cal Z}$ to an open dense
$GL^+(2,\mathbb{R})$-invariant subset of ${\cal C}_+$ 
admits a non-trivial 
$GL^+(2,\mathbb{R})$-invariant real analytic splitting
${\cal L}\cap {\cal Z}={\cal E}_1\oplus {\cal E}_2$ into two complex subbundles
whose restrictions to each orbit of the $GL^+(2,\mathbb{R})$-action are flat.
\end{enumerate}
\end{proposition}
\begin{proof} 
Let ${\cal Q}_+\subset {\cal H}_+$ be a component of a stratum and let 
${\cal C}_+\subset {\cal Q}_+$ be an affine invariant 
manifold of rank $\ell\geq 3$, with  absolute holomorphic tangent bundle  
${\cal Z}\to {\cal C}_+$.
As before, there is a splitting 
\[{\cal Z}={\cal T}\oplus ({\cal L}\cap {\cal Z}).\]
The bundle 
\[{\cal W}={\cal L}\cap {\cal Z}\to {\cal C}_+\] 
is holomorphic.  It also can be viewed 
as a real analytic real vector bundle with a real analytic complex structure
$J$ (which is just a real analytic
section of the real analytic endomorphism bundle of 
${\cal W}$ with $J^2=-{\rm Id}$).

For $1 \leq k\leq \ell-2$ denote 
by ${\rm Gr}(2k)\to {\cal C}_+$ the fibre bundle whose
fibre over $q$ is the Grassmannian of oriented
$2k$-dimensional real subspaces of ${\cal W}_q$. 
This is a real analytic fibre bundle with compact fibre.
It contains a real analytic subbundle 
${\cal P}(k)\to {\cal C}_+$  
whose fibre over $q$ is the 
Grassmannian of \emph{complex} $k$-dimensional
subspaces of ${\cal W}_q$ (for the complex structure $J$).

The Hermitean metric on ${\cal W}$ naturally 
extends to a real analytic Riemannian metric on 
$\wedge_{\mathbb{R}}^{2k}{\cal W}$.  
The bundle ${\rm Gr}(2k)$ can be identified with the
set of pure vectors in the sphere subbundle 
of $\wedge^{2k}_{\mathbb{R}}{\cal W}$ for this metric.
Namely, an oriented $2k$-dimensional
real linear subspace $E$ of ${\cal W}_q$ defines
uniquely a pure vector in $\wedge^{2k}_{\mathbb{R}}{\cal W}_q$ of norm one
which is just the exterior product of an orthonormal basis
of $E$ with respect to the inner product on ${\cal W}_q$.
The points in ${\cal P}(k)$ correspond precisely to those
pure vectors which are invariant under the 
extension of $J$ to an automorphism of $\wedge^{2k}_{\mathbb{R}}{\cal W}$. 

From now on, we work on the real analytic hyperplane
${\cal C}\subset {\cal C}_+$ of abelian differentials in 
${\cal C}$ of area one, and we replace the action of $GL^+(2,\mathbb{R})$
by the action of $SL(2,\mathbb{R})$. 
The tangent bundle ${\cal F}$ of the orbits of the
$SL(2,\mathbb{R})$-action is naturally trivialized
by the generator
$X$ of the Teichm\"uller flow, the generator $Y$ of the 
horocycle flow, and the generator $Z$ of the circle group of rotations.
Let $B^k$ (or $C^k,D^k$) be the contraction of the
tensor field $\Xi^k$ with the vector fields $X,Y,Z$. 
Since these vector fields are real analytic, 
$B^k$ (or $C^k,D^k$) is a real analytic section of the endomorphism bundle 
${\rm End}(\wedge^{2k}_{\mathbb{R}}({\cal W}))$ of 
$\wedge^{2k}_{\mathbb{R}}({\cal W})$.  

For $1\leq k\leq \ell-2$ and $q\in {\cal C}$ let 
\[\rho^{k}(q,0)\subset {\cal P}(k)_q\] be
the set of all $k$-dimensional 
complex linear subspaces $L$ of ${\cal W}_q$ with 
$B^{k}L=0=C^kL=D^kL$ (here we view as before a $k$-dimensional 
complex subspace of ${\cal W}_q$ as a pure $J$-invariant vector in 
$\wedge^{2k}_{\mathbb{R}}{\cal W}_q$). 
Since the contractions $B^k,C^k,D^k$ of the 
tensor field $\Xi^k$ are linear, the set $\rho^k(q,0)$ can be
identified with the
set of all $J$-invariant pure vectors which are contained
in some (perhaps trivial) linear subspace of $\wedge_{\mathbb{R}}^k({\cal W}_q)$.

%

Define a \emph{real analytic subset} of
${\cal P}(k)$ 
to be  
the intersection of the 
zero sets of a finite or countable 
number of real analytic functions on 
${\cal P}(k)$. 
Recall that
this is well defined since 
${\cal P}(k)$ has a natural real analytic structure.
We allow such functions to be constant zero, i.e. we do not exclude that
such a set coincides with
${\cal P}(k)$. 
The real analytic set is \emph{proper}
if it does not coincide 
with ${\cal P}(k)$. 
Then there is at least
one defining function which is not identically zero and hence
the set is closed and nowhere dense.

Since the tensor fields $\Xi^k$ and the vector fields
$X,Y,Z$ are real analytic, 
$\cup_q \rho^k(q,0)$ is a 
real analytic subset of ${\cal P}(k)$, defined as the common
zero set of three real analytic functions. 


For $t\in \mathbb{R}$ define 
\[\rho_1^k(q,t)\subset {\rm Gr}(2k)_q\] 
to be the preimage 
of $\rho^k(u_tq,0)$ under
parallel transport for the Gauss Manin connection along the flow line
$s\to u_sq$ $(s\in [0,t])$ of the horocycle flow. 
By the previous paragraph and the
fact that parallel transport is real analytic, 
$\cup_q\rho_1^k(q,t)$ is a real analytic subset of 
${\rm Gr}(2k)_q$ and hence the same holds true for   
\[\xi_1^k=\cap_{t\in \mathbb{R}}\cup_q\rho_1^k(q,t)\subset {\cal P}(k)\]
(take the intersections for all $t\in \mathbb{Q}$).

By construction, the set $\xi_1^k$ is invariant under the extension 
of the horocycle flow 
by parallel transport with respect to the 
Gauss Manin connection of the fibres of the bundle 
${\cal P}(k)\to {\cal C}$. It also is invariant
under parallel transport with respect to the Chern connection:
Namely, by definition, if $Z\in \xi_1^k(q)$ and 
if $Z(t)$ is the parallel transport of $Z=Z(0)$ for the Gauss Manin
connection along the 
orbit of the horocycle flow through $q$, then the covariant derivative
of the section $t\to Z(t)$ for the Chern connection vanishes.

Similarly, for $t\in \mathbb{R}$ 
define 
\[\rho_2^k(q,t)\subset {\rm Gr}(2k)_q\]
to be the preimage of $\xi_1^k(\Phi^t q)$ under parallel transport
for the Gauss Manin connection along the flow line $s\to \Phi^s q$
$(s\in [0,t])$ of the Teichm\"uller flow and let 
\[\xi_2^k=\cap_{t\in \mathbb{R}}\cup_q \rho_2^k(q,t)\subset {\cal P}(k).\]
Then $\xi^k_2$ is invariant under the extension of the Teichm\"uller flow 
by parallel transport both for the Gauss Manin connection and the 
Chern connection. 
As the Teichm\"uller flow maps orbits of the horocycle
flow to orbits of the horocycle flow with the parametrization
multiplied by a constant, using Lemma \ref{flat} it follows that 
the set $\xi_2^k$ consists
of points whose parallel transports along the orbits of the
upper triangular subgroup of $SL(2,\mathbb{R})$ for both the
Chern connection and the Gauss Manin connection coincide.

Repeat this construction with the circle group of rotations to find a 
set 
\[\xi^k\subset {\cal P}(k).\]
Then $\xi^k$ is a real analytic subset of ${\cal P}(k)$ which is
invariant under the action of $SL(2,\mathbb{R})$ defined by 
parallel transport for the Gauss Manin connection.

If ${\cal U}\subset {\cal C}$ is a proper affine invariant manifold
of rank $2\leq k+1<\ell$ 
then it follows from the discussion preceding this
proof (see \cite{F16}) that for every $q\in {\cal U}$
the projected tangent space $p(T{\cal U})$ 
defines a point in $\xi^k(q)\subset {\cal P}(k)_q$.
In particular, we have $\xi^k\not=\emptyset$.

Let $\pi:{\cal P}(k)\to {\cal C}$ 
be the natural projection and
let 
\[{\cal A}=\pi(\xi^k).\] 
The fibres of $\pi$ are compact and hence $\pi$ is closed.
Therefore  
${\cal A}$ is 
a closed $SL(2,\mathbb{R})$-invariant 
subset of ${\cal C}$ which contains all affine invariant submanifolds of 
${\cal C}$ of rank $k+1$. 
Our goal is to show that either ${\cal A}$ is nowhere dense in 
${\cal C}$, or the second property in the statement of the proposition 
is fullfilled.

To this end 
assume that ${\cal A}$ contains an
open subset of ${\cal C}$. By invariance and 
ergodicity, ${\cal A}$ contains an open dense invariant set.
On the other hand, ${\cal A}$ is closed and hence
we have ${\cal A}={\cal C}$. This is equivalent to stating
that $\xi^k(q)\subset {\cal P}(k)_q$ is a non-empty compact
set for all $q\in {\cal C}$. In other words,
for every $q\in {\cal C}$ there is a line in 
$\wedge^{2k}_{\mathbb{R}}{\cal W}$ spanned by
a pure vector $Z$ 
which is an eigenvector for the extension of 
the complex structure $J$ and which is contained in the kernel of 
$B^k,C^k,D^k$. Moreover, 
the same holds true for the
parallel transport with respect to the Gauss Manin connection 
of this eigenline along the orbits of the $SL(2,\mathbb{R})$-action
(by Lemma \ref{flat}, this makes sense).

With respect to a real analytic local trivialization of the bundle
${\cal P}(k)$ over an open set $V\subset {\cal C}$, the set
$\xi^k$ is of the form $(q,\xi^k(q))$ where $\xi^k(q)$ is 
a real analytic subset of the compact 
projective variety of $k$-dimensional
complex linear subspaces of $\mathbb{C}^{\ell-1}$ depending in
a real analytic fashion on $q$. 
More precisely, $\xi^k(q)$ can be identfied with the
space of all $J$-invariant
pure vectors which are contained in some linear subspace $R_q$ of 
$\wedge_{\mathbb{R}}^{2k}({\cal W}_q)$.  
As a consequence, 
there is an open subset $V$ of ${\cal C}$ so that 
the restriction of $\xi^k$ to $V$ is a fibre bundle over $V$
(take an open set where the dimension of $R_q$ is minimal).
By invariance and ergodicity, 
there is an open dense $SL(2,\mathbb{R})$-invariant connected subset
$V$ of ${\cal C}$ such that $\xi^k\vert V$ is a fibre bundle.

We claim that for $q\in V$, the set
$\xi^k(q)$ consists of finitely many points. 
To this end choose a periodic orbit $\gamma\subset V$ for the
Teichm\"uller flow 
so that for $q\in \gamma$, 
the restriction $A$ to $\mathbb{R}^{2\ell}={\cal Z}_q$ of the transformation
$\Psi(\Omega(\gamma))\in Sp(2\ell,\mathbb{R})$ 
has $2\ell$ distinct real eigenvalues. Such an orbit exists by Corollary \ref{localirreducible}.
Note that $A$ is the return map 
for parallel transport of ${\cal Z}$ along $\gamma$ 
with respect to the Gauss Manin
connection. The map $A$ preserves the decomposition 
${\cal Z}_q={\cal T}_q\oplus {\cal W}_q$. 

We are looking for $k$-dimensional complex 
linear subspaces of ${\cal W}_q$ whose
images under iteration of $A$ are complex. 
Recall that the complex structure and the symplectic structure
define an inner product $\langle, \rangle$ on $\wedge^{2k}{\cal W}_q$. 
Thus the set $\xi^k(q)$ consists of pure vectors 
$Y\in \wedge^{2k}_{\mathbb{R}}{\cal W}_q$ whose norms are
preserved by the symplectic linear map $A$. 
As the eigenvalues of $A$ are all real
and positive, of multiplicity one, 
the eigenvalues for the action of $A$ on $\wedge^{2k}{\cal W}_q$ are
all real and positive. More
precisely, the eigenvalues for this action are precisely the
products of $2k$ eigenvalues of the linear map $A$.
In particular, these eigenvalues are all real and positive, and 
the action of $A$ on $\wedge^{2k}_{\mathbb{R}}{\cal W}_q$ is diagonalizable
over $\mathbb{R}$. As a consequence, a vector 
$Y\in \wedge^{2k}_{\mathbb{R}}{\cal W}_q$ whose norm is preserved by
the action of $A$ is an eigenvector for the eigenvalue one.

We showed so far that a pure vector $Y\in \xi^k(q)$ corresponds to
a fixed point for
the action of $A$ on the 
set ${\cal P}(k)_q$ of all $k$-dimensional
complex subspaces of ${\cal W}_q$, viewed as 
a compact subset of the Grassmannian ${\rm Gr}(2k)_q$
of $2k$-dimensional subspaces of ${\cal W}_q$. 
The fixed points 
for the action of $A$ on ${\rm Gr}(2k)_q$ 
are precisely the $k$-dimensional oriented linear subspaces
which are direct sums of eigenspaces of $A$. Thus
the number of such subspaces is finite and hence
$\xi^k(q)$ is a finite set.

By the choice of $V$, the cardinality of 
the set $\xi^k(q)$ is locally constant 
and hence constant on $V$ since $V$ is connected. 
As the dependence of $\xi^k(q)$ on $q\in V$ is real analytic, 
any choice of a point in 
$\xi^k(q)$ defines locally near $q$ an analytic section of
${\cal P}(k)$ and hence a 
real analytic $J$-invariant subbundle of 
${\cal W}$.  If this local 
section is invariant under parallel transport 
along the orbits of the $SL(2,\mathbb{R})$-action
(this is equivalent to triviality of the 
monodromy),  then 
it defines a real analytic $J$-invariant 
subbundle of ${\cal W}\vert V$.
Otherwise as $\xi^k(q)\subset {\cal P}(k)_q$ 
consists of finitely many points and 
is invariant under parallel transport along the orbits
of the $SL(2,\mathbb{R})$-action, 
this parallel transport acts as a finite
group of permutations on the finite set $\xi^k(q)$. 
Thus we can pass to a finite
cover of ${\cal C}$ 
so that on the covering space, using the same notation, the induced subbundle
of ${\cal W}$ 
is invariant under parallel transport. 

In other words, up to passing
to a finite cover,  
$\xi^k$ defines a real analytic complex $SL(2,\mathbb{R})$-invariant 
$k$-dimensional vector bundle over the open dense invariant subset 
$V$ of ${\cal C}$. By invariance of both the 
complex and the symplectic structure under parallel transport 
along the orbits of the $SL(2,\mathbb{R})$-action, $\xi^k$ then defines 
a splitting of the bundle ${\cal L}\vert V$ as predicted 
in the second part of the proposition.
\end{proof}

\begin{remark}\label{nocontradiction}
By Proposition \ref{localzariski} (see also \cite{W14}), 
a real analytic splitting of the bundle ${\cal L}$ as stated in the second
part of Proposition \ref{realanalytic} can not be flat, i.e. invariant
under the Gauss Manin connection. However, 
Proposition \ref{localzariski} does not rule out a real analytic 
splitting which is not flat. 
\end{remark}

\section{Invariant splittings of the lifted Hodge bundle}\label{uniqueness}

In this section we use information on the moduli space of 
principally polarized abelian varieties 
to rule out the second case in Proposition \ref{realanalytic}.
We continue to use all assumptions and notations from 
Section \ref{connections}.


Recall the splitting ${\cal H}={\cal T}\oplus {\cal L}$ of the lifted
Hodge bundle. The leaves of the absolute period foliation
consist of differentials whose real and imaginary part define fixed
classes in $H^1(S,\mathbb{R})$. Thus 
by the fact that the Hodge bundle is flat, the restriction of the
bundle ${\cal T}$ to each leaf of the absolute period foliation is
flat and hence the same holds true for its symplectic complement ${\cal L}$.  
This is equivalent to stating that parallel transport of 
the Hodge bundle with respect to the Gauss Manin connection 
along curves contained in the absolute
period foliation preserves the subbundle ${\cal L}$. 
Equivalently, the restriction of the Gauss Manin connection
to each such leaf defines a leafwise flat
connection on ${\cal L}$. 

Let ${\cal C}$ be an affine invariant manifold with absolute holomorphic 
tangent bundle ${\cal Z}$. 
Assume that there is an open dense $SL(2,\mathbb{R})$-invariant 
subset $V$ of ${\cal C}$ and there 
is an $SL(2,\mathbb{R})$-invariant real analytic splitting 
${\cal L}\cap {\cal Z}\vert V={\cal E}_1\oplus {\cal E}_2$ 
into complex orthogonal subbundles as
in the second part of Proposition \ref{realanalytic}. The restriction of 
the splitting to an orbit of the $SL(2,\mathbb{R})$-action is 
invariant under the Gauss Manin connection. 
Equivalently, the restrictions of the bundles
${\cal E}_i$ to the orbits of the $SL(2,\mathbb{R})$-action 
are parallel. 

Recall from Section \ref{localaffine} that if 
${\rm dim}({\cal C}_+)>2{\rm rk}({\cal C}_+)$ then the 
absolute period foliation ${\cal A\cal P}({\cal C})$ 
of ${\cal C}$ is defined.

\begin{lemma}\label{flatone} 
The restriction of the bundle ${\cal E}_i\to V$ 
to a leaf of ${\cal A\cal P}({\cal C})$ is 
invariant under the Gauss Manin connection. 
\end{lemma}
\begin{proof} We may assume that the dimension of ${\cal A\cal P}({\cal C})$ is positive.
Furthermore, by passing to a finite cover we assume that the zeros of the 
differentials in 
${\cal C}$ are numbered.

The splitting ${\cal Z}\vert V={\cal T}\oplus {\cal E}_1\oplus {\cal E}_2$ 
can be used 
to project the Gauss Manin connection $\nabla^{GM}$ on $p(T{\cal C})={\cal Z}$ 
to a real analytic connection $\hat \nabla$ along the leaves of the absolute
period foliation which preserves this decomposition. 
More precisely, given a local vector field $X\subset T{\cal A\cal P}({\cal C})$ and
a section $Y$ of ${\cal E}_i$, define
\[\hat \nabla_XY=P_i(\nabla^{GM}_XY)\] where 
\[P_i:{\cal E}_1\oplus {\cal E}_2\to {\cal E}_i\] 
is the natural projection. Recall that this makes sense since
the Gauss Manin connection restricted to a leaf of 
${\cal A\cal P}({\cal C})$ preserves the bundle ${\cal L}\cap {\cal Z}$.   

Let $\mathfrak{a}\in \mathbb{C}^k$ be such that the vector field 
$X_{\mathfrak{a}}$ is tangent to ${\cal C}$. The existence of such a
vector $\mathfrak{a}$ was shown in Lemma \ref{affine}. 
For every 
$q\in V\subset {\cal C}$ and 
every 
$Y\in {\cal E}_1(q)$ we 
can extend $Y$ by parallel transport for $\hat \nabla$ along the flow line
of the flow $\Lambda_{\mathfrak{a}}^t$ generated by
$X_{\mathfrak{a}}$. Let us denote this extension by $\hat Y$; then 
\[\beta(X_{\mathfrak{a}},Y)=\frac{\nabla^{GM}}{dt}\hat Y(\Lambda_{\mathfrak{a}}^t(q))\vert_{t=0}
\in {\cal E}_2(q)\]
only depends on $X_{\mathfrak{a}}$ and $Y$, moreover this dependence is linear
in both variables. 
In this way we obtain a real analytic tensor field 
\[\beta\in \Omega(T{\cal A\cal P}({\cal C})^*
\otimes {\cal E}_1^*\otimes {\cal E}_2).\]
Here as before, $\Omega(T{\cal A\cal P}({\cal C})^*\otimes 
{\cal E}^*_1\otimes {\cal E}_2)$ is the vector space of real analytic sections
of the bundle $T{\cal A\cal P}({\cal C})^*\otimes 
{\cal E}_1^*\otimes {\cal E}_2$.
The splitting ${\cal L}\cap {\cal Z}={\cal E}_1\oplus {\cal E}_2$ is parallel
along the leaves of the absolute period foliation 
if and only if $\beta$ vanishes
identically. 

The Teichm\"uller flow $\Phi^t$ acts on the bundle ${\cal L}$ by parallel transport 
with respect to the Gauss Manin connection, and this action preserves
the bundles ${\cal E}_i$ $(i=1,2)$ and the absolute period foliation of 
${\cal C}$. 
Thus by definition, the tensor field $\beta$ is
equivariant under the action of $\Phi^t$.

Assume to the contrary that $\beta$ does not vanish identically.
As $\beta$ is real analytic and bilinear,
there is then an open subset $U$ of $V\subset {\cal C}$ and either a real 
or a purely imaginary vector $\mathfrak{a}\in 
\mathbb{C}^k$ such that $X_{\mathfrak{a}}$ is tangent to 
${\cal C}$ and that moreover 
the contraction of 
$\beta$ with $X_{\mathfrak{a}}$ does not vanish on $U$.

Assume that $\mathfrak{a}\in \mathbb{R}^k$ 
is real, the case of a purely
imaginary vector is treated in the same way; then
$d\Phi^tX_{\mathfrak{a}}=e^tX_{\mathfrak{a}}$ by
Lemma \ref{invariance3}. 
Let now $\gamma\subset {\cal C}$ be a periodic orbit
with the properties stated in Corollary \ref{localirreducible} 
which passes through $U$. 
Let $q\in \gamma\cap U$. 
The eigenvalues of the matrix 
$A=\Psi(\Omega(\gamma))\vert {\cal Z}_q$ 
(where we identify ${\cal Z}_q$ with the symplectic subspace
of $H^1(S,\mathbb{R})$ it defines) 
are 
real, positive and of multiplicity one. 
The largest eigenvalue of $A$ equals 
$e^{\ell(\gamma)}$ where
$\ell(\gamma)$ is the length of the orbit $\gamma$.

Write ${\cal W}={\cal Z}\cap {\cal L}$; then 
${\cal W}_q$ is a sum of eigenspaces for $A$ 
(viewed as a transformation of ${\cal Z}_q$) for eigenvalues
strictly smaller than $e^{\ell(\gamma)}$. Furthermore,
the decomposition
${\cal W}_q={\cal E}_1(q)\oplus {\cal E}_2(q)$ $(i=1,2)$ 
is invariant under $A$, and ${\cal E}_i(q)$ is a direct
sum of eigenspaces for $A$.

Denote by $\vert \vert^{GM}$ parallel transport for the Gauss Manin connection.
By equivariance of the tensor field $\beta$ under the
flow $\Phi^t$, for $Y\in {\cal E}_1(q)$ we have
\[\beta(d\Phi^{\ell(\gamma)}X_{\mathfrak{a}},\vert \vert^{GM}_{\gamma}Y)=
\vert\vert^{GM}_\gamma\beta(X_{\mathfrak{a}},Y)
\in {\cal E}_2(q).\]

Recall that ${\cal E}_i(q)$ is a direct sum of eigenspaces for $A$.
By the above discussion,
if $Y\in {\cal E}_1(p)$ is an eigenvector of $A$ for the eigenvalue 
$a>0$, then 
\[\beta(d\Phi^{\ell(\gamma)}X_{\mathfrak{a}},AY)=
e^{\ell(\gamma)}a\beta(X_{\mathfrak{a}},Y)=A\beta(X_{\mathfrak{a}},Y)
\in {\cal E}_2(p).\]
In other words, the contraction $B$ of $\beta$ with $X_{\mathfrak{a}}$ 
maps an eigenspace of ${\cal E}_1(p)$ for the eigenvalue $a$ to an eigenspace
for the eigenvalue $e^{\ell(\gamma)}a$.

But $e^{\ell(\gamma)}$ is an eigenvalue of $A$ and by the choice of 
$\gamma$, no product of two eigenvalues of $A$ is an eigenvalue. Therefore
$B$ has to be the zero map. 
This contradicts the assumption that the contraction of 
$\beta$ with $X_{\mathfrak{a}}$ does not vanish at $q$ and hence
that the map $B$ is  
non-trivial.

As a consequence, the parallel transport for $\hat \nabla$
of a vector $Y\in {\cal E}_1$  
along a path which is entirely contained
in a leaf of the absolute period foliation of $V\subset {\cal C}$ 
coincides with parallel transport with respect to the
Gauss Manin connection. Equivalently, the restriction of 
the Gauss Manin connection to a leaf of the absolute
period foliation preserves the splitting 
${\cal W}={\cal E}_1\oplus {\cal E}_2$.
This is what we wanted to show.
\end{proof}

\begin{remark}\label{sufficient?}
It is an interesting question whether it is possible to 
deduce from Lemma \ref{flatone}, Proposition
\ref{latticeordense} and Moore's theorem
for the action of $SL(2,\mathbb{R})$ on 
$Sp(2g,\mathbb{Z})\backslash Sp(2g,\mathbb{R})$ that 
a splitting
as in the second part of Proposition \ref{localzariski} does not
exist.  The main difficulty is that we do not know whether 
there is  a leaf of the foliation ${\cal F}$ of 
the bundle ${\cal S}$ as described in the appendix 
which intersects the image of the period map in more than one component.
\end{remark}

We saw so far that an invariant splitting of the Hodge bundle
over an affine invariant manifold ${\cal C}$ 
as predicted by the second part of Proposition \ref{realanalytic}
has to be flat along the leaves of the absolute period foliation.
If this foliation is non-trivial and if we knew that it has dense 
leaves (this is an open question, see however \cite{H15}), 
we could try to show that such a splitting has to be flat globally.
As such tools are not available, we use the curvature of the 
projection of the Gauss Manin connection to the invariant 
subbundle to arrive at a contradiction similar to the 
reasoning in the proof of Lemma \ref{flatone}.
To compute this curvature we use the fact that
the Gauss Manin connection is the pull-back of the standard flat connection
of the vector bundle ${\cal V}\to {\cal A}_g$ under the Torelli map.
As detailed in the appendix, this connection can be controlled with
standard methods from differential geometry on locally symmetric spaces.

\begin{proposition}\label{nosplit}
Let ${\cal Z}$ be the absolute holomorphic tangent bundle 
of an affine invariant manifold ${\cal C}$ of rank at least three.
Then there is no open dense 
$SL(2,\mathbb{R})$-invariant subset $V$ of ${\cal C}$ such that 
the bundle ${\cal W}\vert V={\cal Z}\cap {\cal L}\vert V$ admits an
$SL(2,\mathbb{R})$-invariant real analytic direct decomposition
into a sum of two complex vector bundles whose restrictions to
each $SL(2,\mathbb{R})$-orbit are flat. 
\end{proposition}
\begin{proof}
We argue by contradiction and we assume that an open dense
invariant set $V$ and such a splitting
${\cal W}\vert V={\cal E}_1\oplus {\cal E}_2$ exists.
By Lemma \ref{flatone}, we know that this splitting is 
invariant under the restriction of the Gauss Manin connection to 
the leaves of the absolute period foliation of ${\cal C}$. 
Furthermore, it naturally induces an invariant splitting of the 
bundle ${\cal W}\oplus \overline{\cal W}\subset p(T{\cal C})$
into two subbundles which are complex for the flat complex structure on 
$H^1(S,\mathbb{C})=H^1(S,\mathbb{R})\otimes \mathbb{C}$. 

Let ${\cal T}(S)$ be the Teichm\"uller space of the surface $S$ and let
${\cal I}_g<{\rm Mod}(S)$ be the Torelli group. 
The group ${\cal I}_g$ acts properly and freely 
from the right on ${\cal T}(S)$, 
with quotient the \emph{Torelli space} ${\cal T}(S)/{\cal I}_g$.
The \emph{period map} $F$ maps the bundle ${\cal D}$ 
of area one abelian differential
over ${\cal T}(S)/{\cal I}_g$ into the sphere subbundle ${\cal S}$ of 
the tautological vector bundle ${\cal V}$ over the Siegel upper half-space
$\mathfrak{D}_g$ (see the beginning of Section \ref{local} and the  
appendix for the notations). 
The map $F:{\cal D}\to {\cal S}$ is equivariant with respect to 
the standard action of $SL(2,\mathbb{R})$ on ${\cal D}$ and the 
right action of $SL(2,\mathbb{R})$ on ${\cal S}$
(which commutes with the left action of $Sp(2g,\mathbb{R})$).
We refer to the appendix for more details on these facts. 

By Lemma \ref{invariance4}, 
the composition of the map $F$ with the equivariant projection 
\[\Pi:{\cal S}=Sp(2g,\mathbb{R})\times_{U(g)}S^{2g-1}\to \Omega=
Sp(2g,\mathbb{R})/Sp(2g-2,\mathbb{R})\] 
is equivariant for the standard $SL(2,\mathbb{R})$-actions
as well. 

Apply the map $\Pi\circ F$ 
to a component $\hat {\cal C}$ of the preimage of ${\cal C}$ in 
${\cal D}$. 
More precisely, the absolute holomorphic tangent bundle of 
${\cal C}$ determines a symplectic subspace 
$\mathbb{R}^{2\ell}$ of $(\mathbb{R}^{2g},\omega)=H^1(S,\mathbb{R})$. 
We denote by $\mathbb{C}^{2\ell}$ its complexification. 
The composition of the map $\Pi\circ F$ with symplectic 
orthogonal projection 
then defines a map 
\[\Xi:\hat{\cal C}\to 
\Omega_\ell=\{x+iy\in \mathbb{C}^{2\ell}\mid \omega(x,y)=1\}.\]

By naturality of this construction, the 
Gauss Manin connection 
on the bundle $p(T{\cal C})$ with fibre ${\cal Z}\oplus
\overline{\cal Z}$ is just the pull-back via $\Xi$ of the
natural flat connection on $T{\cal O}$. 
Using the notations from the appendix, 
the leafwise connection $\nabla^{GM}$ on the bundle ${\cal L}$ is the
pull-back of the connection $\nabla^{\cal R}$ on the 
symplectic complement 
${\cal R}\subset T\Omega_\ell$ of the tangent bundle of the
orbits of the $SL(2,\mathbb{R})$-action.

By Lemma \ref{flatone}, 
the splitting ${\cal W}={\cal E}_1\oplus {\cal E}_2$ is invariant
under the $SL(2,\mathbb{R})$-action and parallel with respect to 
the restriction of the Gauss Manin connection to the leaves of the
absolute period foliation.
Using again Lemma \ref{invariance4} in the appendix, 
this means that the splitting ${\cal W}={\cal E}_1\oplus {\cal E}_2$
is the pull-back by $\Xi$ of a real analytic 
splitting ${\cal R}={\cal R}_1\oplus {\cal R}_2$ 
of the bundle ${\cal R}$ into a sum 
of two complex vector bundles.

We claim that the curvature form $\Theta$ for the connection
$\nabla^{\cal R}$ preserves this decomposition on the image 
of the map $\Xi$. To this end let
$\gamma\subset {\cal C}$ be a periodic orbit for $\Phi^t$ with the properties
as in Corollary \ref{localirreducible}.
Then $\Xi(\gamma)$ is an orbit in $\Omega_\ell$ for the action of the
diagonal subgroup of $SL(2,\mathbb{R})$. This orbit is  periodic 
under the action of an element $A\in Sp(2\ell,\mathbb{R})$ 
(which is the restriction of an element of $Sp(2g,\mathbb{Z})$ stabilizing 
the subspace $\mathbb{R}^{2\ell}$)
whose eigenvalues   
are all real, of multiplicity one, 
and such that no product of two eigenvalues is an eigenvalue. 

By the assumption of the splitting of ${\cal W}$ and 
Lemma \ref{invariance4}, 
the local splitting ${\cal R}={\cal R}_1\oplus {\cal R}_2$ 
is invariant under the action of $SL(2,\mathbb{R})$. 
As a consequence, 
the complex subspaces ${\cal R}_i$ are
sums of eigenspaces for $A$, containing with an eigenspace for the eigenvalue 
$a>1$ the eigenspace for $a^{-1}$. As in the proof of
Lemma \ref{flatone}, we conclude
that the connection $\nabla^{\cal R}$  on ${\cal R}$ preserves the splitting
${\cal R}={\cal R}_1\oplus {\cal R}_2$ along $\Xi(\gamma)$.

Namely, let $\nabla^{{\cal R}_1}$ be the projection of the connection 
$\nabla^{\cal R}$ to a connection on ${\cal R}_1$. Then $\nabla^{\cal R}-\nabla^{{\cal R}_1}$
is a real analytic tensor field $\beta\in \Omega(T^*\Omega_\ell\otimes
{\cal R}_1^*\otimes {\cal R}_2)$. Since ${\cal R}_1$ and $\nabla^{\cal R}$ are invariant
under the action of the diagonal flow $\Psi^t\subset SL(2,\mathbb{R})$
(we use the notation $\Psi^t$ here to indicate that we are looking at a flow on the 
space $\Omega_\ell$), 
this tensor field is equivariant under the action of $\Psi^t$. As in the proof of 
Lemma \ref{flatone}, as no product of two eigenvalues of the matrix $A$ is 
an eigenvalue, this implies that the restriction of 
$\beta$ to $\Xi(\gamma)$ 
vanishes.

By Corollary \ref{localirreducible}, 
the set of points $q\in V\subset {\cal C}$ which are
contained in a periodic orbit with the above properties 
is dense in $V$.  Hence its image under the restriction of 
the map $\Xi$ 
to a small contractible open subset of $V$ is a dense
subset of a nonempty open subset $E$ of $\Omega_\ell$ where the splitting
${\cal R}={\cal R}_1\oplus {\cal R}_2$ is defined. As a consequence, the 
real analytic tensor
field $\beta$ vanishes on $E$ and hence 
the splitting ${\cal R}={\cal R}_1\oplus {\cal R}_2$ is flat on $E$. 

Using the terminology 
from the appendix, this shows that 
the curvature form of the connection $\nabla^{\cal R}$ 
splits ${\cal R}$ as a complex vector bundle contradicting
Lemma \ref{split}.
This shows the proposition.
We refer to 
Theorem 5.1 of \cite{W14} for a similar result using related arguments.
\end{proof}

\begin{remark}\label{nosplitremark}
The reasoning in the proof of Lemma \ref{flatone} and Proposition \ref{nosplit}
also implies that the Lyapunov filtration for the action of the Teichm\"uller flow
on a stratum of abelian differentials is not smooth (or, less restrictive, is not
of the class $C^2$). As we use covariant differentiation in our argument, 
mere continuity of the filtration can not be ruled out immediately. 
\end{remark}

\begin{corollary}\label{smalleststratum}
\begin{enumerate}
\item Let ${\cal Q}$ be a component of a stratum; then 
for every $2\leq \ell\leq g-1$ there are finitely many 
affine invariant submanifolds of ${\cal Q}$  of rank
$\ell$ which contain every affine invariant 
submanifold of rank $\ell$.
\item The smallest stratum of differentials
with a single zero contains only 
finitely many affine invariant submanifolds
of rank at least two.
\end{enumerate}
\end{corollary}
\begin{proof} Let ${\cal C}$ 
be an affine invariant manifold of rank $\ell\geq 3$. 
By Proposition \ref{realanalytic} and Proposition \ref{nosplit},
there are finitely many proper
affine invariant submanifolds of ${\cal C}$ 
which contain every affine invariant
submanifold of ${\cal C}$ of rank $2\leq k\leq \ell-1$. 

An application of this fact to a component 
${\cal Q}$ of a stratum shows that for $2\leq \ell\leq g-1$,  
there are finitely many
proper affine invariant submanifolds 
${\cal C}_1,\dots,{\cal C}_m$ of ${\cal Q}$ 
which contain every affine
invariant submanifold of ${\cal Q}$ of rank $\ell$.
The dimension of ${\cal C}_i$ is strictly smaller than the dimension 
of ${\cal Q}$.

By reordering we may assume that there is some $u\leq m$ such that
for all $i\leq u$ the rank ${\rm rk}({\cal C}_i)$ of ${\cal C}_i$ is bigger than $\ell$ 
and that for $i>u$ the rank ${\rm rk}({\cal C}_i)$ 
of ${\cal C}_i$ is at most $\ell$.
Apply the first paragraph of this proof 
to each of the affine invariant
manifolds ${\cal C}_i$ $(i\leq u)$. 
We conclude that for each $i$ there are finitely many proper affine
invariant submanifolds of ${\cal C}_i$ 
of rank $r\in [\ell, {\rm rk}({\cal C}_i))$
which contain every affine invariant submanifold 
of ${\cal C}_i$ of rank $\ell$. The dimension of each of these
submanifolds is strictly smaller than the dimension of ${\cal C}_i$. 
In finitely many such steps, each applied to all affine
invariant submanifolds of rank stricltly bigger than $\ell$ 
found in the previous step, we
deduce the statement of the first part of the corollary.

Now let ${\cal H}(2g-2)$ be a stratum of differentials
with a single zero. 
Period coordinates for ${\cal H}(2g-2)$ are given
by absolute periods, and the dimension of 
an affine invariant manifold ${\cal C}\subset {\cal H}(2g-2)$ 
of rank $\ell$ equals $2\ell$.  Thus  
${\cal C}$
does not contain any proper affine invariant submanifold of
rank $\ell$. 

By Proposition \ref{realanalytic} and the first part of this proof, 
there are finitely many proper affine 
invariant submanifolds ${\cal C}_1,\dots,{\cal C}_s$
of ${\cal H}(2g-2)$ 
which contain every
affine invariant submanifold of ${\cal H}(2g-2)$ of rank
at most $g-1$. In particular,
there are only finitely many such manifolds of rank $g-1$. 

To show finiteness of affine invariant manifolds of any
rank $2\leq \ell\leq g-1$, 
apply Proposition \ref{realanalytic} and the first part of this proof to  
each of the finitely many 
affine invariant manifolds constructed in some
previous step and proceed by inverse induction on the rank.
\end{proof}

\begin{remark}
The proof of the second part of Corollary \ref{smalleststratum} immediately
extends to the following statement. An affine invariant manifold
${\cal C}$ with trivial absolute period foliation
contains only finitely many affine invariant manifolds of 
rank at least two.
\end{remark}

\section{Nested affine invariant submanifolds of the same rank}\label{nested}

The goal of this section is to analyze affine invariant submanifolds of affine
invariant manifolds ${\cal C}$ of the same rank and to complete the 
proof of Theorem \ref{finite}. 
Our strategy is a variation of the
strategy used in Section \ref{uniqueness}. 
Namely, we reduce the finiteness statement
to the nonexistence of a certain real analytic
$SL(2,\mathbb{R})$-invariant splitting of the tangent bundle of ${\cal C}$.

We continue to use the notations from the previous sections.
In particular, we denote by ${\cal A\cal P}({\cal C})$ the 
absolute period foliation of an affine invariant manifold ${\cal C}$.
We always assume that this foliation is not trivial.
Moreover, we always assume that the zeros of a differential 
$q\in {\cal C}$ are numbered; this may require to replace
${\cal C}$ by a finite cover.

Recall that the Hodge bundle ${\cal H}\to {\cal M}_g$ is a holomorphic
vector bundle in the orbifold sense. We denote by $J$ its complex
structure. Using the notation from the appendix, 
let ${\cal V}\to Sp(2g,\mathbb{Z})\backslash Sp(2g,\mathbb{R})/U(g)$ be 
the tautological vector bundle over the moduli space 
${\cal A}_g$ of
principally polarized abelian differentials. The period map 
$F:{\cal H}\to {\cal V}$ is holomorphic. 

For an affine invariant manifold ${\cal C}$
contained in a component ${\cal Q}$ of a stratum of the moduli space of 
area one abelian differentials 
denote as before by ${\cal C}_+\subset {\cal H}_+\subset {\cal H}$ the
extension of ${\cal C}$ be scaling; then ${\cal C}_+$ is a 
complex suborbifold of ${\cal H}_+$. 
The foliation of 
${\cal C}$ into the orbits of the $GL^+(2,\mathbb{R})$-action is
holomorphic. 

We use these facts in the proof of 
the following analogue of Proposition \ref{realanalytic}. 
Lemma \ref{connection2} in the appendix discusses splittings 
in the sphere subbundle ${\cal S}$ of the tautological 
vector bundle ${\cal V}$ over the moduli space ${\cal A}_g$
of principally polarized abelian differentials which puts this proposition
into the framework of homogenous spaces.

\begin{proposition}\label{realanalytic2}
Let ${\cal C}$ be an affine invariant manifold of rank $\ell\geq 2$.
Then one of the following two possibilities holds true.
\begin{enumerate}
\item There are at most finitely many proper affine invariant submanifolds 
of ${\cal C}$ of rank $\ell$.
\item Up to passing to a finite cover, the tangent bundle $T{\cal C}$ 
of ${\cal C}$ admits
a non-trivial $\Phi^t$-invariant real analytic splitting 
$T{\cal C}={\cal A}\oplus {\cal V}$ 
where ${\cal A}$ is a
complex subbundle of 
$T{\cal A\cal P}({\cal C})$ and where ${\cal V}$ contains the tangent bundle of the
orbits of the $SL(2,\mathbb{R})$-action.
\end{enumerate}
\end{proposition}
\begin{proof} By the main results of \cite{EMM15}, it suffices 
to show that the second possibility in the proposition holds
true under the following assumption: There is a number $k\geq 1$, and there is 
an open subset $V$ of ${\cal C}$ such that the
set of all affine invariant submanifolds of ${\cal C}$ of complex
codimension $k$ 
whose rank
coincides with the rank of ${\cal C}$ is dense in $V$. 
We therefore suppose from now on that this is the case.

Let as before $J$ be the complex structure on ${\cal C}$. 
Let ${\cal N}$ be 
a $J$-invariant $SL(2,\mathbb{R})$-invariant 
real analytic subbundle of $T{\cal C}$ 
which is complementary to the 
tangent bundle of the orbits of the $SL(2,\mathbb{R})$-action and invariant
under the natural complex structure $J$. 
Such a bundle exists because the leaves of the foliation
into the orbits of the
action of $GL^+(2,\mathbb{R})$
of the natural extension ${\cal C}_+$ of ${\cal C}$ by scaling are both complex
and affine. In fact, using the notations from Section \ref{connections}, 
if we write ${\cal W}={\cal L}\cap {\cal Z}$ then we can take 
${\cal N}=p^{-1}({\cal W}\oplus \overline{\cal W})$.

Let 
${\cal P}\to {\cal C}$ be the real analytic
fibre bundle whose fibre at a point $q\in {\cal C}$ equals 
the Grassmannian of all \emph{complex} subspaces of ${\cal N}_q$ of 
complex codimension $k$. This is a real analytic subbundle
of the fibre bundle whose fibre at $q$ equals the Grassmannian of
all oriented real linear subspaces of codimension $2k$ in ${\cal N}_q$.
The bundle ${\cal P}$ admits a natural
$SL(2,\mathbb{R})$-invariant stratification ${\cal P}=\cup_{i=0}^k{\cal P}_i$
where ${\cal P}_i$ consists of all subspaces which intersect
$T{\cal A\cal P}({\cal C})$ in a subspace of complex codimension $k-i$
(note that some of these strata may be empty).




The flow $\Phi^t$ acts on ${\cal C}$ as a one-parameter
group of real analytic transformations, and its
differential preserves the bundle ${\cal N}$.
For $q\in {\cal C}$ and $t\in \mathbb{R}$ let $\rho(q,t)$ 
be the image of ${\cal P}(\Phi^tq)$ under the map 
$d\Phi^{-t}$. Then  
%
\[{\cal R}_\infty=\cap_t\cup_q \rho(q,t)\]
is a (possibly empty) real analytic subvariety of ${\cal P}$. This 
subvariety is invariant under the action of $\Phi^t$. 

The tangent bundle of 
an affine invariant submanifold ${\cal D}$ of ${\cal C}$ 
intersects the complex vector bundle ${\cal N}$ in a complex subbundle
${\cal N}\cap T{\cal D}\vert {\cal D}$.
This subbundle 
is invariant under the action of the flow $\Phi^t$.
Thus if $q\in {\cal C}$ is contained in 
an affine invariant submanifold of ${\cal C}$ of the same rank as
${\cal C}$ and of complex codimension $k$, then 
${\cal R}_\infty\cap {\cal P}_0(q)\not=\emptyset$.  In particular,
by the assumption on $V$, 
the subvariety ${\cal R}_\infty$ of ${\cal P}$ is not empty, and the restriction
of the canonical projection $\pi:{\cal P}\to V$ to ${\cal R}_\infty$
is surjective.  

Following the reasoning in the proof of Proposition \ref{realanalytic},
by restricting to a smaller open $\Phi^t$-invariant set 
$U\subset V$ we may assume
that ${\cal R}_\infty\to U$ is a real analytic fibre bundle. 
In particular, the fibre is not empty and compact. 


Now ${\cal P}_0\subset {\cal P}$ is an open 
$\Phi^t$-invariant subset of ${\cal P}$, and
${\cal R}_\infty\cap {\cal P}_0(q)\not=\emptyset$ for 
a dense set of points $q$ of $U$. 
Therefore
up to perhaps decreasing the set $U$, we may assume that 
the intersection of ${\cal R}_\infty$ with
${\cal P}_0(q)$ is not empty for every $q\in U$.

To summarize, the restriction of ${\cal R}_\infty$ to the preimage of
$U$ under the projection ${\cal P}\to U$ is a fibre bundle with compact fibre
which contains
a non-empty open subbundle
${\cal E}={\cal R}_\infty\cap {\cal P}_0$. 
This subbundle is invariant under the natural action of the Teichm\"uller flow 
$\Phi^t$.


Each point $q\in {\cal E}$ is a complex subspace of ${\cal N}_q$ 
of complex codimension $k$ which intersects 
$T{\cal A\cal P}({\cal C})$ in a subspace of complex codimension $k$.
Thus for $q\in U$ we can define
\[E(q)=\cap_{z\in {\cal E}(q)}z\subset {\cal N}_q\subset T{\cal C}_q.\]
Then $E(q)$ is a (possibly trivial) complex linear subspace of $T{\cal C}_q$.
As ${\cal E}\to U$ is an open subset of a real analytic fibre bundle, by possibly replacing
the set $U$ by a proper open dense $\Phi^t$-invariant subset
we may assume that for $q\in U$, 
the dimension of $E(q)$ is minimal (perhaps zero). In particular, 
if this dimension is positive, then $\cup_{q\in U}E(q)$ is a real analytic
complex subbundle of ${\cal N}\vert U$. 

Our next goal is to show that indeed, for $q\in U$ the complex dimension of 
$E(q)$ is not smaller than $\ell-1\geq 1$. 
To this end let $\gamma\subset U$ be a periodic orbit with the properties stated in 
Corollary \ref{localirreducible}. For $q\in \gamma$, 
the return map $d\Phi^{\ell(\gamma)}(q)$ acts
on ${\cal N}_q$ as a semisimple linear map $A$, with all eigenvalues 
real and positive, and all eigenvalues but the largest $e^{\ell(\gamma)}$ and the
smallest $e^{-\ell(\gamma)}$ simple over $\mathbb{C}$
(i.e. eigenspaces for eigenvalues different from $e^{\ell(\gamma)}$ and 
$e^{-\ell(\gamma)}$ 
are complex subspaces of complex dimension one).
As ${\cal N}_q$ is complementary
to the tangent space of the orbits of the $SL(2,\mathbb{R})$-action, 
the eigenspaces for the
eigenvalues $e^{\ell(\gamma)}, e^{-\ell(\gamma)}$ are contained in 
$T{\cal A\cal P}({\cal C})$. Moreover, if $a>0$ is an eigenvalue of $A$, then
the same holds true for $a^{-1}$. 

By definition, a point 
$z\in {\cal E}(q)$ is a complex subspace of $T{\cal C}_q$ of 
complex codimension $k$ which is complementary to 
some $k$-dimensional complex subspace of $T{\cal A\cal P}({\cal C})_q$, and
its image under the map
$A$ is complex as well.
We claim that such a subspace 
has to contain the sum of the eigenspaces for $A$ with respect to the 
eigenvalues of absolute value different from $e^{\ell(\gamma)},e^{-\ell(\gamma)}$.
 
To this end recall that the fibre ${\cal P}(q)$ of the bundle ${\cal P}$ at $q$
is a closed subset of the
Grassmann manifold of all oriented linear subspaces of $T{\cal C}_q$ 
of real codimension $2k$.
Thus if  $z\in {\cal E}(q)$, then any limit 
of a subsequence of the sequence $A^iz$ $(i\to \pm \infty)$ 
is complex. Such a limit $y$ is
a fixed point for the action of $A$ on ${\cal P}(q)$ and hence
it is a direct sum of subspaces of eigenspaces.
Furthermore it is complex and hence symplectic. Therefore $y$ 
contains with an eigenspace for the eigenvalue $1<a<e^{\ell(\gamma)}$ 
(which is of dimension one by assumption) also the eigenspace for the
eigenvalue $a^{-1}$. 

By the same reasoning, the dimension of the intersection of 
$y$ with the eigenspace for the eigenvalue $e^{\ell(\gamma)}$ coincides with the
dimension of the intersection with the eigenspace for the 
eigenvalue $e^{-\ell(\gamma)}$. 
Furthermore, as $y\in {\cal E}\subset {\cal P}_0$ by assumption
and as $T{\cal A\cal P}({\cal C})_q$ equals the direct sum of 
the eigenspaces for action of $A$ on ${\cal N}_p$ with respect to the
eigenvalues $e^{\ell(\gamma)},e^{-\ell(\gamma)}$, we conclude that 
the real dimension of this intersection equals 
${\rm dim}_{\mathbb{C}}({\cal A\cal P}({\cal C})-k$. For reasons of dimension,
this implies that $y$ contains the sum of the eigenspaces for the
eigenvalues different from $e^{\ell(\gamma)},e^{-\ell(\gamma)}$. 

Now this discussion is valid for each $y\in {\cal E}_q$ and therefore 
for each $q\in \gamma$, 
the sum of these eigenspaces is contained in $\cap E(q)$. 
In particular, we have ${\rm dim}_{\mathbb{C}}E(q)\geq \ell-1$ and hence
this dimension is
contained in the interval $[\ell-1,{\rm dim}({\cal C})-k-1]$, 
moreover ${\cal N}_q=T{\cal A\cal P}({\cal C})_q+ E(q)$.

As $E(q)$ depends in a real analytic fashion on $q\in U$, the assignment
$q\to E(q)$ is a real analytic $\Phi^t$-invariant subbundle of 
$T{\cal C}$. 
This bundle 
can be used  to construct a real analytic 
$SL(2,\mathbb{R})$-invariant splitting of $T{\cal C}$ as claimed in the second
part of the proposition. 

Namely, let $q\in U$ and let $\mathfrak{a}_1,\dots,\mathfrak{a}_{m-1}\in \mathbb{R}^m$
be linearly independent with zero mean 
(here $m$ is the number of zeros of a differential in ${\cal C}$) such that
for some $u\leq m-1$ 
the vector fields $X_{\mathfrak{a}_1},\dots,X_{\mathfrak{a_u}}$ are 
tangent to ${\cal C}$ and such that moreover their complex span is a linear
subspace of $T{\cal A\cal P}({\cal C})$ complementary to $E(q)$. 
By invariance and Lemma \ref{affine}, these vector fields span a subbundle of 
$T{\cal A\cal P}({\cal C})$ which is complementary to 
the bundle $\cup_qE(q)$. This is what we wanted
to show.
\end{proof}

By Proposition \ref{realanalytic2},
if the affine invariant manifold ${\cal C}$ contains infinitely
many affine invariant manifolds of the same rank as ${\cal C}$, 
then there exists an open dense $\Phi^t$-invariant subset 
$U$ of ${\cal C}$, and there is 
a $\Phi^t$-invariant real analytic splitting 
$T{\cal C}\vert U={\cal A}\oplus {\cal V}$ as a sum of complex vector bundles
with the following additional 
property. The bundle ${\cal A}$ is a subbundle of 
$T{\cal A\cal P}({\cal C})$, and ${\cal V}$ contains the tangent bundle
of the orbits of $SL(2,\mathbb{R})$-action.
The following lemma shows that such a splitting does not exist.

\begin{lemma}\label{closed}
Let ${\cal C}$ be an affine invariant manifold; then there is no
nontrivial $\Phi^t$-invariant real analytic splitting
$T{\cal C}={\cal A}\oplus {\cal V}$ on an open dense invariant
subset of ${\cal C}$ with the following property. 
${\cal A}$ is a complex 
subbundle of $T{\cal A\cal P}({\cal C})$, and 
the complex bundle ${\cal V}$ contains the tangent bundle of the orbits of the
$SL(2,\mathbb{R})$-action.
\end{lemma} 
\begin{proof} We proceed as in the proof of Lemma \ref{flatone} and 
Proposition \ref{nosplit}. Namely, assume to the contrary that
there is an open dense invariant set $U\subset {\cal C}$, and there is 
a $\Phi^t$-invariant splitting $T{\cal C}\vert U={\cal A}\oplus {\cal V}$ as in the 
statement of the lemma.
 
An affine invariant manifold is locally defined by
real linear equations in period coordinates (see \cite{W14}).  
The affine structure of ${\cal C}$ defines a flat connection 
$\nabla^{\cal C}$ on 
$T{\cal C}$ which is invariant under affine transformations. 
In particular, this 
connection is invariant under the $SL(2,\mathbb{R})$-action. 
The bundle ${\cal A}\subset T{\cal A\cal P}({\cal C})$ is flat, i.e. invariant
under parallel transport for $\nabla^{\cal C}$. Namely, it is globally trivialized
by globally defined vector fields $X_{\mathfrak{a}_i}$ where $\mathfrak{a}_i\in 
\mathbb{C}^m$ (compare the proof of Proposition \ref{realanalytic2}).

Our goal is to show that the bundle ${\cal V}$ is flat as well. To this end
recall that there is a real analytic complex subbundle 
${\cal N}\subset T{\cal C}$ which is invariant under the $SL(2,\mathbb{R})$-action
and transverse to the tangent bundle ${\cal T}$ 
of the foliation of ${\cal C}$ into
the orbits of the 
$SL(2,\mathbb{R})$-action. The restriction of the bundle
${\cal N}$ to the leaves of the absolute period foliation is flat. 
Using the splitting $T{\cal C}={\cal N}\oplus {\cal T}$, project the
flat connection $\nabla^{\cal C}$ 
on $T{\cal C}$ to a connection $\nabla^{\cal N}$ on 
${\cal N}$. 
The bundle ${\cal N}\vert U$ admits a real analytic 
$\Phi^t$-invariant decomposition
${\cal N}={\cal W}\oplus {\cal A}$ where ${\cal W}={\cal N}\cap {\cal V}$.
Use this decomposition of ${\cal N}$ to project the connection
$\nabla^{\cal N}$ to a connection $\nabla^{\cal W}$ on ${\cal W}$
(compare the proof of Lemma \ref{flatone}). 
This projection determines a $\Phi^t$-equivariant
real analytic tensor field 
\[\beta\in \Omega(T^*{\cal C}\otimes {\cal W}^*\otimes {\cal A})\]
(compare the discussion in Section \ref{uniqueness}).

We aim at showing  
that $\beta$ vanishes identically, and as before,  
for this it suffices to show that $\beta$ vanishes at any point 
$q\in {\cal C}$ which is contained in a periodic orbit $\gamma$ for 
$\Phi^t$ with the properties stated in Corollary \ref{localirreducible}.
For $q\in \gamma$, 
the return map $A=d\Phi^{\ell(\gamma)}(q)$ acts on ${\cal  N}_q$
as a semisimple automorphism, with all eigenvalues real and distinct from one.
Moreover, no product of two eigenvalues is an eigenvalue. 

Denote by ${\cal R}^{su}$ the intersection
of $T{\cal A\cal P}({\cal C})$ with the tangent bundle of 
the strong unstable foliation, and 
let similarly ${\cal R}^{ss}$ be the intersection of $T{\cal A\cal P}({\cal C})$
with the tangent bundle of the strong stable foliation.
Then the eigenspace for the restriction of the return map 
$A$ to ${\cal N}_q$ for the eigenvalue
$e^{-\ell(\gamma)}$ equals the subspace 
${\cal R}^{ss}_q$, and the eigenspace for the eigenvalue 
$e^{\ell(\gamma)}$ is the subspace ${\cal R}^{su}_q$.

The bundle ${\cal A}$ decomposes as ${\cal A}={\cal A}^{su}\oplus 
{\cal A}^{ss}$ where ${\cal A}^{su}={\cal A}\cap {\cal R}^{su}$ and
${\cal A}^{ss}={\cal A}\cap {\cal R}^{ss}$. 
We use this to show that $\beta$ vanishes at $q$.  
Namely, since the decomposition is invariant under the action of the
Teichm\"uller flow, all connections are invariant and hence the
contraction of $\beta$ with the generator $X\in T{\cal C}$ of 
$\Phi^t$ vanishes everywhere.

By equivariance and the reasoning in the proof of Lemma \ref{flatone},
this implies that if $\beta$ does not vanish at $q$ then 
up to replacing
$\Phi^{\ell(\gamma)}$ by $\Phi^{-\ell(\gamma)}$, we may assume that
there are vectors $Y\in {\cal N}_q\subset T{\cal C}_q, 
Z\in {\cal W}_q$ so that the 
product of the growth rates of $Y$ and $Z$ under $\Phi^{\ell(\gamma)}$ 
equals $e^{\ell(\gamma)}$. However, by the choice of $\gamma$
and the fact that the eigenvalues of the restriction of $A$ 
to ${\cal N}_q$ are off the unit circle, there are no 
two nontrivial vectors with this property. 

As a consequence, the tensor field $\beta$ vanishes identically,
and the bundle ${\cal V}$ 
is flat (compare the discussion in Section \ref{uniqueness}).
This implies that 
$T{\cal C}$ splits as a direct sum of two 
flat subbundles. By Theorem 5.1 of \cite{W14}, such a splitting does not 
exist. The lemma follows.
\end{proof}

As an immediate consequence of Proposition \ref{realanalytic} and
Lemma \ref{closed}, we obtain 

\begin{corollary}\label{final}
An affine invariant manifold ${\cal C}$ of rank at least two 
contains only finitely many affine
invariant submanifolds of the same rank.
\end{corollary}

Theorem \ref{finite} from the introduction is now an immediate
consequence of Proposition \ref{smalleststratum} and Corollary \ref{final}.

\section{Algebraically primitive Teichm\"uller curves}\label{algebraically}

A point in the moduli space of abelian differentials can  be
viewed as a 
\emph{translation surface} $(X,\omega)$, where $X$ denotes a
Riemann surface of genus $g$. 
If the \emph{Veech group} of such a translation surface  
contains a pseudo-Anosov element, then the 
\emph{holonomy field} of $(X,\omega)$ coincides with the
trace field of the pseudo-Anosov element (see e.g. the
appendix of \cite{KS00}- we do not need any additional
information on this field). 
By Lemma 2.10 of \cite{LNW15}, if ${\cal C}$ is a rank
one affine invariant manifold then for all $(X,\omega)\in {\cal C}$, 
the holonomy field of $(X,\omega)$ equals the 
\emph{field of definition} of ${\cal C}$ \cite{W14}.



For the proof of Corollary \ref{cyclic} we have 
a closer look at rank one affine invariant
manifolds ${\cal C}$ 
whose field of definition $\mathfrak{k}$ is of degree $g$ over $\mathbb{Q}$.
Then $\mathfrak{k}$ is a totally real \cite{F16} 
number field of degree $g$, with 
ring of integers ${\cal O}_{\mathfrak{k}}$. Via the $g$ field embeddings
$\mathfrak{k}\to \mathbb{R}$, the group 
$SL(2,{\cal O}_{\mathfrak{k}})$ embeds into 
$G=SL(2,\mathbb{R})\times \dots \times SL(2,\mathbb{R})<Sp(2g,\mathbb{R})$ 
and in fact,
$SL(2,{\cal O}_{\mathfrak{k}})$ is a lattice in $G$.
The trace field of every periodic orbit $\gamma$ in ${\cal C}$
equals $\mathfrak{k}$ and hence the image of the
corresponding pseudo-Anosov element
$\Omega(\gamma)$ under the homomorphism 
$\Psi:{\rm Mod}(S)\to Sp(2g,\mathbb{R})$ 
is contained in $SL(2,{\cal O}_{\mathfrak{k}})$.

The following observation is immediate from Lemma \ref{localzariski}
and \cite{G12}. For its formulation, define the
\emph{extended local monodromy group} of an open contractible subset $U$ 
of ${\cal C}$ to be the subgroup of 
$SL(2,{\cal O}_{\mathfrak{k}})$ which is generated by the monodromy 
of those (parametrized) 
periodic orbits for $\Phi^t$ in ${\cal C}$ which pass through $U$. 
Compare with Proposition \ref{latticeordense}. 

\begin{lemma}\label{zariski2}
For a rank one affine invariant manifold ${\cal C}$ whose field of 
definition is of degree $g$ over
$\mathbb{Q}$, any extended local monodromoy group is Zariski dense in 
$SL(2,\mathbb{R})\times \dots \times SL(2,\mathbb{R})$.
\end{lemma}
\begin{proof}
By Proposition \ref{localzariski}, the projection 
of the extended local monodromy group of 
an open set $U\subset {\cal C}$ to the first 
factor $SL(2,\mathbb{R})$ of 
$G=SL(2,\mathbb{R})\times \dots \times SL(2,\mathbb{R})$ is 
Zariski dense in $SL(2,\mathbb{R})$ and hence it is non-elementary. 
Moreover, by definition and \cite{KS00,LNW15}, the invariant trace
field of the extended local monodromy 
group equals $\mathfrak{k}$ 
(compare \cite{G12} for the definition of the invariant 
trace field).
Thus by Corollary 2.2 of \cite{G12}, the extended 
local monodromy group of $U$ is
Zariski dense in $G$.
\end{proof}

In the statement of the next corollary, 
the affine invariant manifold ${\cal B}_+$ may be a
component of a stratum. As before, we put a lower index $+$ 
whenever we do not normalize the area of a holomorphic differential.

\begin{corollary}\label{fieldembed}
Let ${\cal C}_+$ be a rank one affine invariant 
manifold whose field of definition is of degree
$g$ over $\mathbb{Q}$. Assume that ${\cal C}_+$ is properly contained
in an affine invariant manifold ${\cal B}_+$ of rank at least three.
Let ${\cal Z}\to {\cal B}_+$ be the absolute holomorphic tangent bundle
of ${\cal B}_+$; then ${\cal Z}\vert {\cal C}_+$ splits as a sum of 
holomorphic 
line bundles which are invariant under both the Chern connection and 
the Gauss Manin connection.
\end{corollary} 
\begin{proof} In the case that the rank of ${\cal B}_+$ equals $g$  
(and hence ${\cal Z}={\cal H}$), the statement is immediate from 
Theorem 1.5 of \cite{W14}. Thus assume that
the rank of ${\cal B}_+$ is at most $g-1$. 

Since ${\cal C}_+\subset {\cal B}_+$, the restriction of ${\cal H}$
to ${\cal C}_+$ has two splittings
which are invariant under the extended local monodromy
of ${\cal C}_+$. The first splitting is the
splitting into $g$ line bundles obtained from the different field 
embeddings of the field of definition of ${\cal C}_+$ into $\mathbb{R}$
(see Theorem 1.5 of \cite{W14}).
The second splitting is the splitting into the 
absolute holomorphic tangent bundle 
${\cal Z}$ of 
${\cal B}_+$ (which is a holomorphic subbundle of 
${\cal H}\vert_{{\cal B}_+}$ 
whose complex rank equals the rank of 
${\cal B}_+$) and its symplectic complement. 
Since by Lemma \ref{zariski2} the extended local monodromy 
group of ${\cal C}_+$ is Zariski dense in 
$SL(2,\mathbb{R})\times \dots\times SL(2,\mathbb{R})$, the 
bundle ${\cal Z}\vert {\cal C}_+$ is a sum of invariant line bundles.  
\end{proof}

\begin{corollary}\label{teichmullercurve}
For a component ${\cal Q}_+$ of a stratum in genus $g\geq 3$, 
all affine invariant submanifolds
of rank one 
whose fields of definition are of degree $g$  over $\mathbb{Q}$
are contained in a finite collection of 
affine invariant submanifolds of rank at most two.
\end{corollary}
\begin{proof}
Let $\mathfrak{C}$ be the collection of all rank one 
affine invariant submanifolds of ${\cal Q}$ whose
field of definition
%
%
%
%
%
is a number field of degree $g$ over $\mathbb{Q}$. 
For each ${\cal C}_+\in \mathfrak{C}$, 
the restriction of the bundle ${\cal L}$ to ${\cal C}_+$ 
splits as a sum of flat holomorphic line bundles
which are invariant under the Gauss-Manin connection. 
Thus by Proposition \ref{realanalytic}
and its proof, all but finitely many elements of $\mathfrak{C}$ 
are submanifolds of a finite set of affine invariant submanifolds
of ${\cal Q}$ of rank at most $g-1$.

Now if ${\cal B}_+$ is an affine invariant manifold of rank
contained in $[3,g-1]$ which contains some 
${\cal C}_+\in \mathfrak{C}$ 
then by Corollary \ref{fieldembed},
the above reasoning can be applied to ${\cal B}_+$.
In finitely many steps we
find finitely many proper affine invariant manifolds 
${\cal C}_1,\dots,{\cal C}_k \subset {\cal Q}_+$ of rank
at most two which contain every 
${\cal C}_+\in \mathfrak{C}$.
This is the statement of the corollary. 
\end{proof}

Now we are ready to complete the proof of Corollary \ref{cyclic}.

\begin{corollary}\label{orbitclosures}
For $g\geq 3$ 
the $SL(2,\mathbb{R})$-orbit closure of 
a typical periodic orbit in any component of a stratum
is the entire stratum.
\end{corollary}
\begin{proof} Let ${\cal Q}$ be a component of a stratum and let 
$U\subset {\cal Q}$ be a non-empty open set.
Then 
a typical periodic orbit for $\Phi^t$ passes through $U$
\cite{H13}. Thus by Theorem \ref{finite} (see also \cite{MW15,MW16}), 
the $SL(2,\mathbb{R})$-orbit closure 
of a typical periodic orbit either equals the entire stratum, or it is 
an affine invariant
manifold of rank one.

By the second part of Theorem \ref{theolyapunov}, 
the trace field of a typical perodic orbit $\gamma$ is 
totally real and of degree $g$ over
$\mathbb{Q}$. If the rank of the $SL(2,\mathbb{R})$-orbit 
closure ${\cal C}$ of $\gamma$ equals one then
this trace field is the field of definition of ${\cal C}$
\cite{LNW15}. 
Thus the corollary follows from Corollary \ref{teichmullercurve}.
\end{proof}

We complete this work with the proof of 
Theorem \ref{finiteteich}.
We begin with

\begin{proposition}\label{fakerankone}
Let $g\geq 3$ and let ${\cal B}_+\subset {\cal Q}_+$ 
be a rank two affine invariant manifold. Then 
the union of 
all algebraically primitive Teichm\"uller curves which are
contained in ${\cal B}_+$ is nowhere dense in 
${\cal B}_+$.  
\end{proposition}
\begin{proof} Let ${\cal B}_+\subset {\cal Q}_+$ be 
a rank two affine invariant manifold. We argue by contradiction,
and we assume that the closure of the union of 
all algebraically primitive Teichm\"uller curves 
${\cal C}_+\subset {\cal B}_+$ contains some open subset
$V$ of ${\cal B}_+$.

Let ${\cal Z}\to {\cal B}_+$ be the absolute holomorphic
tangent bundle of ${\cal B}_+$. 
Let $U$ be a small contractible subset of $V$ 
so that there is a trivialization of the Hodge bundle
over $U$ defined by the Gauss Manin connection. 
The extended local monodromy group of $U$ preserves 
${\cal Z}$. Let ${\cal C}_i\subset {\cal B}_+$ be a 
sequence of algebraically primitive Teichm\"uller curves 
which pass through $U$ and whose closures contain
a compact subset of $U$ with non-empty interior $W$.

Let $\Pi:{\cal Q}_+\to {\cal M}_g$ be the canonical
projection and let ${\cal I}_g:{\cal M}_g\to {\cal A}_g$ be the 
Torelli map. The image under $\Pi$ of the curve 
${\cal C}_i$ is an algebraic curve which admits a 
\emph{modular embedding}.
Namely, by the main result of \cite{Mo06}, 
there is a totally real number field $K_i$ of degree 
$g$ over $\mathbb{Q}$, there is an order 
$\mathfrak{o}_{K_i}$ in $K_i$, 
and there is an embedding
\[SL(2,\mathfrak{o}_{K_i})\to SL(2,\mathbb{R})\times 
\dots \times SL(2,\mathbb{R})\to Sp(2g,\mathbb{R})\]
which maps $SL(2,\mathfrak{o}_{K_i})$ into $Sp(2g,\mathbb{Z})$ 
and such that 
the image of
${\cal C}_i$ under the Torelli map is contained in the Hilbert
modular variety 
$H(\mathfrak{o}_{K_i})$. This Hilbert modular variety is 
the quotient of 
${\bf H}^2\times \dots \times{\bf H}^2$ under the lattice
$SL(2,\mathfrak{o}_{K_i})$ in a Lie subgroup $G_i$ of 
$Sp(2g,\mathbb{R})$ which is isomorphic to 
$SL(2,\mathbb{R})\times\dots \times SL(2,\mathbb{R})$.

We claim that $G_i=G_j=G$ for all $i$.
Namely, assume otherwise. Then there are algebraically primitive
Teichm\"uller curves 
${\cal C}_i,{\cal C}_j$ which intersect $U$
and for which the groups $G_i,G_j$ are distinct. 
By Lemma \ref{zariski2}, the extended  
local monodromy groups of ${\cal C}_i\cap U$ and 
${\cal C}_j\cap U$ are Zariski dense in $G_i,G_j$. Therefore
the Zariski closure in $Sp(2g,\mathbb{R})$ of the extended 
local monodromy group of $U\subset {\cal B}_+$ contains
$G_i\cup G_j$. But as $G_i\not=G_j$, 
a subgroup of $Sp(2g,\mathbb{R})$ which 
contains $G_i\cup G_j$ can not preserve
the subspace ${\cal Z}$. This is a contractiction and implies 
that indeed, $G_i=G_j=G$ for all $i$.

Write $SL(2,\mathfrak{o})=SL(2,\mathfrak{o}_{K_i})$. 
The Hilbert 
modular variety $H(\mathfrak{o})=H(\mathfrak{o}_{K_i})\subset {\cal A}_g$ 
consists of abelian varieties with real multiplication with the field
$K=K_i$. The image of ${\cal C}_i$ under the map
${\cal I}_g\circ \Pi$ is contained in 
$H(\mathfrak{o})$.
As a consequence, the set of points in 
${\cal B}_+$ which are mapped by 
the composition of the foot-point projection 
$\Pi:{\cal B}_+\to {\cal M}_g$ with
the Torelli map ${\cal I}_g$ into 
$H(\mathfrak{o})$ contains  
a dense subset of the open set $W$.
But $H(\mathfrak{o})$ is a complex submanifold of 
${\cal A}_g$ and this composition map is holomorphic and 
therefore the image of ${\cal B}_+$ is contained in 
$H(\mathfrak{o})$. 

We showed so far that each point in ${\cal B}_+$ is an abelian
differential whose Jacobian has real multiplication with
$K$. Now a point on an algebraically primitive 
Teichm\"uller curve is mapped to an eigenform for real 
multiplication \cite{Mo06} and hence the closure of the 
set of differentials in ${\cal B}_+$ 
which are mapped to eigenforms for this real multiplication
contains an open set. This implies as before that
each point in ${\cal B}_+$ corresponds to such an eigenform and
hence  ${\cal B}_+$ is
a rank one affine invariant manifold, contrary to our assumption.
The proposition follows
\end{proof}

\noindent
{\it Proof of Theorem \ref{finiteteich}:}

Let ${\cal Q}$ be a component of a stratum in genus $g\geq 3$. 
By Corollary \ref{teichmullercurve}, 
there are finitely many affine invariant submanifolds 
${\cal B}_1,\dots,{\cal B}_k$ of rank two which contain
all but finitely many algebraically primitive Teichm\"uller curves.

Let ${\cal B}_i$ be such an affine invariant manifold of rank
two. Assume that its dimension equals $r$ for some $r\geq 4$. 
By Proposition \ref{fakerankone}, the closure of the 
union of all algebraically
primitive Teichm\"uller curves which 
are contained in ${\cal B}_i$ 
is nowhere dense in ${\cal B}_i$. As this closure is invariant under
the action of $GL(2,\mathbb{R})$, it consists of a finite
union of affine invariant manifolds. 
The dimension of each of these invariant submanifolds
is at most $r-1$. 

If there are submanifolds of rank two in this collection 
then we can repeat this argument with each of these finitely 
many submanifolds. 
By inverse induction on the dimension, this yields that all but finitely 
many algebraically primitive Teichm\"uller curves are
contained in one of finitely many 
affine invariant manifolds of rank one.
The field of definition of such a manifold coincides with the
field of definition of the Teichm\"uller curve, in particular
it is of degree $g$ \cite{LNW15}.

By the main result of \cite{LNW15}, a rank one affine invariant
manifold with field of definition of degree $g$ over $\mathbb{Q}$ 
only contains finitely many Teichm\"uller curves.
Thus the number of algebraically primitive 
Teichm\"uller curves in ${\cal Q}$ is finite as promised.

%

\begin{appendix}
\section{Structure of the homogeneous space
$Sp(2g,\mathbb{Z})\backslash Sp(2g,\mathbb{R})$}
\label{structure}

In this appendix we collect some geometric properties of the 
Siegel upper half-space $\mathfrak{D}_g=
Sp(2g,\mathbb{R})/U(g)$ and its quotient
${\cal A}_g=Sp(2g,\mathbb{Z})\backslash \mathfrak{D}_g$
which are either directly or indirectly used in the proofs of our
main results.

The tautological vector bundle 
\[{\cal V}\to \mathfrak{D}_g\] over the 
Hermitean symmetric space 
$\mathfrak{D}_g=Sp(2g,\mathbb{R})/U(g)$
is obtained as follows. 

Via the right action of the unitary group $U(g)$, 
the symplectic group $Sp(2g,\mathbb{R})$ is an 
$U(g)$-principal
bundle over $\mathfrak{D}_g$. The bundle ${\cal V}$ is the associated
vector bundle 
\[{\cal V}=Sp(2g,\mathbb{R})\times_{U(g)}\mathbb{C}^g\]
where $U(g)$ acts from the right by $(x,y,\alpha)\to 
(x\alpha,\alpha^{-1}y)$.
The bundle ${\cal V}$ is holomorphic, and it is equipped with
a hermitean metric obtained from an $U(g)$-invariant 
hermitean inner product 
on $\mathbb{C}^g$.
As $U(g)$ acts transitively on the unit sphere in $\mathbb{C}^g$, 
with isotropy group $U(g-1)$, 
the associated sphere bundle 
\[{\cal S}=Sp(2g,\mathbb{R})\times_{U(g)}S^{2g-1}\] in 
${\cal V}\to \mathfrak{D}_g$ 
can naturally be identified with the homogeneous space
\[{\cal S}=Sp(2g,\mathbb{R})/U(g-1)\]
(Proposition 5.5 of \cite{KN63}).

The group $Sp(2g-2,\mathbb{R})$ is the isometry group of 
Siegel upper half-space 
\[\mathfrak{D}_{g-1}=Sp(2g-2,\mathbb{R})/U(g-1).\]
Since the action of $Sp(2g-2,\mathbb{R})$ on 
$\mathfrak{D}_{g-1}$ is transitive, with isotropy group 
$U(g-1)$, the bundle 
${\cal S}=Sp(2g,\mathbb{R})/U(g-1)\to \mathfrak{D}_g$ can also be identified with 
the associated bundle
\[{\cal S}=Sp(2g,\mathbb{R})\times_{Sp(2g-2,\mathbb{R})}
\mathfrak{D}_{g-1}\]
where $Sp(2g-2,\mathbb{R})$ act via
\[(g,x)h=(gh,h^{-1}(x)).\]
The first factor projection then defines a 
projection 
\[\Pi:{\cal S}\to Sp(2g,\mathbb{R})/Sp(2g-2,\mathbb{R}).\]

Let $\omega=\sum_idx_i\wedge dy_i$ be
the standard symplectic form on $\mathbb{R}^{2g}$. 
The standard 
representation of $Sp(2g,\mathbb{R})$ on 
$(\mathbb{R}^{2g},\omega)$ naturally 
extends to an action of $Sp(2g,\mathbb{R})$ 
on $\mathbb{R}^{2g}\otimes \mathbb{C}=
\mathbb{C}^{2g}$. 
The open subset
\[\mathcal{O}=\{x+iy\mid x,y\in \mathbb{R}^{2g}, \omega(x,y)>0\}\subset
\mathbb{C}^{2g}\] is $Sp(2g,\mathbb{R})$-invariant.
It contains the
invariant hypersurface   
\[\Omega=\{x+iy\in \mathbb{C}^{2g}\mid  
\omega(x,y)=1\}.\]

We claim that $\Omega$ can naturally and $Sp(2g,\mathbb{R})$-equivariantly
be identified with the homogeneous space 
\[Sp(2g,\mathbb{R})/Sp(2g-2,\mathbb{R}).\]
To this end just observe that 
the diagonal action of the group $Sp(2g,\mathbb{R})$ 
on $\Omega$ is transitive.  
The stabilizer in $Sp(2g,\mathbb{R})$ of a point 
$x+iy\in \Omega$ is isomorphic to 
a standard embedded 
\[{\rm{Id}}\times Sp(2g-2,\mathbb{R})<
Sp(2,\mathbb{R})\times Sp(2g-2,\mathbb{R})<Sp(2g,\mathbb{R}).\]



The group $SL(2,\mathbb{R})$ acts from the right on $\Omega$
as follows. The real and imaginary part of 
a point $x+iy\in \Omega$ define the basis of a two-dimensional
symplectic subspace of $\mathbb{R}^{2g}$. The group $SL(2,\mathbb{R})$
acts by basis transformation on this subspace, 
preserving the symplectic form. 
The $SL(2,\mathbb{R})$-orbit through a point 
$x+iy\in \Omega$
preserves pointwise the symplectic complement of the subspace $V$ of 
$\mathbb{R}^{2g}$ spanned by $x$ and $y$ as well as the complexification
of $V$.


As the left action of $Sp(2g,\mathbb{R})$ on ${\cal O}$ is the restriction
of a linear action on $\mathbb{C}^{2g}$, the tangent bundle 
of ${\cal O}$ admits an $Sp(2g,\mathbb{R})$-invariant flat connection. 
Vector fields which are tangent to the orbits of a one-parameter
group of translations are globally parallel. 
This connection restricts to a flat connection $\nabla^{GM}$ on 
$T{\cal O}\vert \Omega$. 

The bundle $T{\cal O}\vert \Omega$ splits a sum 
\[T{\cal O}\vert \Omega=T\Omega\oplus \mathbb{R}\]
where the trivial line bundle $\mathbb{R}$ is the tangent bundle of the
orbits of the one-parameter group of deformations
$((x+iy),t)\to e^tx+ ie^ty$ transverse to
$\Omega$.  This splitting is not flat, i.e. it is not invariant under 
the connection $\nabla^{GM}$.

The tangent bundle  $T\Omega$ of $\Omega$ has a natural
splitting 
\[T\Omega={\cal R}\oplus {\cal T}\]
where for a point $x+iy\in \Omega$ we have 
\[{\cal R}_{x+iy}=\{u+iv\mid \omega(u,x)=\omega(u,y)=
\omega(v,x)=\omega(v,y)=0\}\]
and where ${\cal T}$ is tangent 
to the orbits
of the right action of $SL(2,\mathbb{R})$. 
Moreover, we have ${\cal R}={\cal E}\otimes \mathbb{C}$ where
${\cal E}_{x+iy}=\{u\in \mathbb{R}^{2g}\mid \omega(u,y)=\omega(u,x)=0\}$.

By definition of the action of $Sp(2g,\mathbb{R})$, the bundles 
${\cal R}$ and ${\cal T}$ are invariant under both the left action of 
$Sp(2g,\mathbb{R})$ and the right action of $SL(2,\mathbb{R})$. Thus 
the flat left $Sp(2g,\mathbb{R})$-invariant connection $\nabla^{GM}$ on 
$T{\cal O}\vert \Omega$ projects to a left $Sp(2g,\mathbb{R})$-invariant 
right $SL(2,\mathbb{R})$-invariant connection 
$\nabla^{\cal R}$ on ${\cal R}$ defined as follows.
Let 
\[P:T{\cal O}\vert \Omega=
{\cal R}\oplus {\cal T}\oplus \mathbb{R}\to {\cal R}\]
be the canonical projection, and 
for $X\in T\Omega$ and a local section $Y$ of ${\cal R}$ define
$\nabla^{\cal R}_XY=P\nabla_X^{GM}(Y)$.

We summarize this as follows.

\begin{lemma}\label{connection}
The flat $Sp(2g,\mathbb{R})$-invariant connection on $T{\cal O}$ 
projects to a connection $\nabla^{\cal R}$ 
on ${\cal R}$ which is invariant under both the left 
$Sp(2g,\mathbb{R})$ action and the right $SL(2,\mathbb{R})$ action. 
\end{lemma}

The curvature of the connection $\nabla^{\cal R}$ is a
two-form on $\Omega$ with values in the Lie algebra $\mathfrak{sp}(2g-2,\mathbb{R})$
of $Sp(2g-2,\mathbb{R})$, acting as an 
algebra of transformation on ${\cal R}$. 
The restriction of this two-form to the tangent bundle of the
orbits of the $SL(2,\mathbb{R})$-action vanishes. Moreover, 
the two-form is equivariant with respect to the left action of 
$Sp(2g,\mathbb{R})$ and the right action of $SL(2,\mathbb{R})$.

We say that the curvature form $\Theta$ for a connection $\nabla$ 
on a complex vector bundle
$E\to M$  
\emph{splits $E$ as a complex vector bundle} 
if there is a nontrivial $\Theta$-invariant decomposition $E=E_1\oplus E_2$ 
as a Whitney sum of two complex 
vector bundles. This means that for any $x\in M$ and any two vectors 
$Y,Z\in T_xM$ the map $\Theta(Y,Z)$ preserves the decomposition $E=E_1\oplus E_2$. 

Since $\Omega$ is not locally affine, 
the curvature form of the connection $\nabla^{\cal R}$ on ${\cal R}$ does not
vanish identically. 
Furthermore, the stabilizer in $Sp(2g,\mathbb{R})$ of 
a point $z\in \Omega$ can be identified with the subgroup $Sp(2g-2,\mathbb{R})$, 
which act on the fibre of ${\cal R}$ at $z$ via the standard representation of
$Sp(2g-2,\mathbb{R})$ on 
$\mathbb{C}^{2g-2}$, viewed as the complexification of the standard representation on 
$\mathbb{R}^{2g-2}$. Since the standard representation of 
$Sp(2g-2,\mathbb{R})$ on the complex vector space 
$\mathbb{C}^{2g-2}$ is irreducible,     
by equivariance, we have 

\begin{lemma}\label{split}
The curvature form of $\nabla^{\cal R}$ does not split 
${\cal R}$ as a complex vector bundle. 
\end{lemma}

The complement of the zero section ${\cal V}_0\subset {\cal V}$
of the bundle ${\cal V}\to {\cal D}_g$ 
is a complex manifold. 
The fibration ${\cal S}\to \Omega$ extends to a holomorphic fibration
${\cal V}_0\to {\cal O}$ of complex
manifolds. The fibres of the fibration  
define a foliation ${\cal F}$ of
${\cal V}_0$. 

The following is immediate from the definition of the complex structure on 
${\cal V}_0$ and on ${\cal O}\subset \mathbb{C}^{2g}$.  

\begin{lemma}\label{fholom}
The foliation ${\cal F}$ is holomorphic. A fibre is biholomorphic
to $\mathfrak{D}_{g-1}$. 
\end{lemma}

Let for the moment $G$ be an arbitrary Lie group. 
A \emph{$G$-connection} for a $G$-principal bundle
$X\to Y$ is given by a ${\rm Ad}(G)$-invariant
subbundle of the tangent
bundle of $X$ transverse to the tangent bundle of the 
fibres. Such a bundle 
is called \emph{horizontal}. 

The group $Sp(2g,\mathbb{R})$ is an 
$Sp(2g-2,\mathbb{R})$-principal bundle over  
$\Omega$.
In the statement of the following Lemma, 
the type $(2g,2g-1)$ stems from the fact that
$\Omega$ is a hypersurface in the manifold ${\cal O}$ 
with invariant indefinite metric of type $(2g,2g)$.

\begin{lemma}\label{connection2}
The $Sp(2g-2,\mathbb{R})$-principal bundle 
$Sp(2g,\mathbb{R})\to \Omega$ admits a natural 
real analytic $Sp(2g-2,\mathbb{R})$-connection 
which is invariant
under the left action of $Sp(2g,\mathbb{R})$
and the right action of $SL(2,\mathbb{R})$. The 
horizontal bundle 
${\cal Z}_0$ contains 
the tangent bundle ${\cal T}$ 
of the orbits of the $SL(2,\mathbb{R})$-action,
and it admits an $SL(2,\mathbb{R})$-invariant
$Sp(2g,\mathbb{R})$-invariant pseudo-Riemannian metric $h$ 
of type $(2g,2g-1)$. The $h$-orthogonal complement 
${\cal Y}_0$ of ${\cal T}$ in ${\cal Z}_0$ is a real analytic
$SL(2,\mathbb{R})$-invariant $Sp(2g,\mathbb{R})$invariant 
bundle.  
\end{lemma}
\begin{proof} 
The fibre containing the identity induces 
an embedding of Lie algebras
\[\mathfrak{sp}(2g-2,\mathbb{R})\to
\mathfrak{sp}(2g,\mathbb{R}).\] The restriction of 
the Killing form $B$ 
of $\mathfrak{sp}(2g,\mathbb{R})$ to 
the Lie algebra $\mathfrak{sp}(2g-2,\mathbb{R})$ is non-degenerate.
Thus the $B$-orthogonal complement $\mathfrak{z}$ 
of $\mathfrak{sp}(2g-2,\mathbb{R})$ 
is a linear subspace of 
$\mathfrak{sp}(2g,\mathbb{R})$ which is complementary
to $\mathfrak{sp}(2g-2,\mathbb{R})$ and invariant under the 
restriction of the adjoint representation ${\rm Ad}$ of 
$Sp(2g,\mathbb{R})$ to $Sp(2g-2,\mathbb{R})$.
The restriction to $\mathfrak{z}$ of the Killing form 
is a non-degenerate bilinear form of type 
$(2g,2g-1)$. 

The group $Sp(2g,\mathbb{R})$ acts by left translation
on itself, and this commutes with the right action of
$Sp(2g-2,\mathbb{R})$. Hence this action 
defines an action by automorphisms of the 
principal bundle. 

Define a $\mathfrak{sp}(2g-2,\mathbb{R})$-valued one-form
$\theta$ on $Sp(2g,\mathbb{R})$ by requiring that 
$\theta(e)$ equals the canonical projection
\[T_eSp(2g,\mathbb{R})=
\mathfrak{z}\oplus \mathfrak{sp}(2g-2,\mathbb{R})\to
\mathfrak{sp}(2g-2,\mathbb{R})\]
and 
\[\theta(g)=\theta\circ dg^{-1}.\]
Then for every $h\in Sp(2g-2,\mathbb{R})$ we have
\[\theta(gh)={\rm Ad}(h^{-1})\circ \theta(g)\]
and hence this 
defines an $Sp(2g,\mathbb{R})$-invariant connection
on the $Sp(2g-2,\mathbb{R})$-principal bundle 
$Sp(2g,\mathbb{R})\to \Omega$. 
Denote by ${\cal Z}_0$ the horizontal bundle. 
It is invariant under the left action of $Sp(2g,\mathbb{R})$ and
the right action of $Sp(2g-2,\mathbb{R})$, and it 
is equipped with 
an invariant pseudo-Riemannian metric of type $(2g,2g-1)$, 

Now $\mathfrak{sp}(2,\mathbb{R})\subset \mathfrak{z}$,
and hence the tangent bundle for the right action of
$Sp(2,\mathbb{R})$ is contained in the 
horizontal bundle ${\cal Z}_0$. Thus the subbundle 
${\cal Y}_0$ of ${\cal Z}_0$ defined by the $B$-orthogonal 
complement $\mathfrak{y}$ in $\mathfrak{z}$ 
of the Lie algebra 
$\mathfrak{sp}(2,\mathbb{R})$ is invariant as well.  
The lemma follows.
\end{proof}

Since 
${\cal S}=Sp(2g,\mathbb{R})\times_{Sp(2g-2,\mathbb{R})}
\mathfrak{D}_{g-1}$ and since the subgroups $SL(2,\mathbb{R})$ and 
$Sp(2g-2,\mathbb{R})$ commute, the right action of 
$SL(2,\mathbb{R})$ on $Sp(2g,\mathbb{R})$ descend to an
action of $SL(2,\mathbb{R})$ on ${\cal S}$. 
The action of the unitary subgroup 
$U(1)$ of  $Sp(2,\mathbb{R})$ 
is just the standard
circle action on the fibres 
of the sphere bundle ${\cal S}\to \mathfrak{D}_{g}$ 
given by multiplication with complex
numbers of absolute value one. 
The connection ${\cal Z}_0={\cal T}\oplus {\cal Y}_0$ induces 
a real analytic splitting
\[T{\cal S}=T{\cal F}\oplus {\cal Z}=T{\cal F}\oplus {\cal T}\oplus {\cal Y}\]
where $T{\cal F}$ denotes the tangent bundles of the fibres
of the fibration ${\cal S}\to \Omega$, 
the \emph{horizontal bundle} ${\cal Z}$
is the image of ${\cal Z}_0\times T\mathfrak{D}_{g-1}$ under the 
projection 
$Sp(2g,\mathbb{R})\times \mathfrak{D}_{g-1}\to {\cal S}$ and
as before, ${\cal T}$ is the tangent bundle of the orbits of the
$SL(2,\mathbb{R})$-action.

We also note

\begin{lemma}\label{invariance4}
The right action of $SL(2,\mathbb{R})$ on ${\cal S}$ 
projects to the standard action of $SL(2,\mathbb{R})$ on $\Omega$.
\end{lemma}
\begin{proof}
This follows as before from naturality and bi-invariance of the Killing form. 
\end{proof}

The group $Sp(2g,\mathbb{Z})$ acts 
properly discontinuously from the left on the
bundle ${\cal S}\to \Omega$ as a group of real analytic
bundle automorphisms.
In particular, 
it preserves the real analytic splitting of the tangent bundle
of ${\cal S}$ 
into the tangent bundle of the leaves of the foliation ${\cal F}$ 
and the complementary bundle.
Thus this splitting descends to an $SL(2,\mathbb{R})$-invariant
real analytic
splitting of the tangent bundle of the quotient.
This quotient is just the sphere bundle  of the quotient vector 
bundle (in the orbifold sense) 
over the locally symmetric space 
 \[{\cal A}_g=Sp(2g,\mathbb{Z})\backslash 
Sp(2g,\mathbb{R})/U(g).\]

\end{appendix}

\bigskip\noindent
MATHEMATISCHES INSTITUT DER UNIVERSIT\"AT BONN\\
ENDENICHER ALLEE 60, 53115 BONN, GERMANY\\
e-mail: ursula@math.uni-bonn.de

\end{document}